\documentclass[11pt]{article}
\usepackage{cancel}
\usepackage{rotating}
\usepackage{url}
\usepackage[margin=2.7 cm,nohead]{geometry}
\usepackage{hyperref} 
\usepackage[ruled]{algorithm2e} 
\usepackage{epstopdf}
\usepackage{cite}
\usepackage{algorithmic}
\usepackage{color}
\usepackage{url}
\usepackage{cite}

\usepackage{stackrel,cancel}
\usepackage{amsmath,amsthm,amsfonts,amssymb,amscd, amsxtra, mathrsfs}
\numberwithin{equation}{section}
\usepackage{amsfonts,amssymb}
\usepackage{amsmath}
\usepackage{graphicx,epstopdf}
\usepackage{algorithm2e}
\usepackage[caption=false]{subfig} 
\usepackage{algorithmic, enumitem}

\newtheorem{theorem}{Theorem}[section]
\newtheorem{lemma}[theorem]{Lemma}
\newtheorem{definition}[theorem]{Definition}
\newtheorem{corollary}[theorem]{Corollary}
\newtheorem{proposition}[theorem]{Proposition}

\newtheorem{remark}[theorem]{Remark}
\newtheorem{example}[theorem]{Example}
\usepackage{hyperref}
\newtheorem{problem}[theorem]{Problem}

\begin{document}

\title{A boosted DC algorithm for non-differentiable DC components with non-monotone line search\thanks{Originally submitted at 2021-11-02, 14:49. Current version was submitted to the editors at 2022-06-13, 13:07. {The first author was supported in part by CNPq grant 304666/2021-1. The second author was supported in part by CAPES. The third author was supported in part by CNPq grant 424169/2018-5.}}}
\author{O. P. Ferreira \thanks{Instituto de Matem\'atica e Estat\'istica, Universidade Federal de Goi\'as,  CEP 74001-970 - Goi\^ania, GO, Brazil, E-mail:  {\tt  orizon@ufg.br}},
E. M. Santos \thanks{Instituto Federal de Educa\c{c}\~{a}o, Ci\^{e}ncia e Tecnologia do Maranh\~{a}o, A\c{c}ail\^{a}ndia, MA, Brazil E-mail: {\tt elianderson.santos@ifma.edu.br}.},
J. C. O. Souza \thanks{Department of Mathematics, Federal University of Piau\'{i}, Teresina, PI, Brazil, E-mail: {\tt joaocos.mat@ufpi.edu.br}
  }
}

\maketitle

\begin{abstract}
We introduce a new approach to apply the boosted difference of convex functions algorithm (BDCA) for solving non-convex and non-differentiable problems involving difference of two convex functions (DC functions). Supposing the first DC component differentiable and the second one possibly non-differentiable, the main idea of BDCA is to use the point computed by the DC algorithm (DCA) to define a descent direction and perform a monotone line search to improve the decreasing the objective function accelerating the convergence of the DCA. However, if the first DC component is non-differentiable, then the direction computed by BDCA can be an ascent direction and a monotone line search cannot be performed.
Our approach uses a non-monotone line search in the BDCA (nmBDCA) to enable a possible growth in the objective function values controlled by a parameter. Under suitable assumptions, we show that any cluster point of the sequence generated by the nmBDCA is a critical point of the problem under consideration and provide some iteration-complexity bounds. Furthermore, if the first DC component is differentiable, we present different iteration-complexity bounds and prove the full convergence of the sequence under the Kurdyka-\L{}ojasiewicz property of the objective function. Some numerical experiments show that the nmBDCA outperforms the DCA such as its monotone version.
\end{abstract}


\begin{AMS}
{90C26}, 65K05, 65K10, 47N10
\end{AMS}

\section{Introduction}

In this paper, we are interested in solving the following non-convex and non-differentiable DC optimization problem:
\begin{equation}\label{Pr:DCproblem}
\begin{array}{c}
\min \phi(x):=g(x)-h(x), \qquad \mbox{s.t. } ~x\in \mathbb{R}^{n}.
\end{array}
\end{equation}
where $g, h:\mathbb{R}^n \to \mathbb{R} $ are convex functions possibly non-differentiable. DC programming has been developed and studied in the last decades and successfully applied in different fields including but not limited to image processing \cite{LouZengOsherXin2015}, compressed sensing \cite{YinLouQi2015}, location problems \cite{AnNamYen2916, Brimberg1995, CruzNetoLopesSantosSouza2019}, sparse optimization problems \cite{GotohTakeda2018}, the minimum sum-of-squares clustering problem \cite{ARAGON2019, CuongYaoYen2020, OrdiBagirov2015}, the bilevel hierarchical clustering problem \cite{NamGeremewReynoldsTran2018}, clusterwise linear regression \cite{BagirovUgon2018}, the multicast network design problem \cite{GeremewNamSemenovBoginski2018} and multidimensional scaling problem \cite{LeTao2001, ARAGON2019, BeckHallak2020}. 

To the best of our knowledge, DCA was the first algorithm directly designed to solve \eqref{Pr:DCproblem}; see \cite{TaoLe1997,Pham1986}. Since then, several variants of DCA have arisen and its theoretical and practical properties have been investigated over the years, resulting in a wide literature on the subject; see \cite{OliveiraABC,DCAFirst2018}. Algorithmic aspects of DC programming have received significant attention lately such as subgradient-type (see \cite{BeckHallak2020, KhamaruWainwright2019}), proximal-subgradient \cite{CruzNetoEtAl2018, Moudafi2006, SOUZA3016, SunSampaio2003}, proximal bundle \cite{Welington2019}, double bundle \cite{KaisaBagirov2018}, codifferential \cite{BagirovUgon2011} and inertial method \cite{WelingtonTcheou2019}. Nowadays, DC programming plays an important role in high dimension non-convex programming and DCA (and its variants) is commonly used due to its simplicity, inexpensiveness and efficiency.

Recently, \cite{ARAGON2017, ARAGON2019} proposed a boosted DC algorithm (BDCA) for solving \eqref{Pr:DCproblem} which has the property of accelerating the convergence of DCA. In \cite{ARAGON2017}, the convergence of BDCA is proved if both functions $g$ and $h$ are differentiable, and the non-differentiable case is considered in \cite{ARAGON2019} but still supposing that $g$ is differentiable and $h$ is possibly non-differentiable. Although DCA is a decent method, the main idea of BDCA is to define a descent direction using the point computed by DCA and a line search to take a longer step than the DCA obtaining in this way a larger decreasing of the objective value per iteration. In addition to accelerating the convergence of DCA, we have been noticed that BDCA escapes from bad local solutions thanks to the line search, and hence, BDCA can be used to either accelerate DCA and provide better solutions. The differentiability of the DC component $g$ is needed to guarantee a descent direction from the point computed by DCA, otherwise such direction can be an ascent direction; see \cite[Remark~1]{ARAGON2017} and \cite[Example~3.4]{ARAGON2019}.

The aim of this paper is to provide a version of BDCA which can still be applied if both DC components are non-differentiable. The key idea is to use a non-monotone line search allowing some growth in the objective function values controlled by a parameter. We study the convergence analysis of the non-monotone BDCA (nmBDCA) as well as iteration-complexity bounds. Therefore, nmBDCA enlarges the applicability of BDCA to a broader class of non-smooth functions keeping up with its efficiency and simplicity. The concept of non-monotone line search, that we use here as a synonym for inexact line search, was firstly proposed by \cite{Grippo1986}, and later a new non-monotone search was considered by \cite{ZhangHager2004}. In \cite{SachsSachs2011}, an interesting general framework for non-monotone line research was proposed and more recently some variants have been considered; see for instance \cite{GrapigliaSachs2017, GrapigliaSachs2020}.

Furthermore, we prove the global convergence under the Kurdyka-\L{}ojasiewicz property as well as different iteration-complexity bounds for the case where the DC component $g$ is differentiable. We present some numerical experiments involving academic test functions to show the efficiency of the method comparing its performance with DCA.

The rest of this paper is organized as follows. In Section~\ref{sec2}, we recall some definitions and preliminary results used throughout the paper. The methods DCA, BDCA and nmBDCA are presented in Section~\ref{sec3}. In Section{ \ref{sec4}, we show the convergence analysis and iteration-complexity analysis of nmBDCA for a possible non-differentiable $g$. In Section~\ref{sec5} we established some iteration-complexity bounds for the sequence generated by nmBDCA under the assumption that the function $g$ is differentiable and a full convergence of sequence is also established under Kurdyka-\L{}ojasiewicz property. In Section~\ref{Sec6}, we present some numerical experiments to illustrate the performance of the method. The last section contains some remarks and future research directions are presented.

\section{Preliminaries} \label{sec2}
In this section we present some notations, definitions and results that will be used throughout the paper which can be found in \cite{Amir, clarke1983optimization, Lemarechal1993}.
\begin{definition}[{\cite[Definition 1.1.1, p. 144, and Proposition 1.1.2, p. 145]{Lemarechal1993}}] 	
	A function $f\colon\mathbb{R}^{n}\to \mathbb{R} $ is said to be convex if $f(\lambda x + (1-\lambda)y)\leq \lambda f(x) + (1-\lambda) f(y)$, for all $x,y\in {\mathbb{R}^{n}}$ and $\lambda \in\, ]0,1[$.
	Moreover, $f$ is said to be strongly convex with modulus $\sigma > 0 $ if $f(\lambda x + (1-\lambda)y)\leq \lambda f(x) + (1-\lambda) f(y)-\frac{\sigma}{2}\lambda(1-\lambda)\|x-y\|^2$, for all $x,y\in {\mathbb{R}^{n}}$ and $\lambda \in \, ]0,1[$.
\end{definition}
If in the above definition $f$ is strongly convex with modulus $\sigma=0$, then $f$ is convex. 
\begin{definition}[{\cite[p. 25]{clarke1983optimization}}]
	We say that $f\colon\mathbb{R}^{n}\to\mathbb{R}$ is locally Lipschitz if, for all $x\in \mathbb{R}^{n}$, there exist a constant $K_{x}>0$ and a neighborhood $U_{x}$ of $x$ such that $|f(u)-f(y)|\leq K_{x}\|u-y\|$, for all $ u,y\in U_{x}.$
\end{definition}
If $f\colon\mathbb{R}^{n}\to \mathbb{R}$ is convex, then $f$ is locally Lipschitz; see \cite[p. 34]{clarke1983optimization}.
\begin{definition}[{\cite[p. 27]{clarke1983optimization}}] 
	Let $f\colon\mathbb{R}^{n}\to\mathbb{R}$ be a locally Lipschitz function.
	The Clarke's subdifferential of $f$ at $x\in \mathbb{R}^{n}$ is given by $\partial _{c} f(x)=\{v \in \mathbb{R}^{n}\:|\: f^{\circ }(x;d)\geq \langle v,d \rangle, ~
\forall d \in \mathbb{R}^{n} \},$
where $f^{\circ }(x;d)$ is the generalized directional derivative of $f$ at $x$ in the direction $d$ given by
$ f^{\circ }(x;d)= \limsup _{ u\rightarrow x, t\downarrow 0} (f(u+td)-f(u))/t.$ 
\end{definition}
If $f$ is convex, then $\partial _{c}f(x)$ coincides with the subdifferential $\partial f(x)$ in the sense of convex analysis and $f^{\circ }(x;d)$ coincides with the usual directional derivative $f'(x;d)$; see {\cite[p. 36]{clarke1983optimization}}. 
\begin{theorem}[{\cite[Proposition 2.1.2, p. 27]{clarke1983optimization}}] 
	Let $f\colon\mathbb{R}^{n}\to\mathbb{R}$ be a locally Lipschitz function. Then, for all $x\in \mathbb{R}^{n}$, $\partial _{c}f(x)$ is a non-empty, convex, compact subset of $\mathbb{R}^{n}$ and $\|v\|\leq K_{x},$ for all $v\in \partial _{c}f(x)$, where $K_{x}>0$ is the Lipschitz constant of $f$ around $x$.
\end{theorem}
\begin{proposition}
	\label{cont_subdif}
	Let $f\colon\mathbb{R}^{n}\rightarrow \mathbb{R}$ be convex and $(x^{k})_{k\in\mathbb{N}}$ such that ~$\lim _{k\rightarrow\infty}x^{k}=x^{*}$.
	If $(y^{k})_{k\in\mathbb{N}}$ is a sequence such that $y^{k}\in \partial f(x^{k})$ for every $k\in \mathbb{N}$, then $(y^{k})_{k\in\mathbb{N}}$ is bounded and its cluster points belong to $\partial f(x^{*}).$
\end{proposition}
\begin{theorem}[{\cite[Theorem 5.25, p. 122 and Corollary 3.68, p. 76]{Amir}}] 
Let $f\colon\mathbb{R}^{n}\to \mathbb{R} $ be a strongly convex function. Then, $f$ has a unique minimizer $x^{*}\in \mathbb{R}^{n}$, and $0\in \partial f(x^{*})$.
\end{theorem}
\begin{lemma}[{\cite[Lemma 5.20, p. 119]{Amir}}] 
Let $f\colon\mathbb{R}^{n}\to \mathbb{R} $ be a strongly convex function with modulus $\sigma>0$, and
let $\bar{f}:\mathbb{R}^{n}\to \mathbb{R}$ be convex. Then $f + \bar{f}$ is strongly convex function with modulus $\sigma>0$.
\end{lemma}
\begin{theorem}[{\cite[Theorem 5.24, p. 119]{Amir}}] \label{teo2}
	Let $f\colon\mathbb{R}^{n}\to \mathbb{R}$ be a convex function. Then, for a given $\sigma> 0$, the following statements are equivalent:
	\begin{enumerate}[label={({\roman*})}, ref={(\roman*)}]
		\item
			$f$ is a strongly convex function with modulus $\sigma> 0$.
		\item \label{it:teo2:fct-geq}
			$f(y)\geq f(x) + \langle v, y-x \rangle + \frac{\sigma}{2} \| y-x\|^{2}$, for all $x,y\in \mathbb{R} ^{n} $ and all $v\in \partial f(x)$.
		\item \label{it:teo2:inner-geq}
			$\langle w-v,x-y \rangle \geq \sigma \| y-x\|^{2}$, for all $x,y\in \mathbb{R} ^{n}$, all $w\in \partial f(x)$ and all $v\in \partial f(y).$
	\end{enumerate}
\end{theorem}
\begin{definition}[{\cite[p. 107]{Amir}}]\label{def.lipschtz}
A differentiable function $f\colon\mathbb{R}^n \to \mathbb{R} $ has gradient Lipschitz continuous with constant $L >0$ whenever $\|\nabla f(x) -\nabla f(y)\|\leq L\|x-y\|$, for all $x,y \in \mathbb{R}^{n}$.
\end{definition}
\begin{lemma}[Descent lemma {\cite[Lemma 5.7, p. 109]{Amir}}]\label{eq:IneqLip}
Assume that $f$ satisfies Definition~\ref{def.lipschtz}. Then, for all $x, d\in \mathbb{R}^n$ and all $\lambda \in \mathbb{R}$, there holds $f\left(x+ \lambda d\right) \leq f(x) +\lambda \left\langle \nabla f(x), d \right\rangle + L\lambda^2 \|d\|^2/2$.
\end{lemma} 
\begin{theorem}[{\cite[p. 38-39]{clarke1983optimization}}]\label{subdif_DC}
Let $g:\mathbb{R}^n \to \mathbb{R} $ be a locally Lipschitz and differentiable function and $h:\mathbb{R}^n \to \mathbb{R}$ be a convex function. Then, for every $x\in \mathbb{R}^{n},$ we have $\partial _{c}(g-h)(x) =\{\nabla g(x)\}-\partial h(x)$

\end{theorem}
\section{Non-monotone BDCA} \label{sec3}
In this section we introduce the non-monotone boosted DC algorithm (nmBDCA) to solve \eqref{Pr:DCproblem}. Throughout the paper we need the following two assumptions:
\begin{enumerate}[label={(\textbf{H\arabic*})}, ref={(H\arabic*)}]
	\item\label{it:ghstronglyconvex}
		$g,h:\mathbb{R}^{n}\to \mathbb{R}$ are both strongly convex functions with modulus $\sigma>0$;
	\item\label{it:phistarinf}
		$\phi ^{*}:=\inf _{x\in \mathbb{R}^n} \{ \phi (x)=g(x)-h(x) \}>-\infty .$
\end{enumerate}
Before proceeding with our study let us first discuss the assumptions \ref{it:ghstronglyconvex} and \ref{it:phistarinf} in next remark.
\begin{remark}\label{infinitydecomp}
 We first note that \ref{it:ghstronglyconvex} is not restrictive. Indeed, given two convex functions $g$ and $h$ we can add to both a strongly convex term $({\sigma}/{2})\| x \|^{2}$ to obtain $\phi(x):=(g(x)+{\sigma}/{2}\| x \|^{2}) - (h(x)+{\sigma}/{2}\| x \|^{2})$. Hence, $\phi$ is rewritten as a difference of two strongly convex functions with modulus $\sigma>0$. Assumption \ref{it:phistarinf} is usual in the context of DC programming, see e.g. \cite{ARAGON2017,ARAGON2019,CruzNetoEtAl2018}.
\end{remark}

Let us present the conceptual statement of BDCA, but first we recall the DCA given in Algorithm~\ref{Alg:DCA}.

\begin{algorithm}[h]
	\caption{The DC Algorithm (DCA)}\label{Alg:DCA}
	\begin{algorithmic}[1]
	\STATE { Choose an initial point $x^0\in \mathbb{R}^{n}$. Set $k=0$.}
	\STATE{Choose $w^{k}\in\partial h(x^{k})$ and compute $y^{k}$ the solution of the following problem
	\begin{equation} \label{eq:DCAS}
	\min _{x\in \mathbb{R}^{n}} g(x)-\left \langle w^{k},x-x^{k} \right\rangle.
	\end{equation}}
	\STATE{If $y^{k}=x^{k}$, then STOP and return $x^{k}$. Otherwise, set $x^{k+1}:=y^{k}$, $k \leftarrow k+1$ and go to Step~2.}
	\end{algorithmic}
\end{algorithm}

Note that, in the DCA, if $d^k\neq 0$, then the next iterate is $x^{k+1} = y^k$.
In this case, under the assumption~\ref{it:ghstronglyconvex}, it can be proved that $\phi (y^{k})< \phi (x^{k})-\sigma \|d^k\| ^{2}$, see for example \cite[Proposition~3.1]{ARAGON2019}. The conceptual monotone BDCA is given in Algorithm~\ref{Alg:BDCAM}.

\begin{algorithm}[h]
	\caption{The monotone boosted DC Algorithm (BDCA)}\label{Alg:BDCAM}
	\begin{algorithmic}[1]
	\STATE {Fix $\lambda _{-1}>0$, $\rho>0$ and $\zeta \in (0,1)$. Choose an initial point $x^0\in \mathbb{R}^{n}$. Set $k=0$.}
	\STATE{Select $w^{k}\in\partial h(x^{k})$ and compute $y^{k}$ the solution of the following problem
	\begin{equation} \label{eq:BDCAMu}
	\min _{x\in \mathbb{R}^{n}} g(x)-\left \langle w^{k},x-x^{k} \right\rangle.
	\end{equation}}
	\STATE{Set $ d^{k}:=y^{k}-x^{k}$. If $d^{k} =0$, then STOP and return the point $x^{k}$. Otherwise, set $\lambda _{k}:= \zeta^{j_k}\lambda_{k-1}$, where
	\begin{equation} \label{eq:BDCAjku}
	 j_k:=\min \big\{j\in {\mathbb N}: ~\phi( y^{k}+\zeta^{j} \lambda _{k-1}d^{k})\leq \phi (y^{k})-\rho \left(\zeta^{j}{\lambda _{k-1}}\right)^{2}\| d^{k}\| ^{2}\big\}.
	\end{equation}}
	\STATE{Set $x^{k+1}:=y^{k}+\lambda _{k}d^{k}$; set $k \leftarrow k+1$ and go to Step~2.}
	\end{algorithmic}
\end{algorithm}

It is known that a necessary condition for a point $x$ to be local minimizer of $\phi$ is that $\partial h(x)\subseteq \partial g(x)$; see \cite{Toland}. In this case, $x$ is called inf-stationary. Such a condition is not easy to verify and hence a relaxed form of inf-stationary has been considered in DC literature.

\begin{definition}
 A point $x^{*}\in \mathbb{R}^{n}$ is critical of $\phi(x)=g(x)-h(x)$ as in \eqref{Pr:DCproblem} if $$\partial g(x^{*}) \cap \partial h(x^{*}) \neq~\emptyset.$$
\end{definition}

From subdifferential calculus, we have that $\partial_c \phi(x)\subseteq \partial g(x)-\partial h(x)$; see \cite{clarke1983optimization}. Thus, criticality is a weaker condition than Clarke stationary, i.e., $0\in \partial _{c} \phi(x)$. Reference \cite{KaisaBagirov2018} pointed out some interesting relationships between inf-stationary, Clarke stationary and critical points. Namely, inf-stationarity implies Clarke stationarity and Clarke stationary implies critical point. However, the converse of these implications do not hold without some extra assumptions. In other words, it is possible that a critical point is neither a local optimum nor a saddle point of the objective $\phi$. Therefore, the quality of the solution found by an algorithm is something that must to be discussed. 

\subsection{The algorithm}
Next, we formally introduce our non-monotone version of BDCA to solve~\eqref{Pr:DCproblem}.

\begin{algorithm}[h]
\caption{Non-monotone boosted DC Algorithm (nmBDCA)} \label{Alg:ASSPM}
\begin{algorithmic}[1]
\STATE {Fix $\lambda _{-1}>0$, $\rho>0$ and $\zeta \in (0,1)$. Choose an initial point $x^0\in \mathbb{R}^{n}$. Set $k=0$.}
\STATE{Select $w^{k}\in\partial h(x^{k})$ and compute $y^{k}$ the solution of the following problem
\begin{equation} \label{eq:ASSPMu}
\min _{x\in \mathbb{R}^{n}} \psi _{k}(x):= g(x)-\left \langle w^{k},x-x^{k} \right\rangle.
\end{equation}}
\STATE{Set $ d^{k}:=y^{k}-x^{k}$. If $d^{k} =0$, then STOP and return $x^{k}$. Otherwise, take $\nu _{k}\in \mathbb{R}_{+}$ (to be specified later) and set $\lambda _{k}:= \zeta^{j_k}\lambda_{k-1}$, where
\begin{equation} \label{eq:jku}
 j_k:=\min \big\{j\in {\mathbb N}: ~\phi( y^{k}+\zeta^{j} \lambda _{k-1}d^{k})\leq \phi (y^{k})-\rho \left(\zeta^{j}{\lambda _{k-1}}\right)^{2}\| d^{k}\| ^{2} + \nu _{k}\big\}.
\end{equation}}
\STATE{Set $x^{k+1}:=y^{k}+\lambda _{k}d^{k}$; set $k \leftarrow k+1$ and go to Step~2.}
\end{algorithmic}
\end{algorithm}

From now on, we denote by $(x^k)_{k\in\mathbb{N}}$ the sequence generated by Algorithm~\ref{Alg:ASSPM} which is a non-monotone version of BDCA for solving \eqref{Pr:DCproblem} with both DC components possibly non-differentiable. It is worth to mention that we will study the convergence analysis of Algorithm~\ref{Alg:ASSPM} for both cases where $g$ is differentiable and non-differentiable. If $g$ is non-differentiable, we will suppose that $\nu_k>0$, for all $k\in\mathbb{N}$, and hence, we will extend the results of BDCA~\cite{ARAGON2017,ARAGON2019} to the non-differentiable setting. If $g$ is differentiable, we will suppose that $\nu_k\geq 0$, for all $k\in\mathbb{N}$. In this case, if $\nu _{k}= 0$, for all $k\in\mathbb{N}$, then the non-monotone line search \eqref{eq:jku} coincides with the monotone line search \eqref{eq:BDCAjku}. Otherwise, if $\nu _{k}> 0$, for all $k\in\mathbb{N}$, then nmBDCA can be viewed as an inexact version of BDCA.

Before, recall that \eqref{eq:ASSPMu} always has a unique solution $y^{k}$, which is characterized by 
\begin{equation}\label{eq:charyk}
w^{k}\in \partial g (y^{k}), \qquad \forall k\in \mathbb{N}.
\end{equation} 

It is remarked in \cite[Remark 1]{ARAGON2017} (differentiable case) and \cite[Example 3.4]{ARAGON2019} (non-differentiable case) that the differentiability of the DC component $g$ is necessary to apply the boosted technique proposed by the authors. The next example shows that the search direction $d^k\neq 0$ used by Algorithm~\ref{Alg:ASSPM} can be an ascent direction at the point $y^k$. Consequently, the line search in usual BDCA proposed in \cite{ARAGON2017,ARAGON2019} cannot be performed. However, the non-monotone line search proposed in Algorithm~\ref{Alg:ASSPM} overcome this drawback as we will illustrate in the sequel. 
\begin{example}{\cite[Example 3.4]{ARAGON2019}}\label{ExampleAlgorithm}
Consider the problem \eqref{Pr:DCproblem}, where the functions $g$ and $h$ are given, respectively, by 
\begin{equation*}
	g(x_1,x_2)=-\frac{5}{2}x_1+x_1^2+x_2^2+|x_1|+|x_2|, \qquad h(x_1,x_2)=\frac{1}{2}(x_1^2+x_2^2).
\end{equation*}
The function $\phi(x_1,x_2) = g(x_1,x_2)-h(x_1,x_2)=\frac{1}{2}(x_1^2+x_2^2)+|x_1|+|x_2|-\frac{5}{2}x_1$ has only one critical point (global minimum) at $x^*=(1.5, 0)$. Clearly, $g$ is a non-differentiable function. Some calculations show that, letting $x^0=(\frac{1}{2},1)$, we have that $w^0=(\frac{1}{2},1)$, $y^0=(1,0)$ is the solution of \eqref{eq:ASSPMu} and $d^0=(\frac{1}{2}, -1)$. We can check that the directional derivative of $\phi$ at $y^0$ in the direction of $d^0$ is $\phi'(y^0, d^0)=\frac{3}{4}$. Thus, $d^0$ is not a descent direction for $\phi$ at $y^0$. Indeed, due to $\phi(y^0)=-1$ and $\phi(y^0+\lambda d^0) =-1+\frac{3}{4}\lambda+\frac{5}{8}\lambda^2$, we conclude that $\phi(y^0+\lambda d^0)>\phi(y^0)$, for all $\lambda>0$. Hence, a usual monotone line search cannot be performed. On the other hand, 
\begin{equation*}
	\phi(y^0+\lambda d^0)-\phi(y^0)+\rho \lambda ^{2}\| d^{0}\| ^{2}= \frac{3}{4}\lambda+\frac{5}{8}\lambda^2 + \frac{5}{4} \rho \lambda ^{2}, 
\end{equation*}
and $ \lim_{\lambda \to 0^+}\left(\frac{3}{4}\lambda+\frac{5}{8}\lambda^2 + \frac{5}{4} \rho \lambda ^{2}\right)=0$. Thus, for $\nu_0>0$ there exists $\delta_0>0$ such that $\phi(y^0+\lambda d^0)-\phi(y^0)+\rho \lambda ^{2}\| d^{0}\| ^{2}<\nu_0$, for all $\lambda\in (0, \delta_0)$. Therefore, the non-monotone line search \eqref{eq:jku} can be performed; see Figure~\ref{fig1a}. Using $\lambda _{-1}=1$, $\rho=0.1$ and $\zeta = 0.5$ one has that although $f(x^{k+1})$ does not decrease the corresponding iteration in DCA (namely, $f(y^k)$), for all $k$ (see Figure~\ref{fig1b}), Algorithm~\ref{Alg:ASSPM} still has a better performance than DCA as we can see in Figure~\ref{fig1c}. Both algorithms return the global minimum $x^*=(1.5, 0)$ with Algorithm~\ref{Alg:ASSPM} requiring 6 iterates while DCA computing 17 iterates until the stopping rule is satisfied. We will return to this example with more details in Section~\ref{Sec6}.
\begin{figure*}[tbp]
\centering
\subfloat[Iterations]{\label{fig1a}\includegraphics[width=0.32\linewidth]{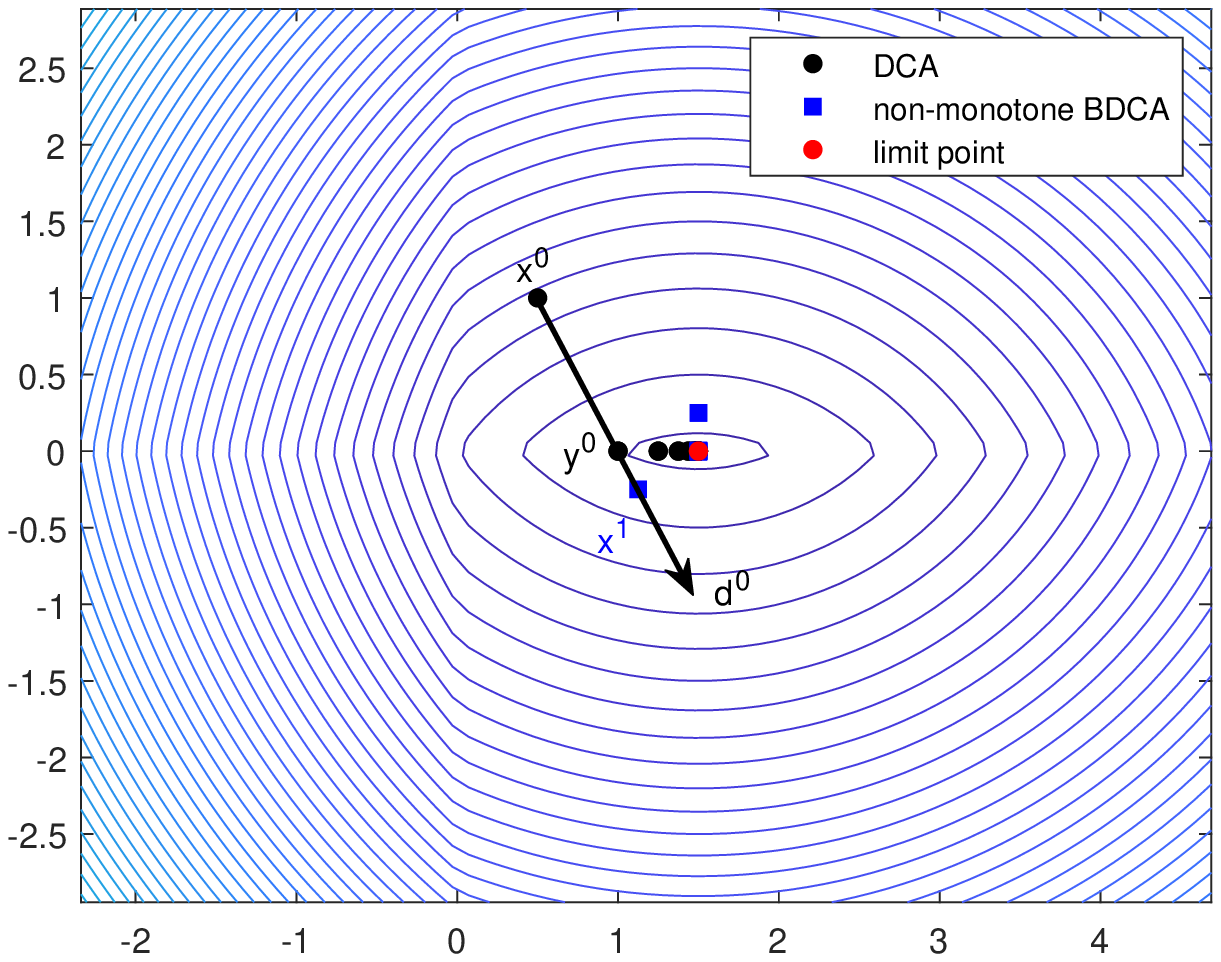}}
\subfloat[Values of $\phi(x^k)$]{\label{fig1b}\includegraphics[width=0.32\linewidth]{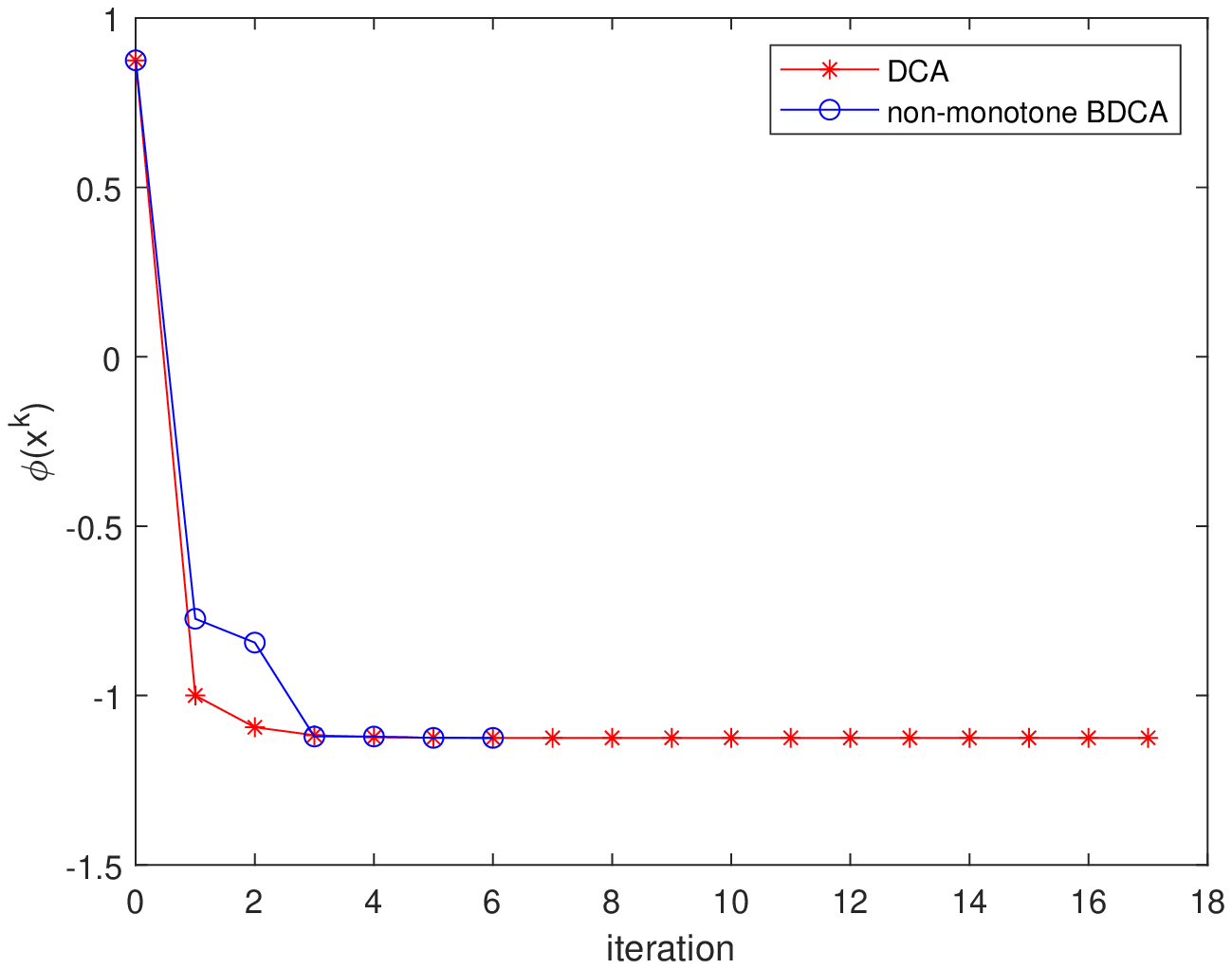}}
\subfloat[Values of $\lVert d^k\rVert$]{\label{fig1c}\includegraphics[width=0.32\linewidth]{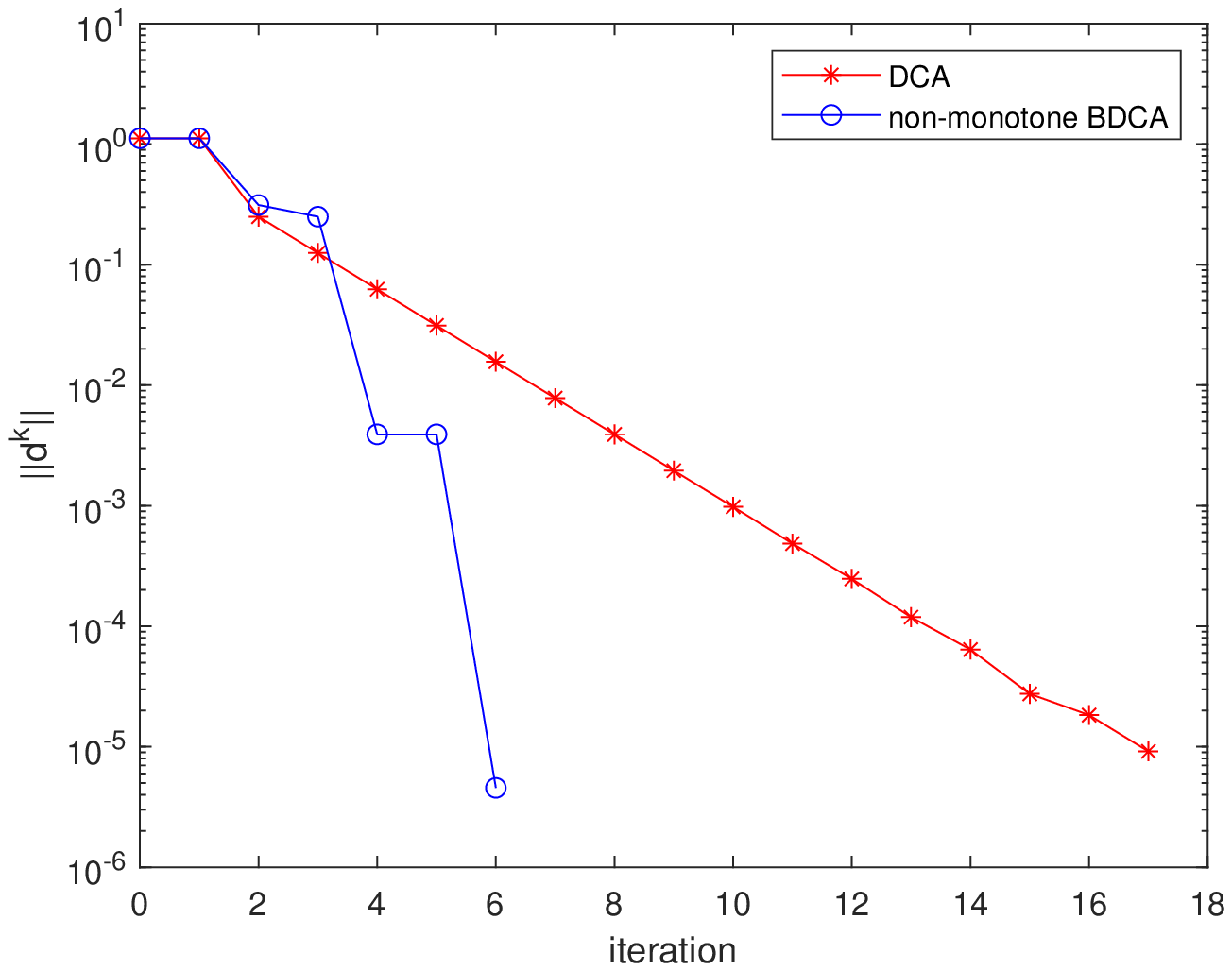}}
\caption{Illustration of Example~\ref{ExampleAlgorithm} for both DCA (Algorithm~\ref{Alg:DCA}) and nmBDCA (Algorithm~\ref{Alg:ASSPM}).}
\end{figure*}
\end{example}
In the previous example, we illustrated how a non-monotone line search \eqref{eq:jku} can be performed in BDCA. In fact, in next section, we will show that in general the non-monotone line search \eqref{eq:jku} can be performed. We end this section with a basic result in the study of DCA, see for example \cite[Proposition 2]{TaoLe1997}. In particular, it shows that the solution of Problem~\eqref{eq:ASSPMu}, which coincides with the solution of Problems~\eqref{eq:DCAS} and \eqref{eq:BDCAMu} in Algorithms~\ref{Alg:DCA} and \ref{Alg:BDCAM} respectively, provides a decrease in the value of the objective function $\phi$. For the sake of completeness, we include its proof here.
\begin{proposition}\label{pr:ffr}
	For each $k\in\mathbb{N}$, the following statements hold:
	\begin{enumerate}[label={({\roman*})}, ref={(\roman*)}]
		\item\label{it:ffr:critical}
			If $d^{k}=0$, then $x^{k}$ is a critical point of $\phi$;
		\item\label{it:ffr:ineq}
			There holds $\phi (y^{k})\leq \phi (x^{k})-\sigma \| d^{k}\| ^{2}$.
	\end{enumerate}
\end{proposition}
\begin{proof} 
Before starting the proof, we remind that $d^{k}=y^{k}-x^{k}$. To prove item~\ref{it:ffr:critical} recall that due to $y^{k}$ being the solution of \eqref{eq:ASSPMu}, it satisfies \eqref{eq:charyk}. Thus, if $d^{k}=0$, then $y^{k}=x^{k}$ and $w^k\in \partial g (x^{k})\cap \partial h (x^{k})\not=\emptyset$. Consequently, $x^k$ is a critical point of $\phi$. To prove item~\ref{it:ffr:ineq}, take $w^{k}\in\partial h(x^{k})$.
Now, we use the strong convexity of $g$ with modulus $\sigma>0$ and Theorem~\ref{teo2}\,\ref{it:teo2:fct-geq} to conclude that 
\begin{equation*}
	g(x^k)\geq g(y^k) - \langle v, d^k \rangle + \frac{\sigma}{2} \| d^k\|^{2}, \qquad \forall v\in \partial g(y^k).
\end{equation*}
Thus, due to $y^{k}$ being the solution of Problem~\eqref{eq:ASSPMu}, \eqref{eq:charyk} and the last inequality, we have
\begin{equation} \label{eq:ffrs}
g(x^k)\geq g(y^k) - \langle w^k, d^k \rangle + \frac{\sigma}{2} \|d^k\|^{2}.
\end{equation}
We also know that $h$ satisfies \ref{it:ghstronglyconvex}, i.e., $h$ is strong convex with modulus $\sigma>0$. Thus, since $w^{k}\in\partial h(x^{k})$, it follows from Theorem~\ref{teo2}\,\ref{it:teo2:fct-geq} that
\begin{equation} \label{eq:ffrsa}
h(y^k)\geq h(x^k) + \langle w^k, d^k \rangle + \frac{\sigma}{2} \| d^k\|^{2}.
\end{equation}
Thus, adding \eqref{eq:ffrs} and \eqref{eq:ffrsa} we have that $g(x^k)+h(y^k)\geq g(y^k)+h(x^k) + {\sigma}\| d^k\|^{2}$. Hence, using the definition of $\phi$ in \eqref{Pr:DCproblem} we conclude that $
\phi(x^k)\geq \phi(y^k) + {\sigma}\|d^k\|^{2}$, which is equivalente to the desired inequality finishing the proof of item~\ref{it:ffr:ineq}. 
\end{proof}
\subsection{Strategies to choose $\nu_k$}

 Next, we discuss some strategies to choose the sequence of parameters $(\nu_{k})_{k\in \mathbb{N}}$. We emphasize that throughout the paper each one of the following strategies will be used separately and only when explicitly stated:
\begin{enumerate}[label={(\textbf{S\arabic*})}, ref={(S\arabic*)}]
	\item \label{it:deltamin}
	Given $\delta _{min}\in [0,1)$, the sequence $(\nu _{k})_{k\in \mathbb{N}}\subset\mathbb{R}_{+}$ is defined as follows: $\nu _{0}\geq 0$ and $\nu_{k+1}$, for each $ \delta _{k+1} \in [\delta _{min},1]$, satisfies the following condition
\begin{equation}\label{nmr1}
			0 \leq \nu_{k+1} \leq (1-\delta _{k+1})(\phi (x^{k})- \phi (x^{k+1})+\nu _{k}), \quad \forall k\in \mathbb{N}.
		\end{equation}
	\item \label{it:nusummable}
		$(\nu _{k})_{k\in \mathbb{N}}\subset\mathbb{R}_{+}$ is such that $\sum _{k=0}^{+\infty}\nu _{k} <+\infty$;
	\item \label{it:nubound}
		$(\nu _{k})_{k\in \mathbb{N}}\subset\mathbb{R}_{+}$ is such that for every $\delta>0$, there exists $k_{0}\in \mathbb{N}$ such that $\nu _{k}\leq \delta \|d^{k}\|^{2},$ for all $k\geq k_{0}$.
\end{enumerate}

\begin{remark} \label{re:pip}
First note that, by using Proposition~\ref{prop15u}\,\ref{it:prop15u:decrease}, we have $ 0\leq \sigma \|d^{k}\|^{2} \leq \phi (x^{k})-\phi (x^{k+1})+\nu _{k}$, for all $k\in \mathbb{N}$. Thus, we can take $\nu_{k+1}\geq 0$ satisfying \eqref{nmr1}. Furthermore, if $(\nu _{k})_{k\in \mathbb{N}}$ satisfies \ref{it:deltamin} with $\delta_{min}>0$, then $(\nu _{k})_{k\in \mathbb{N}}$ also satisfies \ref{it:nusummable}. Indeed, it follows from \eqref{nmr1} that 
\begin{equation} \label{eq:ccsa}
0\leq \delta _{k+1}(\phi (x^{k})-\phi (x^{k+1})+\nu _{k}) \leq (\phi (x^{k})+\nu _{k})- (\phi (x^{k+1})+\nu _{k+1}).
\end{equation} 
Since $\delta _{k+1}\geq\delta _{min}>0$, $\phi (x^{k})-\phi (x^{k+1})+\nu _{k}\geq 0$, for all $k\in \mathbb{N}$, and $\phi$ satisfies \ref{it:phistarinf} we obtain $\delta _{min} \sum _{k=0}^{N} (\phi (x^{k})-\phi (x^{k+1})+\nu _{k}) \leq \phi (x^{0})-\phi ^{*} + \nu _{0}<\infty$. Hence, due to $\nu _{k+1}\leq (1-\delta _{min})(\phi (x^{k})-\phi (x^{k+1})+\nu _{k})$ for all $k\in \mathbb{N}$, we have
$\sum _{k=0}^{+\infty}\nu _{k} < \nu_0+ ((1-\delta _{min})/\delta _{min})(\phi (x^{0})-\phi ^{*} + \nu _{0})<\infty$. Therefore, $(\nu _{k})_{k\in \mathbb{N}}$ satisfies \ref{it:nusummable} and the claim is proved.
\end{remark}

Although strategy \ref{it:deltamin} seems to be theoretical, we will see in the sequel a practical and efficient example satisfying this condition. A sequence $(\nu _{k})_{k\in \mathbb{N}}$ satisfying \ref{it:nusummable} is simple and exogenous, i.e., it can be taken {\it a priori}. Finally, at a first glance, strategy \ref{it:nubound} seems to be a strong condition, but actually there are simple examples of sequences $(\nu _{k})_{k\in \mathbb{N}}$ satisfying \ref{it:nubound} which are easy to implement numerically. 

Alternatively, we can consider the following strategy:
\begin{enumerate}[label={(\textbf{S3'})}, ref={(S3')}]
\item \label{eq:nub2} Fix $\bar{\delta} \in (0, \sigma)$. There exists $k_{0}\in \mathbb{N}$ such that $\nu _{k}\leq \bar{\delta} \|d^{k}\|^{2},$ for all $k\geq k_{0}$.
\end{enumerate}
Since it changes according to $d^k$, in our point of view, it makes this strategy very interesting.

Next, we present some examples of sequences $(\nu_{k})_{k\in \mathbb{N}}$ according to the above strategies. 
\begin{example}
 Let us first recall the definition of the sequence of ``cost updates" $(C_k)_{k\in\mathbb{N}}$ that characterizes the non-monotone line search proposed in \cite{ZhangHager2004}. Consider $0\leq \eta_{min}\leq \eta_{max}<1$, $C_0 >\phi (x^0)$ and $Q_0 = 1$. Choose $\eta_k\in [\eta_{min}, \eta_{max}]$ and set 
\begin{equation} \label{eq:zhs}
Q_{k+1}:=\eta_kQ_{k}+1, \qquad C_{k+1} := ({\eta_k}Q_kC_k + \phi(x^{k+1}))/Q_{k+1}, \qquad \forall k \in \mathbb{N}.
\end{equation}
Note that, after some algebraic manipulations, we can show that \eqref{eq:zhs} is equivalent to 
$C_{k+1} = (1-1/Q_{k+1})C_{k}+\phi(x^{k+1})/Q_{k+1}$, for all $k \in \mathbb{N}$.
Thus, setting $\nu_{k}=C_k-\phi(x^k)$ and $\delta_{k+1}=1/Q_{k+1}$, we conclude that $\nu _{k+1} =(1-\delta _{k+1})(\phi (x^{k})-\phi (x^{k+1})+\nu _{k})$, for all $k \in \mathbb{N}$. Moreover, from \eqref{eq:zhs} we have $Q_{k+1}>1,$ for all $k\in \mathbb{N}$, and then $(1-\delta _{k+1})>0$ for all $k\in \mathbb{N}$. Since $\nu_0=C_0 -\phi (x^0)>0$, induction argument combined with Proposition~\ref{prop15u}\,\ref{it:prop15u:decrease} imply that $\nu _{k+1}>0$ for all $k\in \mathbb{N}$. Moreover, $(\nu _{k})_{k\in \mathbb{N}}$ satisfies \ref{it:deltamin}. It is worth noting that the non-monotone line search technique proposed in \cite{ZhangHager2004} outperforms the one proposed in \cite{Grippo1986} in many problems; see \cite[Section 4]{ZhangHager2004}.
\end{example}
\begin{example}
 Take any $\nu _{0}>0$, and define $\delta _{k+1}$ and $\nu _{k}$ as follows
\begin{equation}\label{eq:ndgeo}
0< \delta _{min}\leq \delta _{k+1}<1, \qquad 0< \nu _{k+1}:= (1-\delta _{k+1})(\sigma +\rho \lambda _{k}^{2})\|d^{k}\|^{2}, \qquad \forall k\in \mathbb{N}.
\end{equation}
Then Proposition~\ref{prop15u}\,\ref{it:prop15u:decrease} yields $(\sigma+\rho \lambda _{k}^{2})\|d^{k}\|^{2}\leq \phi (x^{k})-\phi (x^{k+1})+\nu _{k}.$ Thus, whenever $d^k\neq 0$, we have $0<\nu _{k+1} \leq (1-\delta _{k+1})(\phi (x^{k})-\phi (x^{k+1})+\nu_{k}).$ Therefore, $(\nu _{k})_{k\in \mathbb{N}}$ defined in \eqref{eq:ndgeo} satisfies \ref{it:deltamin}.
\end{example}
\begin{example} \label{eq:omega} 
Let $\omega>0$ be a constant. Then, the sequence $(\nu _{k})_{k\in \mathbb{N}}\subset\mathbb{R}_{++}$ defined by $ \nu _{k}:=\omega\|d^{k}\|^{2}/(k+1)$, for all $ k\in \mathbb{N}$, satisfies \ref{it:nubound}. Indeed, due to $\lim _{k\to\infty} \omega /(k+1)=0$, for every $\delta>0$, there exists $k_{0}\in \mathbb{N}$ such that $k\geq k_{0}$ implies that $ \omega/(k+1)\leq \delta $. Thus, we have that $ \nu _{k}\leq \delta \|d^{k}\|^{2}.$ Similarly, we can show that $(\nu _{k})_{k\in \mathbb{N}}\subset\mathbb{R}_{++}$ defined by $ \nu _{k}:=\omega\|d^{k}\|^{2}/\ln(k+2)$, for all $ k\in \mathbb{N}$, also satisfies \ref{it:nubound}. 
\end{example}
\begin{example} 
Take an integer $M>0$, set $m_{0}=0$ and for $k>0$ take $0\leq m_{k}\leq \min \{m_{k-1}+1,M\}$. Setting $\phi (x^{\ell(k)}):=\max _{0\leq j\leq m_{k}}\phi (x^{k-j})$ and 
\begin{equation}\label{eq:grippo}
\nu _{k}:=\phi (x^{\ell(k)})-\phi (x^{k}),\qquad 0=\delta_{min}\leq \delta_{k+1}\leq \frac{ \phi(x^{\ell(k)})- \phi(x^{\ell(k+1)})}{\phi(x^{\ell(k)})-\phi(x^{k+1})},
\end{equation}
the definitions of $\nu _{k}$ and $\delta _{k+1}$ in \eqref{eq:grippo} satisfies \ref{it:deltamin} with $\delta _{min}=0$.
In fact, from the definition of $\phi (x^{\ell(k)})$ it follows that $\nu _{0}=0$ and $\phi (x^{k})\leq \phi (x^{\ell (k)}),$ for all $k\in\mathbb{N}$, which ensures that $\nu _{k}\geq 0.$
From Proposition~\ref{prop15u}\,\ref{it:prop15u:decrease} and definition of $\nu _{k}$ in \eqref{eq:grippo} it follows that $\phi (x^{k+1}) < \phi (x^{\ell (k)}).$ Since $m_{k+1}\leq m_{k}+ 1$, we conclude that $\phi (x^{\ell(k+1)})\leq \phi (x^{\ell(k)}).$
Thus, we have $\phi(x^{\ell(k)})- \phi(x^{\ell(k+1)}) \leq \phi(x^{\ell(k)})-\phi(x^{k+1}),$ which shows that $\delta _{k+1}\in [0,1].$ By using the definitions of $\nu _{k}$ and $\delta _{k+1}$ in \eqref{eq:grippo}, we have 
 \begin{align*}
 \nu_{k+1} &=\frac{\phi(x^{\ell(k)})-\phi(x^{k+1})-\big(\phi(x^{\ell(k)})- \phi(x^{\ell(k+1)})\big)}{\phi(x^{\ell(k)})-\phi(x^{k+1})}\big( \phi (x^{k})-\phi(x^{k+1})+\nu _{k} \big) \\
   &=\left(1-\frac{\phi(x^{\ell(k)})- \phi(x^{\ell(k+1)})}{\phi(x^{\ell(k)})-\phi(x^{k+1})}\right)\left( \phi (x^{k})-\phi(x^{k+1})+\nu _{k} \right) \\
& \leq (1-\delta _{k+1})\big(\phi (x^{k})-\phi(x^{k+1})+\nu _{k} \big), 
\end{align*}
which shows that $\nu _{k+1}$ satisfies \eqref{nmr1}.
Therefore, the strategy defined in \eqref{eq:grippo} is a particular instance of \ref{it:deltamin} which turns Algorithm~\ref{Alg:ASSPM} into a non-monotone boosted version of the DCA employing the non-monotone line search proposed in \cite{Grippo1986}. It worth to mention that, since the non-monotone rule \eqref{eq:grippo} does not satisfy the condition $\nu_{k}>0$ for all $k\in \mathbb{N}$, it can not be used to boost the DCA in the case of $g$ is non-differentiable (see Proposition~\ref{prop15u} in the sequel). Therefore, such rule will be used only in the case of $g$ is continuously differentiable.
\end{example}

\section{Convergence analysis: $g$ possibly non-differentiable} \label{sec4}
The aim of this section is to present convergence results and iteration-complexity analysis of nmBDCA when the function $g$ is possibly non-differentiable. It is worth to mention that in the next result we need to assume that $\nu _{k}>0$. We begin by showing that Algorithm~\ref{Alg:ASSPM} is well-defined.
\begin{proposition}\label{prop15u}
 Let $(x^{k} )_{k \in \mathbb{N}}$ be the sequence generated by Algorithm~\ref{Alg:ASSPM}. For each $k\in\mathbb{N}$, assume that $d^{k}\neq 0$ and $\nu _{k}>0$. Then, the following statements hold:
\begin{enumerate}[label={({\roman*})}, ref={(\roman*)}]
	\item\label{it:prop15u:ddsc}
	There holds ${\hat{\delta} _{k}}:=\nu _{k}/ (g(y^{k}+d^{k})+g(x^k)-2g(y^{k}))>0$, and 
		\begin{equation*}
			\phi (y^{k}+\lambda d^{k})\leq \phi (y^{k}) - \rho\lambda^2 \| d^{k}\| ^{2}+\nu_k, \qquad \forall \lambda \in (0,{\delta}_{k}], 
		\end{equation*}
		where ${\delta} _{k}:= \min \{ {\hat{\delta} _{k}},1, ({3\sigma})/({2\rho})\}$. Consequently, the line search in Step 3 is well-defined.
	\item\label{it:prop15u:decrease}
		 $\phi (x^{k+1}) \leq \phi (x^{k}) - (\sigma+\rho\lambda_k^2) \| d^{k}\| ^{2}+\nu_k$.
\end{enumerate}
\end{proposition}
\begin{proof} 
Before starting the proof, we remind that $d^{k}=y^{k}-x^{k}$. To prove item~\ref{it:prop15u:ddsc}, assume that $d^{k}\neq 0$ and take $w^{k}\in\partial h(x^{k})$. Since $h$ is strongly convex with modulus $\sigma>0$, it follows from Theorem~\ref{teo2}\,\ref{it:teo2:fct-geq} that 
\begin{equation} \label{eq:sch}
h(y^{k}+\lambda d^{k})\geq h(y^k) + \lambda \langle s, d^k \rangle + \frac{\sigma}{2} \lambda^2\|d^k\|^{2}, \qquad \forall s \in\partial h(y^{k}).
\end{equation}
Moreover, taking into account that $w^{k}\in\partial h(x^{k})$, we can apply Theorem~\ref{teo2}\,\ref{it:teo2:inner-geq} to obtain that $\langle s, d^k \rangle \geq \langle w^k, d^k \rangle+ {\sigma}\|d^k\|^{2}$. Hence, \eqref{eq:sch} becomes 
\begin{equation} \label{eq:sscs}
h(y^{k}+\lambda d^{k})\geq h(y^k) + \lambda \langle w^k, d^k \rangle + {\sigma}\lambda\|d^k\|^{2}+ \frac{\sigma}{2} \lambda^2\|d^k\|^{2}.
\end{equation}
Considering that $y^{k}$ is the solution of \eqref{eq:ASSPMu} we have that $g(y^{k})- \left \langle w^{k},d^{k} \right\rangle \leq g(x^{k})$, which combining with \eqref{eq:sscs} yields 
\begin{equation} \label{eq:tscs}
-(h(y^{k}+\lambda d^{k})-h(y^k)) \leq \lambda\left( g(x^{k})- g(y^{k}) \right) - {\sigma}\lambda\|d^k\|^{2}- \frac{\sigma}{2} \lambda^2\|d^k\| ^{2}.
\end{equation}
On the other hand, by using the strong convexity of $g$ with modulus $\sigma>0$ we have 
\begin{align}\label{eq:IX}
g(y^{k}+\lambda d^{k})-g(y^{k}) & = g( \lambda (y^{k}+d^{k}) +(1-\lambda) y^{k} )-g(y^{k}) \notag \\
 & \leq \lambda g(y^{k}+d^{k}) + (1-\lambda)g(y^{k})- \frac{\sigma}{2}\lambda(1-\lambda)\|d^k\|^2 -g(y^{k})\notag \\
 & = \lambda\left(g(y^{k}+d^{k})-g(y^{k})\right)- \frac{\sigma}{2}\lambda(1-\lambda)\|d^k\|^2, 
\end{align}
for all $\lambda \in (0,1]$.
Combining the definition of $ \phi $ in \eqref{Pr:DCproblem} with \eqref{eq:tscs} and \eqref{eq:IX}, we obtain 
\begin{align}\label{eq:phiin}
 \phi (y^{k}+\lambda d^{k})-\phi (y^{k}) &=g(y^{k}+\lambda d^{k})-g(y^{k})-\left(h(y^{k}+\lambda d^{k})-h(y^{k}) \right) \notag \\
 &\leq -\frac{3\sigma}{2}\lambda\|d^k\|^2+\lambda\left(g(y^{k}+d^{k})+g(x^{k})-2g(y^{k})\right).
\end{align}
Moreover, it follows from Theorem~\ref{teo2}\,\ref{it:teo2:fct-geq} that 
\begin{equation*}
g(y^{k}+d^{k})\geq g(y^{k}) +\langle w, d^k \rangle + \frac{\sigma}{2} \|d^k\|^{2}, \qquad g(x^{k}) \geq g(y^{k}) -\langle w, d^k \rangle + \frac{\sigma}{2} \|d^k\|^{2},
\end{equation*}
for all $w\in \partial g(y^{k}),$ which implies that $g(y^{k}+d^{k})+g(x^k)-2g(y^{k}) \geq \sigma \|d^k\|^{2} >0$. Thus, due to $\nu _{k}>0$, we have $0<{\hat{\delta} _{k}}:=\nu _{k}/ (g(y^{k}+d^{k})+g(x^k)-2g(y^{k}))$, which proves the first statement of item~\ref{it:prop15u:ddsc}. Moreover, we have 
 \begin{equation*}
 0<\lambda \left(g(y^{k}+d^{k})+g(x^k)-2g(y^{k}) \right) \leq \nu _{k}, \qquad \lambda \in ( 0,{\hat{\delta} _{k}}].
 \end{equation*}
 Set ${\delta} _{k}:= \min \{ {\hat{\delta} _{k}},1, ({3\sigma})/({2\rho})\}$. Hence, the last inequality together with \eqref{eq:phiin} implies
\begin{align*} 
 \phi(y^{k}+\lambda d^{k})- \phi(y^{k}) \leq - \rho\lambda^2\|d^{k}\|^{2} +\nu _{k} ,\qquad \forall \lambda \in (0, \delta _{k}], 
\end{align*}
which concludes the second statement of the item~\ref{it:prop15u:ddsc}.
Finally, considering that $\lim _{j\to \infty} \zeta ^{j} \lambda _{k-1} =0$, it follows from the last inequality that the line search in {Step~3} is well-defined, and the proof of item~\ref{it:prop15u:ddsc} is concluded. To prove item~~\ref{it:prop15u:decrease}, we first note that item~\ref{it:prop15u:ddsc} implies that Step 4 is well-defined for $\nu _{k}>0$. Thus, \eqref{eq:jku} implies $\phi (y^{k}+\lambda_k d^{k})\leq \phi (y^{k}) - \rho\lambda_k^2 \| d^{k}\| ^{2}+\nu_k$, which combined with Proposition~\ref{pr:ffr}\,\ref{it:ffr:ineq} implies item~\ref{it:prop15u:decrease} and the proof of the proposition is completed.
\end{proof}

Note that if $x^{k+1}=x^k$, then from the definition of Algorithm~\ref{Alg:ASSPM} one can easily show that $d^{k}=0$, and hence, $x^{k}$ is a critical point of $\phi$. Therefore, from now on we assume that $d^{k}\neq 0$, or equivalently, that the sequence $(x^{k}) _{k\in \mathbb{N}}$ generated by Algorithm~\ref{Alg:ASSPM} is infinite.

\begin{remark} \label{re:aapn}
In the case that $g$ is convex and non-differentiable the direction $d^{k}\neq 0$ generated by Step~3 in Algorithm~\ref{Alg:ASSPM} is not in general a descent direction of $\phi$ at $y^{k}$, see Example~\ref{ExampleAlgorithm}. For this reason in Step~3 of Algorithm~\ref{Alg:ASSPM} we must assume that $\nu _{k}>0$, otherwise we cannot compute $\lambda _{k}>0$ satisfying \eqref{eq:jku}. However, as we will see in the Section~\ref{sec5}, whenever $g$ is convex and differentiable we just need to assume that $\nu _{k}\geq 0$ to compute $\lambda _{k}>0$ satisfying \eqref{eq:jku}.
\end{remark}


\subsection{Asymptotic convergence analysis}

Next, we prove the main results of this section.
\begin{theorem}\label{pr:assym}
If $\lim _{k\to\infty}\|d^{k}\|=0$, then every cluster point of $(x^k)_{k\in\mathbb{N}}$, if any, is a critical point of $\phi$. 
\end{theorem}
\begin{proof}
Let ${\bar x}$ be a cluster point of $(x^k)_{k\in\mathbb{N}}$, and $(x^{k_{\ell}})_{\ell\in \mathbb{N}}$ a subsequence of $(x^k)_{k\in\mathbb{N}}$ such that $\lim _{\ell\to \infty}x^{k_{\ell}}={\bar x}$. Let $(w^{k_{\ell}})_{\ell\in \mathbb{N}}$ and $(y^{k_{\ell}})_{\ell\in \mathbb{N}}$ be the according sequences generated by Algorithm~\ref{Alg:ASSPM}, i.e., $w^{k_{\ell}}\in\partial h(x^{k_{\ell}})$. From \eqref{eq:charyk} we have that $ w^{k_{\ell}} \in \partial g (y^{k_{\ell}})$. Since $\lim _{k\to\infty}\|d^{k}\|=0$ and $\lim _{\ell\to \infty}x^{k_{\ell}}={\bar x}$ we obtain that $\lim _{\ell\to \infty}y^{k_{\ell}}={\bar x}$. Considering that $w^{k_{\ell}}\in \partial h(x^{k_{\ell}})\cap \partial g(y^{k_{\ell}})$ and due to the convexity of $g$ and $h$, without loss of generality, we can apply Proposition \ref{cont_subdif} to obtain that $\lim _{\ell\to \infty}w^{k_{\ell}}=\bar{w}\in \partial g({\bar x}) \cap \partial h({\bar x}),$ which concludes the proof.
\end{proof}
\begin{theorem}\label{prop1a} 
If $(\nu_k)_{k\in\mathbb{N}}$ is chosen according to strategy \ref{it:deltamin}, then $(\phi (x^{k})+\nu _{k}) _{k\in \mathbb{N}}$ is non-increasing and convergent.
\end{theorem}
\begin{proof}
It follows from \eqref{eq:ccsa} in Remark~\ref{re:pip} that $(\phi (x^{k})+\nu _{k})_{k\in \mathbb{N}}$ is non-increasing. Therefore, by using \ref{it:phistarinf} and $(\nu _{k})_{k\in \mathbb{N}}\subset\mathbb{R}_{+}$, the desired result directly follows. 
\end{proof}
\begin{corollary}
If $(\nu_k)_{k\in\mathbb{N}}$ is chosen according to strategy \ref{it:deltamin} and $\lim _{k\to\infty}\nu _{k}=0$, then every cluster point of $(x^k)_{k\in\mathbb{N}}$, if any, is a critical point of $\phi$. 
\end{corollary}
\begin{proof}
Since $\lim _{k\to\infty} \nu_{k}=0$ from Theorem~\ref{prop1a} we have that $(\phi (x^{k})) _{k\in \mathbb{N}}$ is convergent. On the other hand, by Proposition~\ref{prop15u}\,\ref{it:prop15u:decrease} we obtain $0\leq \sigma \| d^{k}\| ^{2} \leq \phi (x^{k})+ \nu _{k}-\phi (x^{k+1})$, for all $k\in \mathbb{N}$. Therefore, taking limit in the last inequality we have that $\lim _{k\to\infty}\|d^{k}\|=0$. Finally, we apply Theorem~\ref{pr:assym} and the proof is complete.
\end{proof}
Next result proves the asymptotic convergence of Algorithm~\ref{Alg:ASSPM} when $(\nu _{k})_{k\in\mathbb{N}}$ is summable.
\begin{corollary}\label{coro:A2} 
If $(\nu_k)_{k\in\mathbb{N}}$ is chosen according to strategy \ref{it:nusummable}, then every cluster point of $(x ^{k})_{k\in\mathbb{N}}$, if any, is a critical point of $\phi$.
\end{corollary}
\begin{proof} Proposition~\ref{prop15u}\,\ref{it:prop15u:decrease} gives $0\leq \sigma ^{2}\|d^{k}\|^{2}\leq \phi (x^{k})-\phi (x^{k+1})+\nu _{k}$, for all $k\in \mathbb{N}$. Thus, using \ref{it:phistarinf} we obtain 
\begin{equation*}
\sum _{k=0}^{\infty}\|d^{k}\|^{2}\leq \frac{1}{\sigma}\Big(\phi (x^{0})-\phi ^{*}+\sum _{k=0}^{\infty}\nu _{k}\Big)<+\infty,
\end{equation*}
which implies that $\lim _{k\to\infty}\| d^{k}\|=0$. The desired result follows from Theorem~\ref{pr:assym}.
\end{proof}
\begin{corollary} Suppose that $(\nu_k)_{k\in\mathbb{N}}$ is chosen according to strategy \ref{it:deltamin}. If $\delta _{min}>0,$ then every cluster point of $(x^{k})_{k\in \mathbb{N}}$, if any, is a critical point of $\phi$.
\end{corollary}
\begin{proof} 
It follows by combining Remark~\ref{re:pip} with Corollary~\ref{coro:A2}. 
\end{proof}
\begin{corollary}\label{coro:A3} 
If $(\nu_k)_{k\in\mathbb{N}}$ is chosen according to strategy \ref{it:nubound}, then every cluster point of $(x ^{k})_{k\in\mathbb{N}}$, if any, is a critical point of $\phi$.
\end{corollary}
\begin{proof} From the definition of strategy \ref{it:nubound}, there exists $k_{0}\in \mathbb{N}$ such that $0\leq \nu _{k}\leq \sigma\|d^{k}\|^{2}/2$, for all $k\geq k_{0}$. Thus, 
$\sigma\|d^{k}\|^{2}/2\leq \sigma\|d^{k}\|^{2} - \nu_k$, for all $k\geq k_{0}$. 
Hence, using Proposition~\ref{prop15u}\,\ref{it:prop15u:decrease} we have
$0\leq \sigma\|d^{k}\|^{2}/2 \leq \phi (x^{k})-\phi (x^{k+1})$, for all $k\geq k_{0}$. Hence, using \ref{it:phistarinf} we conclude that $(\phi (x^{k})) _{k\geq k_{0}}$ is convergent. Furthermore, it follows that $\lim _{k\to\infty}\|d^{k}\|=0$. Therefore, applying Theorem~\ref{pr:assym} we obtain the desired result.
\end{proof}

\begin{remark}
In particular, Corollary~\ref{coro:A3} is also valid by replacing the strategy~\ref{it:nubound} by the alternative strategy \ref{eq:nub2}. Indeed, if we assume that $(\nu _{k})_{k\in \mathbb{N}}$ satisfies~\ref{eq:nub2}, then using Proposition~\ref{prop15u}\,\ref{it:prop15u:decrease} we have
$0< \sigma\|d^{k}\|^{2} \leq \phi (x^{k})-\phi (x^{k+1}) + \nu _{k} \leq \phi (x^{k})-\phi (x^{k+1}) + \bar{\delta}\|d^{k}\|^{2}$ for all $k\geq k_{0}$, which implies that
$0< (\sigma - \bar{\delta})\|d^{k}\|^{2} \leq \phi (x^{k})-\phi (x^{k+1}),$ for all $k\geq k_{0}$. Thus, using \ref{it:phistarinf} we conclude that $(\phi (x^{k})) _{k\geq k_{0}}$ is convergent and $\lim _{k\to\infty}\|d^{k}\|=0$. Therefore, the assertion holds by using Theorem~\ref{pr:assym}. 
\end{remark}

\subsection{Iteration-complexity analysis} 
 In this section some iteration-complexity bounds for $(x^{k})_{k\in \mathbb{N}}$ generated by Algorithm~\ref{Alg:ASSPM} are presented. Our results establish iteration-complexity bounds for the case where $(\nu_k)_{k\in\mathbb{N}}$ is chosen according to each one of the strategies \ref{it:nusummable} and \ref{it:nubound}.
Before present it, we first note that in particular Proposition~\ref{prop15u}\,\ref{it:prop15u:decrease} implies that 
\begin{equation}\label{eq:CGnd}
\sigma\| d^{k}\| ^{2}\leq \phi (x^{k})-\phi (x^{k+1}) +\nu _{k},\qquad \forall k\in \mathbb{N}.
\end{equation}

\begin{theorem}
Suppose that $(\nu_k)_{k\in\mathbb{N}}$ is chosen according to strategy \ref{it:nusummable}. For each $N\in\mathbb{N},$ we have
\begin{equation}\label{eq:comp_A2nd}
\min \left\{\| d^{k}\|:~k=0,1,\cdots,N-1\right\}\leq {\frac{ \sqrt{\phi(x^{0})-\phi^{*}+\sum _{k=0}^{\infty}\nu _{k}}}{ \sqrt{ \sigma }}}\frac{1}{ \sqrt{N}}.
\end{equation}
Consequently, for a given accuracy $\epsilon>0$, if $N\geq \left({ \phi (x^{0})-\phi ^{*}+\sum _{k=0}^{\infty}\nu _{k} }\right)/(\sigma \epsilon^{2})$, then $\min \{\| d^{k}\|: ~k=0,1,\cdots,N-1\}\leq \epsilon.$
\end{theorem}
\begin{proof} Since $\phi^{*}:=\inf _{x\in \mathbb{R}^{n}} \phi(x)\leq \phi (x^{k})$ for all $k\in\mathbb{N}$, from \eqref{eq:CGnd} we obtain that 
\begin{equation*}
\sum _{k=0}^{N-1}\|d^{k}\|^{2} \leq \frac{1}{\sigma}\Big(\phi(x^{0})-\phi(x^{N})+\sum _{k=0}^{N-1}\nu _{k}\Big) \leq \frac{1}{\sigma} \Big(\phi(x^{0})-\phi^{*}+\sum _{k=0}^{\infty}\nu _{k}\Big).
\end{equation*}
Therefore, $ N\min \{\| d^{k}\| ^{2}: ~k=0,1,\cdots,N-1\}\leq (\phi(x^{0})-\phi^{*}+\sum _{k=0}^{\infty}\nu _{k})/\sigma,$
and \eqref{eq:comp_A2nd} follows. The second statement is a directly consequence of the first one. 
\end{proof}
\begin{theorem}\label{th:comp_A3nd} 
Suppose that $(\nu_k)_{k\in\mathbb{N}}$ is chosen according to strategy \ref{it:nubound}. Let $0< \varsigma <1$ and $k_{0}\in \mathbb{N} $ such that $\nu _{k}\leq \varsigma \sigma\|d^{k}\|^{2}$, for all $k\geq k_{0}$. Then, for each $N\in\mathbb{N}$ such that $N> k_0$, one has
\begin{equation*}
\min \{\|d^{k}\| : k=0,1,\cdots,N-1\} \leq {\frac{ \sqrt{\phi(x^{0})-\phi^{*}+\sum _{k=0}^{k_0-1}\nu _{k}}}{ \sqrt{ (1-\varsigma)\sigma }}}\frac{1}{ \sqrt{N}}.
\end{equation*}
 Consequently, for a given $\epsilon>0$ and $k_{0}\in \mathbb{N}$ such that $\nu _{k}\leq \varsigma \sigma\|d^{k}\|^{2}$ for all $k\geq k_{0}$, if $N\geq \max\{k_0, ({ \phi (x^{0})-\phi ^{*}+\sum _{k=0}^{k_0-1}\nu _{k} })/(\sigma(1-\varsigma ) \epsilon^{2})\}$, then the following inequality holds $\min \left\{\| d^{k}\|: ~k=0,1,\cdots,N-1\right\}\leq \epsilon.$
\end{theorem}
\begin{proof} Let $\varsigma \in (0,1)$ and $k_{0}\in \mathbb{N} $ such that $\nu _{k}\leq \varsigma \sigma\|d^{k}\|^{2}$, for all $k\geq k_{0}$. 
It follows from \eqref{eq:CGnd} that $\sigma\|d^{k}\|^{2}\leq \phi (x^{k})-\phi (x^{k+1})+\nu _{k}$, for all $ k=0,1,\cdots, N-1$. Summing up the last inequality from $k=0$ to $k=N-1$ and using assumption \ref{it:phistarinf} we have
\begin{equation*}
\sigma\sum _{k=0}^{N-1}\|d^{k}\|^{2} \leq \phi (x^{0})-\phi ^{*}+\sum _{k=0}^{k_{0}-1} \nu _{k}+\sum _{k=k_{0}}^{N-1} \nu _{k}.
\end{equation*}
Hence, considering that $\nu _{k}\leq \varsigma \sigma\|d^{k}\|^{2}$, for all $k\geq k_{0}$, the last inequality becomes 
\begin{equation*}
\sum _{k=0}^{N-1}\sigma\|d^{k}\|^{2} \leq \phi (x^{0})-\phi ^{*}+\sum _{k=0}^{k_{0}-1} \nu _{k}+\sum _{k=0}^{N-1} \varsigma \sigma\|d^{k}\|^{2}, 
\end{equation*}
which is equivalent to $ \sum _{k=0}^{N-1} (1-\varsigma)\sigma\|d^{k}\|^{2} \leq \phi (x^{0})-\phi ^{*}+\sum _{k=0}^{k_{0}-1} \nu _{k}$. Therefore, we have $ N\min \{\| d^{k}\| ^{2}: ~k=0,1,\cdots,N-1\}\leq (\phi(x^{0})-\phi^{*}+\sum _{k=0}^{k_0-1}\nu _{k})/((1-\varsigma){\sigma})$, and the first inequality follows. The last inequality follows from the first one.
\end{proof}

\begin{remark}\label{rmk:ExemploA3} 
Theorem~\ref{th:comp_A3nd} may not seem very useful at first look, since the integer $k_0$ is not always known. However, specifically for the sequences $(\nu_{k})_{k\in \mathbb{N}}$ given in Example~\ref{eq:omega}, we are able to compute such integer $k_0$ explicitly. Indeed, given $\omega>0$ and $0<\varsigma <1$, if $\nu _{k}=\omega\|d^{k}\|^{2}/(k+1),$ then the integer $k_0$ such that $k\geq k_0$ implies $\nu_k\leq \varsigma \sigma\|d^k\|^{2}$ must satisfies $k_{0}\geq ({\omega}/{\varsigma\sigma})-1$. On the other hand, if $\nu _{k}={\omega}\|d^{k}\|^{2}/\ln (k+2),$ then some calculations show that $k_0\geq e^{({\omega}/{\varsigma \sigma})}-2.$
\end{remark}
\section{Convergence analysis: $g$ continuously differentiable} \label{sec5}
In this section we present an iteration-complexity analysis of nmBDCA when the function $g$ is continuously differentiable. We remark that in this section we just need to assume that $\nu _{k}\geq 0$. Hence, it is worth mentioning that, if $\nu _{k}= 0$ for all $k\in\mathbb{N}$, then non-monotone line search \eqref{eq:jku} merges into monotone line search \eqref{eq:BDCAjku}, i.e., Algorithm~\ref{Alg:ASSPM} is a natural extension of the BDCA introduced in \cite{ARAGON2019}. If $\nu _{k}> 0$, for all $k\in\mathbb{N}$, then nmBDCA can be viewed as an inexact version of BDCA. To proceed with the analysis of Algorithm~\ref{Alg:ASSPM} we need to assume, in addition to \ref{it:ghstronglyconvex} and \ref{it:phistarinf}, the following condition:

\begin{enumerate}[label={(\textbf{H\arabic*})}, start=3, ref={(H\arabic*)}]
	\item\label{it:gdiff}
		$g$ is continuously differentiable and $\nabla g$ is Lipschitz continuous with constant $L>0$.
\end{enumerate}

 Our first task is to establish the well-definition of Algorithm~\ref{Alg:ASSPM}, which will be done in the next proposition. 
\begin{proposition}\label{prop:gdif}
	Suppose that $g:\mathbb{R}^{n}\to \mathbb{R}$ satisfies \ref{it:gdiff}.
	For each $k\in\mathbb{N}$, assume that $d^{k}\neq 0$ and $\nu _{k}\geq 0$. Then, the following statements hold:
\begin{enumerate}[label={({\roman*})}, ref={(\roman*)}]
	\item\label{it:gdif:deltak}
		$\phi '(y^{k};d^{k})\leq -\sigma \|d^{k}\|^{2}<0$ and there exists a constant ${\delta}_{k}>0$ such that $\phi (y^{k}+\lambda d^{k})\leq \phi (y^{k}) - \rho\lambda^2 \| d^{k}\| ^{2}+\nu_k$, for all $\lambda \in (0,{\delta}_{k}]$. 
	Consequently, the line search in Step 4 is well-defined.
	\item\label{it:gdif:ineq}
		$\phi (x^{k+1}) \leq \phi (x^{k}) - (\sigma+\rho\lambda_k^2) \| d^{k}\| ^{2}+\nu_k$.
\end{enumerate}
\end{proposition}
\begin{proof}
The proof of item~\ref{it:gdif:deltak} follows from \cite[Proposition 3.1(ii)-(iii)]{ARAGON2019} together with the fact of $\nu _{k}\geq 0$. Finally, the proof of item~\ref{it:gdif:ineq} follows from Proposition~\ref{pr:ffr}\,\ref{it:ffr:ineq} and item~\ref{it:gdif:deltak}.
\end{proof}
In the sequel, we will establish a positive lower bound to the step-size $\lambda _{k}>0$ defined in {Step 4} of Algorithm~\ref{Alg:ASSPM} when $g$ satisfies \ref{it:gdiff}. Before proving such result, we will obtain a result that generalizes Lemma~\ref{eq:IneqLip} for DC functions. In fact, instead of assuming that the whole function $\phi=g-h$ has gradient Lipschitz, we assume that only the first DC component $g$ has such a property. The statement is as follows. 
\begin{lemma}\label{le:LipcLeDC}
Let $\phi:\mathbb{R}^{n}\to \mathbb{R}$ be given by $\phi(x)=g(x)-h(x)$, where $g$ satisfies \ref{it:gdiff} and $h$ is convex. Then, for all $x, d\in \mathbb{R}^n$ and $\lambda \in \mathbb{R}$, there holds 
\begin{equation*}
\phi (x+ \lambda d) \leq \phi (x) +\lambda \left\langle \nabla g(x)-w, d \right\rangle + \frac{L}{2} \lambda^2 \|d\|^2, \qquad \forall w \in \partial h(x).
\end{equation*}
Moreover, if $h$ is strongly convex with modulus $\sigma>0$, then
\begin{equation*} 
\phi (x+ \lambda d) \leq \phi (x) +\lambda \left\langle \nabla g(x)-w, d \right\rangle + \frac{(L-\sigma)}{2} \lambda^2 \|d\|^2, \qquad \forall w \in \partial h(x).
\end{equation*}
\end{lemma}
\begin{proof}
Let $x\in \mathbb{R}^{n}$ and an arbitrary $w \in \partial h(x)$. Define the function $\psi :\mathbb{R}^{n}\to \mathbb{R}$ by $\psi (z)=g(z)-\langle w,z\rangle$. Thus, we have $\nabla \psi(z)= \nabla g(z)-w$ and, since $\nabla g$ is Lipschitz continuous with constant $L$ we obtain that $\nabla \psi $ is also Lipschitz continuous with constant $L$. Given $d\in \mathbb{R}^n$ and $\lambda \in \mathbb{R}$, by using Lemma \ref{eq:IneqLip} with $\phi=\psi$, we obtain that $\psi(x+\lambda d)\leq \psi(x) + \lambda \langle \nabla g(x)-w, d \rangle + L\lambda^2\|d\|^2/2$. Since $\psi(z)=g(z)-\langle w,z\rangle$, the last inequality is equivalent to 
\begin{align}\label{eq:dsclemmaDC}
g(x+\lambda d) \leq g(x)+\lambda \langle w, d \rangle + \lambda \langle \nabla g(x)-w,d \rangle + \frac{L}{2}\lambda^2\|d\|^{2}.
\end{align}
 Since $h$ is convex and $w\in \partial h(x)$, we have $\lambda \langle w,d \rangle\leq h(x+\lambda d)-h(x)$. Thus, by using \eqref{eq:dsclemmaDC}, we obtain that $g(x+\lambda d)-h(x+\lambda d) \leq g(x)-h(x) + \lambda \langle \nabla g(x)-w, d \rangle + L\lambda^2\|d\|^{2}/2$. Due to $\phi=g-h$, the last inequality is equivalent to the first assertion. On the other hand, if we assume that $h$ is strongly convex with modulus $\sigma> 0$ and $w\in \partial h(x),$ it follows from item~$(ii)$ of Theorem~\ref{teo2} that $\lambda \langle w,d \rangle\leq h(x+\lambda d)-h(x)-\sigma\lambda^2\|d\|^2/2 $. Hence, the last inequality together \eqref{eq:dsclemmaDC} yield 
$g(x+\lambda d)-h(x+\lambda d) \leq g(x)-h(x) + \lambda \langle \nabla g(x)-w, d \rangle + (L-\sigma)\lambda^2\|d\|^{2}/2$. Therefore, taking into account that $\phi=g-h$ the proof is concluded.
\end{proof} 
\begin{remark}
It is worth to note that in Lemma~\ref{le:LipcLeDC} it is sufficient to assume \ref{it:gdiff}. In this case {\cite[Corollary of Proposition 2.2.1, p. 32]{clarke1983optimization}} ensures that if $g$ is continuously differentiable, then $g$ is locally Lipschitz. Hence, by Theorem~\ref{subdif_DC},
we have $\partial _{c}\phi(x)=\{\nabla g(x)\}-\partial h(x)$, and the desired result follows. We also note that the Lemma~\ref{le:LipcLeDC} generalizes Lemma~\ref{eq:IneqLip}. Indeed, taking $h\equiv 0$ in the first part of Lemma~\ref{le:LipcLeDC}, it becomes Lemma~\ref{eq:IneqLip}.
\end{remark}

Before stating the next result, we need to define the following useful constant:
\begin{equation}\label{eq:thetamu}
\lambda _{\min}:=\min \left\lbrace \lambda _{-1},\frac{2\zeta\sigma}{(L+2\rho)}\right\rbrace.
\end{equation}

\begin{lemma}\label{le:lmin} If $g$ satisfies \ref{it:gdiff}, then $\lambda _{k}\geq \lambda _{\min},$ for all $k\in \mathbb{N}.$
\end{lemma}
\begin{proof} We will show by induction on $k$ that $\lambda _{k}\geq \lambda _{\min},$ for all $k\in \mathbb{N}.$ Set $k=0$. If $j_{0}=0$, then $\lambda _{0}=\lambda _{-1}$. Thus, \eqref{eq:thetamu} implies that $\lambda _{0}\geq \lambda _{\min}$. Otherwise, assume that $j _{0}>0$. Since $\lambda _{0}=\zeta^{j_0}{\lambda _{-1}}$ we conclude from \eqref{eq:jku} that 
\begin{equation}\label{eq:lmin10}
\phi \left(y^{0} + \frac{\lambda _{0}}{\zeta}d^{0} \right)- \phi (y^{0})> - \rho \frac{\lambda_{0}^{2}}{\zeta^{2}}\|d^{0}\|^{2} + \nu _{0}.
\end{equation}
On the other hand, using Lemma~\ref{le:LipcLeDC} with $x=y^{0}$, $\lambda = \lambda_{0}/\zeta$ and $d=d^{0}$ we obtain
\begin{equation}\label{eq:lmin19.5}
\phi \left( y^{0}+\frac{\lambda_{0}}{\zeta} d^{0} \right)-\phi (y^{0}) \leq \frac{\lambda_{0}}{\zeta} \langle \nabla g(y^{0})-s^{0}, d^{0} \rangle + \frac{L}{2}\frac{\lambda_{0}^{2}}{\zeta^{2}}\| d^{0}\| ^{2},\qquad \forall s^{0}\in \partial h(y^{0}).
\end{equation}
 From \eqref{eq:charyk} we have $\nabla g(y^{0})=w^{0}\in \partial h(x^{0}).$ Thus, from the strong convexity of $h$ and Theorem~\ref{teo2}\,\ref{it:teo2:inner-geq} we have
$\langle \nabla g(y^{0})-s^{0}, d^{0} \rangle = \langle w^{0}-s^{0},y^{0}-x^{0} \rangle \leq -\sigma \|d^{0}\|^{2}.$ Therefore, since $\nu _{0}\geq 0,$ we obtain from the last inequality and \eqref{eq:lmin19.5} that
\begin{equation}\label{eq:lmin20}
\phi \left( y^{0}+\frac{\lambda_{0}}{\zeta} d^{0} \right)-\phi (y^{0}) \leq -\frac{\lambda_{0}\sigma}{\zeta} \| d^{0}\|^{2} + \frac{L}{2}\frac{\lambda_{0}^{2}}{\zeta^{2}}\| d^{0}\| ^{2}+\nu_{0}.
\end{equation}
Combining \eqref{eq:lmin10} and \eqref{eq:lmin20} we have
\begin{equation*}
- \frac{\lambda_{0}^{2}\rho }{\zeta^{2}}\|d^{0}\|^{2}< -\frac{ \lambda _{0} \sigma }{ \zeta} \|d^{0}\|^{2} + \frac{L}{2}\frac{\lambda_{0}^{2}}{\zeta^{2}}\| d^{0}\| ^{2} .
\end{equation*}
Hence, considering that $\rho>0$, $L>0$ and $d^{0}\neq 0$, some algebraic manipulations show that $\lambda_{0} > 2 \zeta \sigma /(L +2\rho)\geq \lambda_{\min}$. Therefore, the inequality holds for $k=0$. Now, we assume that $\lambda _{k-1}\geq \lambda _{\min}$ for some $k>0$. If $j^{k}=0$, then $\lambda _{k}=\lambda _{k-1}\geq \lambda _{\min}.$ Otherwise, if $j_{k}>0$, then repeating the above argument with $\lambda_0$ replaced by $\lambda_{k-1}$, we obtain that $\lambda_{k} > 2 \zeta \sigma /(L +2\rho)\geq \lambda_{\min},$ which completes the proof. 
\end{proof}
\begin{corollary}\label{coro:lmin} Assume that the function $g$ satisfies conditions \ref{it:gdiff}. Then $(\sigma + \rho \lambda _{\min}^{2})\|d^{k}\|^{2}\leq \phi (x^{k})-\phi (x^{k+1})+\nu _{k}$, for all $k\in \mathbb{N}$. 
\end{corollary}
\begin{proof}
It follows from Proposition~\ref{prop:gdif}\,\ref{it:gdif:ineq} and Lemma~\ref{le:lmin}.
\end{proof}
\subsection{Iteration complexity bounds} 
The aim of this section is to present some iteration-complexity bounds for the sequence $(x^{k}) _{k\in \mathbb{N}}$ generated by Algorithm~\ref{Alg:ASSPM} in the case that $g$ is differentiable. 
\begin{lemma} \label{eq:nfeas} Suppose that $g$ satisfies \ref{it:gdiff}. Let $j_{k}\in \mathbb{N}$ be the integer defined in \eqref{eq:jku}, and $J_{k}$ be the number of function evaluations $\phi$ in \eqref{eq:jku} after $k\geq 1$ iterations of Algorithm \ref{Alg:ASSPM}. Then, $j_{k}\leq {({\log \lambda_{\min}-\log {\lambda _{-1}} })}/{\log \zeta} $
and 
\begin{equation*}
J_{k}\leq 2(k+1)+ \frac{\log \lambda_{\min} - \log {\lambda _{-1}}}{\log \zeta} .
\end{equation*}
\end{lemma}
\begin{proof} It follows from the {step~4} in Algorithm~\ref{Alg:ASSPM} together with Lemma~\ref{le:lmin} that $0< \lambda_{\min}\leq \lambda _{k}= \zeta ^{j_{k}}{\lambda _{k-1}}\leq \lambda _{-1}$, for all $k\in \mathbb{N}$. Thus, taking logarithm in last inequalities we obtain that $\log \lambda_{\min}\leq\log \lambda _{k}=j_{k} \log \zeta + \log {\lambda _{k-1}}\leq \log {\lambda _{-1}}$, for all $k\in \mathbb{N}$. Hence, taking into account that $\zeta \in (0,1)$ and $ 0<\lambda _{k-1}\leq \lambda _{-1}$, we have 
\begin{equation*}
 j_{k}= \frac{\log \lambda _{k}-\log {\lambda _{k-1}}}{\log \zeta} \leq \frac{\log \lambda_{\min}-\log {\lambda _{k-1}}}{ \log \zeta } \leq \frac{\log \lambda_{\min}-\log {\lambda _{-1}}}{ \log \zeta }, \quad \forall k\in \mathbb{N}.
\end{equation*}
This prove the first inequality. To prove the second assertion, we sum up the above inequality from $l=0$ to $k$ and we obtain 
 \begin{align*}
\sum _{\ell=0}^{k}j_{\ell} = \sum _{\ell =0}^{k} \frac{\log \lambda_{\ell}-\log {\lambda _{\ell -1}} }{\log \zeta} 
 =\frac{\log \ \lambda_{k } - \log {\lambda _{-1}}}{\log \zeta} \leq \frac{\log \lambda_{min } - \log {\lambda _{-1}}}{\log \zeta}.
\end{align*}
On the other hand, the definition of $J_{k}$ implies that $J_{k}= \sum _{\ell=0}^{k}(j_{\ell}+2)=2(k+1)+\sum _{\ell=0}^{k}j_{\ell}$. Therefore, 
by using the last inequality the desired inequality follows. 
\end{proof}
Next results establish iteration-complexity bounds when $(\nu _{k})_{k\in \mathbb{N}}$ is summable.
\begin{theorem}\label{th:comp_A2} 
Suppose that $(\nu_k)_{k\in\mathbb{N}}$ is chosen according to strategy \ref{it:nusummable} and $g$ satisfies \ref{it:gdiff}. For each $N\in\mathbb{N},$ we have
\begin{equation}\label{eq:comp_A2}
\min \left\{\| d^{k}\|: ~k=0,1,\cdots,N-1\right\}\leq {\frac{ \sqrt{\phi(x^{0})-\phi^{*}+\sum _{k=0}^{\infty}\nu _{k}}}{ \sqrt{ \sigma + \rho \lambda _{\min}^{2} }}}\frac{1}{ \sqrt{N}}.
\end{equation}
Consequently, for a given $\epsilon>0$, if $ N\geq { \big(\phi (x^{0})-\phi ^{*}+\sum _{k=0}^{\infty}\nu _{k} \big)}/{((\sigma + \rho \lambda _{\min}^{2}) \epsilon ^{2})}$, then
$\min \left\{\| d^{k}\|: ~k=0,1,\cdots,N-1\right\}\leq \epsilon.$
\end{theorem}
\begin{proof} Since $\phi^{*}:=\inf _{x\in \mathbb{R}^{n}} \phi(x)\leq \phi (x^{k})$ for all $k\in\mathbb{N}$, from Corollary~\ref{coro:lmin}, we obtain
\begin{equation*}
(\sigma + \rho \lambda _{\min}^{2})\sum _{k=0}^{N-1}\|d^{k}\|^{2} \leq \phi(x^{0})-\phi(x^{N+1})+\sum _{k=0}^{N-1}\nu _{k} \leq \phi(x^{0})-\phi^{*}+\sum _{k=0}^{\infty}\nu _{k}.
\end{equation*}
Thus,
\begin{equation*}
 N\min \{\| d^{k}\| ^{2}: ~k=0,1,\cdots,N-1\}
 \leq
 \frac{ \phi(x^{0})-\phi^{*}+\sum_{k=0}^{\infty}\nu _{k}}{ \sigma +\rho\lambda_{\min}^{2}},
\end{equation*}
and \eqref{eq:comp_A2} follows. The second statement is an immediately consequence of the first one. 
\end{proof}
\begin{theorem} Suppose that $(\nu_k)_{k\in\mathbb{N}}$ is chosen according to strategy \ref{it:nusummable} and $g$ satisfies \ref{it:gdiff}. For a given $\epsilon>0$, the number of function evaluations $\phi$ in Algorithm \ref{Alg:ASSPM} to compute $d^k$ such that $\| d^{k}\|\leq \epsilon$ is at most 
 \begin{equation*}
 2\Big(\frac{\phi (x^{0})-\phi ^{*}+\sum _{k=0}^{\infty}\nu _{k}}{(\sigma +\rho\lambda_{\min}^{2}) \epsilon ^{2}}+1\Big)+ \frac{\log \lambda_{\min} - \log {\lambda _{-1}}}{\log \zeta}.
 \end{equation*}
\end{theorem}
\begin{proof}
The proof follows upon combining Lemma~\ref{eq:nfeas} with Theorem~\ref{th:comp_A2}.
\end{proof}

\begin{theorem}\label{th:comp_A3_dif} 
Suppose that \ref{it:nubound} holds and $g$ satisfies \ref{it:gdiff}. Let $0<\varsigma <1$ and $k_{0}\in \mathbb{N} $ such that $\nu _{k}\leq \varsigma (\sigma +\rho\lambda_{\min}^{2})\|d^{k}\|^{2}$, for all $k\geq k_{0}$. Then, for each $N\in\mathbb{N}$ such that $N >k_0$, there holds
\begin{equation*}
\min \{\|d^{k}\| : k=0,1,\cdots,N-1\} \leq {\frac{ \sqrt{\phi(x^{0})-\phi^{*}+\sum _{k=0}^{k_{0}-1}\nu _{k}}}{ \sqrt{ (1-\varsigma)(\sigma +\rho\lambda_{\min}^{2}) }}}\frac{1}{ \sqrt{N}}.
\end{equation*}
Consequently, for a given $\epsilon>0$ and $k_{0}\in \mathbb{N}$ such that $\nu _{k}\leq \varsigma (\sigma +\rho\lambda_{\min}^{2})\|d^{k}\|^{2}$ for all $k\geq k_{0}$, if $N\geq \max\{k_0, ({\phi (x^{0})-\phi ^{*}+\sum _{k=0}^{k_{0}-1}\nu _{k}})/({(1-\varsigma )(\sigma +\rho\lambda_{\min}^{2}) \epsilon ^{2}})\}$, then the following inequality holds $\min \left\{\| d^{k}\|: ~k=0,1,\cdots,N-1\right\}\leq \epsilon.$
\end{theorem}
\begin{proof} Let $\varsigma \in (0,1)$ and $k_{0}\in \mathbb{N} $ such that $\nu _{k}\leq \varsigma (\sigma +\rho\lambda_{\min}^{2})\|d^{k}\|^{2}$, for all $k\geq k_{0}$. 
It follows from Corollary~\ref{coro:lmin} that 
\begin{align*}
(\sigma +\rho\lambda_{\min}^{2})\|d^{k}\|^{2}\leq \phi (x^{k})-\phi (x^{k+1})+\nu _{k}, \qquad k=0,1,\cdots, N.
\end{align*}
Summing up last inequality from $k=0$ to $k=N$ and using assumption \ref{it:phistarinf}, we have
\begin{equation*}
(\sigma +\rho\lambda_{\min}^{2})\sum _{k=0}^{N-1}\|d^{k}\|^{2} \leq \phi (x^{0})-\phi ^{*}+\sum _{k=0}^{k_{0}-1} \nu _{k}+\sum _{k=k_{0}}^{N-1} \nu _{k}.
\end{equation*}
Hence, due to $\nu _{k}\leq \varsigma (\sigma +\rho\lambda_{\min}^{2})\|d^{k}\|^{2}$, for all $k\geq k_{0}$, the last inequality becomes 
\begin{equation*}
(\sigma +\rho\lambda_{\min}^{2})\sum _{k=0}^{N-1}\|d^{k}\|^{2} \leq \phi (x^{0})-\phi ^{*}+\sum _{k=0}^{k_{0}-1} \nu _{k}+\varsigma (\sigma +\rho\lambda_{\min}^{2})\sum _{k=0}^{N-1} \|d^{k}\|^{2}, 
\end{equation*}
which is equivalent to $(1-\varsigma)(\sigma +\rho\lambda_{\min}^{2}) \sum _{k=0}^{N-1} \|d^{k}\|^{2} \leq \phi (x^{0})-\phi ^{*}+\sum _{k=0}^{k_{0}-1} \nu _{k}$. Thus $ N\min \{\| d^{k}\| ^{2}: ~k=0,1,\cdots,N-1\}\leq (\phi(x^{0})-\phi^{*}+\sum _{k=0}^{k_0-1}\nu _{k})/((1-\varsigma){(\sigma +\rho\lambda_{\min}^{2})})$, and the first inequality follows. The last inequality follows from the first one.
\end{proof}

\begin{remark} 
For each one of the sequences $(\nu_{k})_{k\in \mathbb{N}}$ that appears in Example~\ref{eq:omega}, we already know the value of $k_0$ satisfying Theorem~\ref{th:comp_A3_dif}, see Remark~\ref{rmk:ExemploA3}. 
\end{remark}

\begin{theorem} 
Suppose that $(\nu_k)_{k\in\mathbb{N}}$ is chosen according to strategy \ref{it:nubound} and $g$ satisfies \ref{it:gdiff}. Let $0<\varsigma <1$ and $k_{0}\in \mathbb{N} $ such that $\nu _{k}\leq \varsigma (\sigma +\rho\lambda_{\min}^{2})\|d^{k}\|^{2}$, for all $k\geq k_{0}$. Then, the number of function evaluations in Algorithm \ref{Alg:ASSPM} to compute $d^k$ such that $\| d^{k}\|\leq \epsilon$ is at most 
 \begin{equation*}
2\Big(\frac{\phi (x^{0})-\phi ^{*}+\sum _{k=0}^{k_{0}-1}\nu _{k}}{(1-\varsigma)(\sigma +\rho\lambda_{\min}^{2}) \epsilon ^{2}}+1\Big) + \frac{\log \lambda_{\min} - \log {\lambda _{-1}}}{\log \zeta}.
 \end{equation*}
\end{theorem} 
\begin{proof}
The proof follows combining Lemma~\ref{eq:nfeas} with Theorem~\ref{th:comp_A3_dif}.
\end{proof}
\subsection{Full convergence under the Kurdyka-\L{}ojasiewicz property}

The aim of this section is to present the full convergence for the sequence $(x^{k}) _{k\in \mathbb{N}}$ generated by Algorithm~\ref{Alg:ASSPM} under the assumption that $\phi$ satisfies the Kurdyka-\L{}ojasiewicz property (in short K\L{} property) at a cluster point $x^{*}$ of $(x^{k})_{k\in \mathbb{N}}$. Before, let us recall the definition of K\L{} property; see for instance \cite{AttouchSoubeyran2010, AttouchSvaiter2013} and \cite{ARAGON2019,CruzNetoEtAl2018} for this concept in the DC context.

\begin{definition}
Let $C^{1}[(0,+\infty)]$ be the set of all continuously differentiable functions defined in $(0,+\infty)$, $f\colon\mathbb{R}^{n}\to \mathbb{R}$ be a locally Lipschitz function and $ \partial _{c} f(\cdot)$ be the Clarke's subdifferential of $f$. The function $f$ is said to have the Kurdyka-\L{}ojasiewicz property at $x^{*}$ if there exist $\eta\in (0,+\infty]$, a neighborhood $U$ of $x^{*}$ and a continuous concave function $\gamma : [0,\eta)\to \mathbb{R}_{+}$ (called desingularizing function) such that $\gamma (0)=0,\; \gamma \in C^{1}[(0,+\infty)]$ and $\gamma'(t)>0$ for all $t\in (0,\eta)$. In addition, the function $f$ satisfies $\gamma '(f (x)-f (x^{*})) \mbox{dist}(0, \partial _{c} f(x))\geq 1$, for all $x\in U\cap \{x\in\mathbb{R}^{n}\;|\; f (x^{*})<f (x)<f (x^{*})+\eta \}$, where $\mbox{dist}(0, \partial _{c} f(x)):=\inf\{\|s\|: ~s\in \partial _{c} f(x)\}.$
\end{definition}

The technique in the proof of next theorem is similar to the one used in seminal works \cite{AttouchSoubeyran2010, AttouchSvaiter2013}. Since we used strategy \ref{it:nubound}, we decide to include the proof here for sake of completeness.

\begin{theorem}\label{th:conv_kl}
Suppose that $(\nu_k)_{k\in\mathbb{N}}$ is chosen according to strategy \ref{it:nubound}. Assume that $(x^{k}) _{k\in \mathbb{N}}$ has a cluster point $x^{*}$, $\nabla g$ is locally Lipschitz continuous around $x^{*}$, and that $\phi$ satisfies the K-\L{} property at $x^{*}$. Then $(x^{k}) _{k\in \mathbb{N}}$ converges to $x^{*}$, which is a critical point of $\phi$.
\end{theorem}
\begin{proof} Since $(\nu _{k}) _{k\in \mathbb{N}}$ satisfies \ref{it:nubound}, there exists $k_{0}\in \mathbb{N}$ such that $ \nu _{k}\leq (\sigma / 2) \|d^{k}\|^{2}$, for all $k\geq k_{0}$. Hence, we have 
\begin{equation*}
0< (\sigma / 2) \|d^{k}\|^{2} = \sigma \|d^{k}\|^{2} - (\sigma / 2) \|d^{k}\|^{2} \leq \sigma \|d^{k}\|^{2} -\nu_{k}, \qquad \forall k\geq k_{0}.
\end{equation*}
Combining the last inequality with Proposition~\ref{prop:gdif}\,\ref{it:gdif:ineq} we obtain
\begin{equation}\label{eq:monoto}
0 < (\sigma / 2) \|d^{k}\|^{2} \leq (\sigma +\rho \lambda _{k}^{2}) \|d^{k}\|^{2}-\nu _{k} \leq \phi (x^{k})-\phi (x^{k+1}), \qquad \forall k \geq k_{0}.
\end{equation}
Since $x^{*}$ is a cluster point of $(x^{k}) _{k\in \mathbb{N}},$ there exists a subsequence $(x^{k_{\ell}}) _{\ell \in \mathbb{N}}$ of $(x^{k}) _{k\in \mathbb{N}}$ such that $\lim _{\ell \to\ \infty}x^{k_{\ell}}=x^{*}$, which combined with \eqref{eq:monoto} implies that $\lim _{k\to \infty }\phi (x^{k})=\phi (x^{*})$.
If there exists an integer $k\geq k_{0}$ such that $\phi (x^{k})=\phi (x^{*}),$ then \eqref{eq:monoto} implies that $d^{k}=0.$ In this case, Algorithm~\ref{Alg:ASSPM} stops after a finite number of steps and the proof is concluded. Now, suppose that $\phi (x^{k})>\phi (x^{*})$ for all $k\geq k_{0}.$ Since $\nabla g$ is locally Lipschitz around $x^{*}$, there exist ${\hat \delta}>0$ and $L>0$ such that
\begin{equation}\label{gLip}
\|\nabla g(x)-\nabla g(y) \| \leq L \|x-y \|, \qquad \forall x,y \in B(x^{*},{\hat \delta}).
\end{equation}
Since $\phi$ satisfies the Kurdyka-\L{}ojasiewicz inequality at $x^{*}$, there exist $\eta \in (0, +\infty]$, a neighborhood $U$ of $x^{*}$ , and a continuous and concave function $\gamma: [0,\eta) \to \mathbb{R} _{+} $ such that for every $x \in U$ with $\phi (x^{*}) < \phi(x) < \phi(x^{*}) + \eta$, we have
\begin{equation}\label{eq:kl}
\gamma '(\phi (x)-\phi (x^{*}))\mbox{dist}(0,\partial _{c}\phi (x))\geq 1.
\end{equation}
Take ${\tilde \delta}>0$ such that $B(x^{*},{\tilde \delta})\subset U$ and set $\delta :=\frac{1}{2}\min \{{\hat \delta},{\tilde \delta}\}>0$. Considering that $\lim _{k\to \infty }\phi (x^{k})=\phi (x^{*})$, it follows from \eqref{eq:monoto} that $\lim _{k\to \infty}d^{k}=0$. Then, there exists $k_{1}\in \mathbb{N}$ such that
$\|y^{k}-x^{k}\|=\|d^k\|\leq \delta $ for all $k\geq k_{1}.$ Thus, for all $k\geq k_{1}$ such that $x^{k}\in B(x^{*},{\delta})$ we obtain that $
\|y^{k}-x^{*}\|\leq \|y^{k}-x^{k}\|+\|x^{k}-x^{*}\|\leq 2\delta \leq {\hat \delta}$. Hence, for all $k\geq k_{1}$ such that $x^{k}\in B(x^{*},{\delta})$ we obtain $x^k, y^{k}\in B(x^{*},{\hat \delta})$, and using \eqref{gLip} we conclude that $\|\nabla g(x^k)-\nabla g(y^k) \| \leq L \|x^k-y^k \|$. Hence, using that $\nabla g(x^{k})-w^{k}\in \partial _{c}\phi (x^{k})$, $w^{k}=\nabla g(y^{k})$ and $x^{k+1}-x^k=(1+\lambda_k) (y^k-x^k)$ we have 
\begin{equation} \label{eq:fckl}
\mbox{dist}(0,\partial _{c}\phi (x^{k}))\leq \|\nabla g(x^{k})-w^{k}\|= \|\nabla g(x^{k})-\nabla g(y^{k})\|\leq \frac{L}{1+\lambda _{k}} \|x^{k+1}-x^{k}\|, 
\end{equation}
 for all $k\geq k_{1}$ such that $x^{k}\in B(x^{*},{\delta})$. To simplify the notations we set 
 \begin{equation} \label{eq:Kkl}
K:= \frac{2L(1+\lambda _{-1})}{\sigma}>0.
\end{equation}
Since $\lim _{\ell \to \infty}x^{k_{\ell}} = x^{*}$, $\lim _{k\to \infty} \phi(x^{k}) =
\phi (x^{*})$ and $\phi( x^{k} ) > \phi (x^{*})$, for all $k\geq k_{0}$, and $\phi$ is continuous, we can take an index $N\geq \max\{ k_{0},k_{1}\}$ such that
\begin{equation}\label{eq:xN}
x_{N}\in B(x^{*},\delta)\subset U, \qquad \phi (x^{*})<\phi (x^{N})<\phi (x^{*})+\eta.
\end{equation}
Furthermore, due to $\gamma (0)=0$, we can also assume that $N\geq \max\{ k_{0},k_{1}\}$ satisfies 
\begin{equation}\label{eq:xN2}
\|x^{N}-x^{*}\| + K \gamma (\phi (x^{N})-\phi (x^{*}))<\delta.
\end{equation}
On the other hand, for $k \geq N$ such that $x^{k} \in B(x^{*}, \delta)\subset U$, \eqref{eq:kl} and \eqref{eq:fckl} yield 
\begin{equation*}
 \gamma '(\phi (x^{k})-\phi (x^{*}))\geq \frac{1}{\mbox{dist}(0,\partial _{c}\phi (x^{k}))}\geq \frac{1+\lambda_{k} }{L\|x^{k}-x^{k+1}\|}.
\end{equation*}
Thus, due to $\gamma$ be concave, combining the last inequality with \eqref{eq:monoto} we have 
\begin{align*}
\gamma (\phi(x^{k})-\phi (x^{*})) -\gamma (\phi(x^{k+1})-\phi (x^{*})) &\geq \gamma '(\phi(x^{k})-\phi (x^{*}))(\phi (x^{k})-\phi (x^{k+1})) \\
& \geq \frac{ 1+\lambda _{k}}{L\|x^{k+1}-x^{k}\|} \frac{ \sigma \| d^{k}\|^{2}}{2}.
\end{align*}
Hence, using that $x^{k+1}-x^k=(1+\lambda_k) d^k$, $0<\lambda_k\leq \lambda_{-1}$ and \eqref{eq:Kkl}, we obtain
\begin{equation}\label{eq:sum1}
\|x^{k+1}-x^{k}\| \leq K\left( \gamma (\phi(x^{k})-\phi (x^{*})) -\gamma (\phi(x^{k+1})-\phi (x^{*})) \right),
\end{equation}
 for all $k \geq N$ such that $x^{k}\in B(x^{*},\delta)$. In the next step we will prove by induction that $x^{k}\in B(x^{*},\delta)$ for all $k\geq N$. For $k=N,$ the statement is valid due to the inclusion in \eqref{eq:xN}. Now, suppose that $x^{k}\in B(x^{*},\delta)$ for all $k=N+1, \cdots,N+p-1$ for some $p\geq 2$. Since $\phi( x^{k} ) > \phi (x^{*})$ for all $k\geq k_{0}$, from \eqref{eq:monoto}, \eqref{eq:xN} and $N\geq \max\{ k_{0},k_{1}\}$ we conclude that $\phi(x^{*})<\phi (x^{k+1})<\phi (x^{k})<\phi(x^{*})+\eta$, for all $k=N+1, \cdots,N+p-1$. We proceed to prove that $x^{N+p}\in B(x^{*},\delta)$. First, by using triangular inequality, induction hypothesis and \eqref{eq:sum1}, we have 
 \begin{align*}
\| x^{N+p}-x^{*}\| &\leq \| x^{N}-x^{*} \| + \sum _{i=1}^{p}\| x^{N+i}-x^{N+i-1}\|\\
 &\leq \| x^{N}-x^{*} \| + K \sum _{i=1}^{p}\left[ \gamma (\phi(x^{N+i-1})- \phi (x^{*})) -\gamma (\phi(x^{N+i})-\phi (x^{*})) \right].
 \end{align*}
Summing up last inequality and taking into account that $ \gamma(\phi(x^{N+p})-\phi (x^{*}))\geq 0$ and \eqref{eq:xN2} we obtain
\begin{align*}
\| x^{N+p}-x^{*}\| &= \| x^{N}-x^{*} \| + K \gamma(\phi(x^{N})-\phi (x^{*}))- K \gamma(\phi(x^{N+p})-\phi (x^{*}))\\
 &\leq \| x^{N}-x^{*} \| + K \gamma\big(\phi(x^{N})-\phi (x^{*})\big) <\delta,
 \end{align*}
which concludes the induction. Finally, considering that $x^{k}\in B(x^{*},\delta)$ for all $k\geq N$, similar argument used above together with \eqref{eq:xN2} and \eqref{eq:sum1} yields 
\begin{align*}
\sum _{k=N}^{N+p} \|x^{k+1}-x^{k}\| & \leq \sum _{k=N}^{N+p} K\left( \gamma (\phi(x^{k})-\phi (x^{*})) -\gamma (\phi(x^{k+1})-\phi (x^{*})) \right)\\
 & = K \gamma(\phi(x^{N})-\phi (x^{*}))- K \gamma(\phi(x^{N+p})-\phi (x^{*}))\\
 &\leq K \gamma(\phi(x^{N})-\phi (x^{*}) <\delta. 
\end{align*} 
Taking the limit in last inequality as $p$ goes to $\infty$ we have $\sum _{k=N}^{\infty }\| x^{k}-x^{k+1} \| <\infty,$ which implies that $ (x^k)_{k\in\mathbb{N}}$ is a Cauchy sequence. Hence, due to $x^{*}$ be a cluster point of $ (x^k)_{k\in\mathbb{N}}$, then the whole sequence $ (x^k)_{k\in\mathbb{N}}$ converges to $x^{*}$. Therefore, by using Corollary~\ref{coro:A3}, the proof is concluded. 
\end{proof}
\begin{remark} If $\nu_{k}\equiv 0$, then Algorithm~\ref{Alg:ASSPM} becomes the BDCA given in \cite{ARAGON2019}, and consequently Theorem~\ref{th:conv_kl} merges into \cite[Theorem 4.3]{ARAGON2019}. 
\end{remark}
\section{Numerical experiments} \label{Sec6}
In this section, we present some numerical experiments to verify the practical efficiency of the proposed non-monotone BDCA. The experiments were coded in MATLAB R2020b on a notebook 8 GB RAM Core i7. To evaluate its performance, we run it for some academic tests functions existing in the DC literature (see \cite{ARAGON2019,CruzNetoLopesSantosSouza2019,Joki2017}). 

The aim of this section is to show that the non-monotone BDCA has a good performance as its monotone version proposed by \cite{ARAGON2019} compared to the classical DC Algorithm (DCA \cite{Pham1986}). Additionally, we also compare its performance with the proximal point method for DC functions (PPMDC~\cite{Moudafi2006, SOUZA3016, SunSampaio2003}). All the methods require to solve a minimization problem (here it is called ``subproblem"). We solve the subproblems of all methods using \texttt{fminsearch}, a build-in MATLAB solver, with the \texttt{optionset(`\,TolX\,',1e-7,`\,TolFun\,',1e-7)}. The stopping rule of the outer loop in all methods is $||x^{k+1}-x^k||< 10^{-7}$. In each running, the methods take the same random initial point $x^0 \in [-10,10]^n$ in $\mathbb{R}^n$. In the PPMDC, we take the proximal parameter in the regularization term $\alpha_k=0.01$, for all $k\in\mathbb{N}$. In nmBDCA, for all problems, we take the same configuration of parameters $\rho=\zeta =0.5$ and $ \nu _{k}:=\omega\|d^{k}\|^{2}/(k+1)$ as suggested in Example~\ref{eq:omega} with $\omega=0.01$. The initial value $\lambda _{-1}$ are taken as follows: $\lambda _{-1}=3.9$, $\lambda _{-1}=16.0$, $\lambda _{-1}=1.5$, $\lambda _{-1}=5.4$, $\lambda _{-1}=2.8$, $\lambda _{-1}=30.0$ and $\lambda _{-1}=6.6$ for Problem~\ref{prob7}~--~\ref{prob5}, respectively. We make the MATLAB implementation of solvers nmBDCA, DCA and PPMDC as well as the list of initial points and the test problems freely available at the link \href{https://sites.google.com/ufpi.edu.br/souzajco/publications}{https://sites.google.com/ufpi.edu.br/souzajco/publications}.

To proceed the comparison, we perform all methods 100 times starting from the same random initial point. We show the results of nmBDCA, DCA and PPMDC in Tables~\ref{tabnmBDCA}, where $n$ denotes the number of variables of the problem, the columns \texttt{min. k} (resp. \texttt{min. time}), \texttt{max. k} (resp. \texttt{max. time}) and \texttt{med. k} (resp. \texttt{med. time}) present the minimum, maximum and median of iterations (resp. CPU time in seconds) until the stopping rule is satisfied, $\phi(x^k)$ denotes the best value of the objective function for all the solutions found and \texttt{\% opt. value} presents the rate in which the method approximately found the best solution known. The values of $|\phi(x^k)-\phi^*|$ and $||x^{k+1}-x^k||$ for one run of each method are presented in Figures~\ref{fig8}--\ref{fig6}. In these figures, we can clearly see that although nmBDCA sometimes does not decrease DCA (when the red line is above of blue line in the $||\phi(x^k)-\phi^*||$-axis) nmBDCA still outperforms DCA and PPMDC. We run all the methods for two different starting points which for one of them all the methods find the global solution and the other one which DCA and PPMDC stop at a critical point while nmBDCA keeps running until to find the global solution (except for Problem~\ref{prob1} where all the methods always find the global minimum). 

\begin{problem}\label{prob7}\cite[Problem 4.1]{CruzNetoLopesSantosSouza2019}
Let $\phi:\mathbb{R}^2 \to \mathbb{R}$ be a DC function with DC components 
$$
g(x)=\sin\left(\sqrt{|3x_1+|x_1-x_2| +2x_2|}\right)+5(x_1^2+x_2^2)$$
and 
$$h(x)=5(x_1^2+x_2^2).$$
The optimum value is $\phi^*=-1$.
\end{problem}

\begin{problem}\label{prob6}Example~\ref{ExampleAlgorithm} revisited (\cite[Example 3.4]{ARAGON2019})
Let $\phi:\mathbb{R}^2 \to \mathbb{R}$ be a DC function with DC components 
$$
g(x)=-\frac{5}{2}x_1+x_1^2+x_2^2+|x_1|+|x_2|$$
and 
$$h(x)=\frac{1}{2}(x_1^2+x_2^2).$$
The minimum point of $\phi$ is $x^*=(1.5, 0)^{\top}$ and the optimum value is $\phi^*=-1.125$.
\end{problem}

\begin{problem}\label{prob1}\cite[Problem 1]{Joki2017}
Let $\phi:\mathbb{R}^2 \to \mathbb{R}$ be a DC function with DC components 
$$g(x)=\max\{f_{1,1}(x),f_{1,2}(x),f_{1,3}(x)\}+f_{2,1}(x)+f_{2,2}(x)+f_{2,3}(x)$$
and 
$$h(x)=\max\{f_{2,1}(x)+f_{2,2}(x),f_{2,2}(x)+f_{2,3}(x),f_{2,1}(x)+f_{2,3}(x)\},$$
where $f_{1,1}(x)=x_1^4+x_2^2$, $f_{1,2}(x)=(2-x_1)^2+(2-x_2)^2$, $f_{1,3}(x)=2e^{-x_1+x_2}$, $f_{2,1}(x)=x_1^2-2x_1+x_2^2-4x_2+4$, $f_{2,2}(x)=2x_1^2-5x_1+x_2^2-2x_2+4$ and $f_{2,3}(x)=x_1^2+2x_2^2-4x_2+1$. The minimum point of $\phi$ is $x^*=(1,1)^{\top}$ and the optimum value is $\phi^*=2$.
\end{problem}

\begin{problem}\label{prob2}\cite[Problem 2]{Joki2017}
Let $\phi:\mathbb{R}^2 \to \mathbb{R}$ be a DC function with DC components 
$$g(x)=|x_1-1|+200\max\{0,|x_1|-x_2\}$$
and 
$$h(x)=100(|x_1|-x_2).$$
The minimum point of $\phi$ is $x^*=(1,1)^{\top}$ and the optimum value is $\phi^*=0$.
\end{problem}

\begin{problem}\label{prob3}\cite[Problem 3]{Joki2017}
Let $\phi:\mathbb{R}^4 \to \mathbb{R}$ be a DC function with DC components 
\begin{eqnarray*}
g(x)&=& |x_1-1|+200\max\{0,|x_1| - x_2\}+180\max\{0,|x_3| - x_4\}+|x_3-1|\\
&&+10.1(|x_2-1|+|x_4-1|)+4.95|x_2+x_4-2|
\end{eqnarray*}
and 
$$h(x)=100(|x_1| - x_2)+90(|x_3| - x_4)+4.95|x_2 - x_4|.$$
The minimum point of $\phi$ is $x^*=(1,1,1,1)^{\top}$ and the optimum value is $\phi^*=0$.
\end{problem}

\begin{problem}\label{prob4}\cite[Problem 7]{Joki2017}
Let $\phi:\mathbb{R}^2 \to \mathbb{R}$ be a DC function with DC components 
\begin{eqnarray*}
g(x)&=& |x_1-1|+200\max\{0,|x_1| - x_2\}\\
&&+10\max\{x_1^2+x_2^2+|x_2|,x_1+x_1^2+x_2^2+|x_2|-0.5,|x_1-x_2|+|x_2|-1,x_1+x_1^2+x_2^2\}
\end{eqnarray*}
and 
$$h(x)=100(|x_1| - x_2)+10(x_1^2+x_2^2+|x_2|).$$
The minimum point of $\phi$ is $x^*=(0.5,0.5)^{\top}$ and the optimum value is $\phi^*=0.5$.
\end{problem}

\begin{problem}\label{prob5}\cite[Problem 8]{Joki2017}
Let $\phi:\mathbb{R}^3 \to \mathbb{R}$ be a DC function with DC components 
\begin{eqnarray*}
g(x)&=& 9-8x_1-6x_2-4x_3+2|x_1|+2|x_2|+2|x_3|\\
&&+4x_1^2+2x_2^2+2x_3^2 + 10\max\{0,x_1+x_2+2x_3-3,-x_1,-x_2,-x_3\}
\end{eqnarray*}
and 
$$h(x)=|x_1 - x_2|+|x_1 - x_3|.$$
The minimum point of $\phi$ is $x^*=(0.75,1.25,0.25)^{\top}$ and the optimum value is $\phi^*=3.5$.
\end{problem}

As we can see in Table~\ref{tabnmBDCA}, nmBDCA outperforms DCA and PPMDC in the quality of the solution found in all the test problems. In three of the seven test problems, nmBDCA finds the global minimum in all run while the other methods do the same only in one test problem. In terms of performance (number of iterates and CPU time), nmBDCA also outperforms DCA and PPMDC. This is clear when we see the lines for Problem~\ref{prob7} and \ref{prob1} in Table~\ref{tabnmBDCA} where all the methods have the same rate of finding the global solution. In these cases, nmBDCA is 16 times and 3 times more efficient in terms of the median of iterates and CPU time than DCA and PPMDC for Problem~\ref{prob7} and \ref{prob1}, respectively. In Problem~\ref{prob6} and \ref{prob4}, nmBDCA has a better performance in terms of median of iteration and CPU time compare with DCA and PPMDC. Note that in these problems nmBDCA finds the global solution in the rate of $100\%$ and $56\%$, respectively, against $63\%$ and $30\%$ for DCA and PPMDC, respectively. In Problem~\ref{prob2}, \ref{prob3} and \ref{prob5}, nmBDCA needs more iterates and CPU time to obtain a solution than DCA and PPMDC. This is why in these problems, DCA and PPMDC stop in few iterates but they find the global solution only in the rate of $49\%$, $17\%$ and $18\%$ for DCA and $48\%$, $18\%$ and $18\%$ for PPMDC while nmBDCA finds the global solution in the rate of $100\%$, $31\%$ and $67\%$, respectively. It is worth to note that, in these problems, nmBDCA underperforms the other methods meanly because the sequence $\nu _{k}:=\omega\|d^{k}\|^{2}/(k+1)$ (with $\omega=0.01$) enables a large increasing of $\phi(x^{k+1})$ compared to $\phi(y^{k})$ in the first steps. Our simulations show that if we consider small values of $\omega$ in these problems, then the performance of nmBDCA is quite similar to DCA and PPMDC but with better rates of finding global minimum.

Summing up, our numerical experiments show that nmBDCA has a good performance compared with DCA and PPMDC such as its monotone version BDCA. The freedom to a possible growth given by the parameter $\nu_k$ does not affect the efficiency of the method. This is an important feature which increases the range of application of boosted DC algorithms.

\begin{sidewaystable}
\captionsetup{position=top} 
\caption{Summary of the numerical results of nmBDCA, DCA and PPMDC for 100 run.}
\label{tabnmBDCA} 
\begin{tabular}{|c|c|c|c|c|c|c|c|c|c|c|c|}
\hline
Problem & $n$ &  min. $k$ & max. $k$ & med. $k$ & min. time & max. time & med. time & $\phi(x^k)$ & \% opt. value \\ \hline
\multicolumn{10}{|c|}{nmBDCA} \\ \hline
\ref{prob7} & 2 & 3 & 83 & 46.28 & 0.0026649 & 0.0568138 & 0.030465146 & -0.999999999999859 & 97 \\ \hline%
\ref{prob6} & 2 & 7 & 16 & 10.82 & 0.005585 & 0.0136846 & 0.008780246 & -1.125000000000000 & 100 \\ \hline%
\ref{prob1} & 2  &6& 15 & 9.81 & 0.0082839 & 0.0211974 & 0.012511895 & 2.000000000000004 & 100 \\ \hline%
\ref{prob2} & 2& 3 & 6 & 4.02 & 0.0022413 & 0.0081915 & 0.004182691 & 3.960432204408448e-09 & 100 \\ \hline%
\ref{prob3} & 4& 4 & 13 & 7.28 & 0.0078985 & 0.0395438 & 0.021201167 & 4.348665016973285e-08 & 31 \\ \hline%
\ref{prob4} & 2 & 3 & 21 & 8.8 & 0.0022562 & 0.0213988 & 0.009153369 & 0.500000002033778 & 56 \\ \hline%
\ref{prob5} & 3 & 3 & 8 & 6.41 & 0.009874 & 0.0205232 & 0.013033188 & 3.499999999999999 & 67 \\ \hline%
\multicolumn{10}{|c|}{DCA} \\ \hline
\ref{prob7} &2&2&1072&749.5599999&0.0023714 & 0.8463685 & 0.499823531 & -0.999999999996628 & 97 \\ \hline%
\ref{prob6} & 2 & 2 & 27 & 17.19 & 0.0011749 & 0.0150196 & 0.008793059 & -1.125000000000000 & 63 \\ \hline%
\ref{prob1}&2&23&35 & 30.5599999 & 0.0268213 & 0.0439881 & 0.036820655 & 2.000000000000052 & 100 \\ \hline%
\ref{prob2} & 2 & 2 & 5 & 2.15 & 0.0012365 & 0.0040823 & 0.001967687 & 2.274440946692380e-08 & 49 \\ \hline%
\ref{prob3} & 4 & 3 & 12 & 6.59 & 0.0068355 & 0.043065 & 0.018513736 & 9.012336862346258e-08 & 17 \\ \hline%
\ref{prob4} &2&2& 359 & 58.4399999 & 0.0013686 & 0.3862458 & 0.062961467 & 0.500000013656637 & 30 \\ \hline%
\ref{prob5} & 3 & 2 & 6 & 2.54 & 0.0035983 & 0.0169846 & 0.006560558 & 3.500000000000000 & 18 \\ \hline%
\multicolumn{10}{|c|}{PPMDC} \\ \hline
\ref{prob7} &2&2&1067&751.4299999&0.0025088 & 0.8645138 & 0.510410394 & -0.999999999997184 & 97 \\ \hline%
\ref{prob6} & 2 & 2 & 27 & 17.51 & 0.0012693 & 0.0180685 & 0.009237261 & -1.125000000000000 & 63 \\ \hline%
\ref{prob1} & 2 & 23 & 35 & 30.53 & 0.027967 & 0.0532839 & 0.038431424 & 2.000002000000055 & 100 \\ \hline%
\ref{prob2} & 2 & 2 & 4 & 2.09 & 0.0012844 & 0.0036622 & 0.001996839 & 1.526064075108025e-08 & 48 \\ \hline%
\ref{prob3} & 4 & 3 & 14 & 6.5 & 0.0064865 & 0.0394513 & 0.019424618 & 6.207930470791823e-08 & 18 \\ \hline%
\ref{prob4} & 2 & 2 & 359 & 58.46 & 0.0014268 & 0.4039013 & 0.065365052 & 0.500000011319287 & 30 \\ \hline%
\ref{prob5} & 3 & 2 & 29 & 5.3 & 0.0067166 &0.0536524 & 0.0127099 & 3.500000000000002 & 18 \\ \hline%
\end{tabular}
\end{sidewaystable}

\begin{figure}[h!]
\centering
\subfloat{\label{fig8:a}\includegraphics[width=0.25\linewidth]{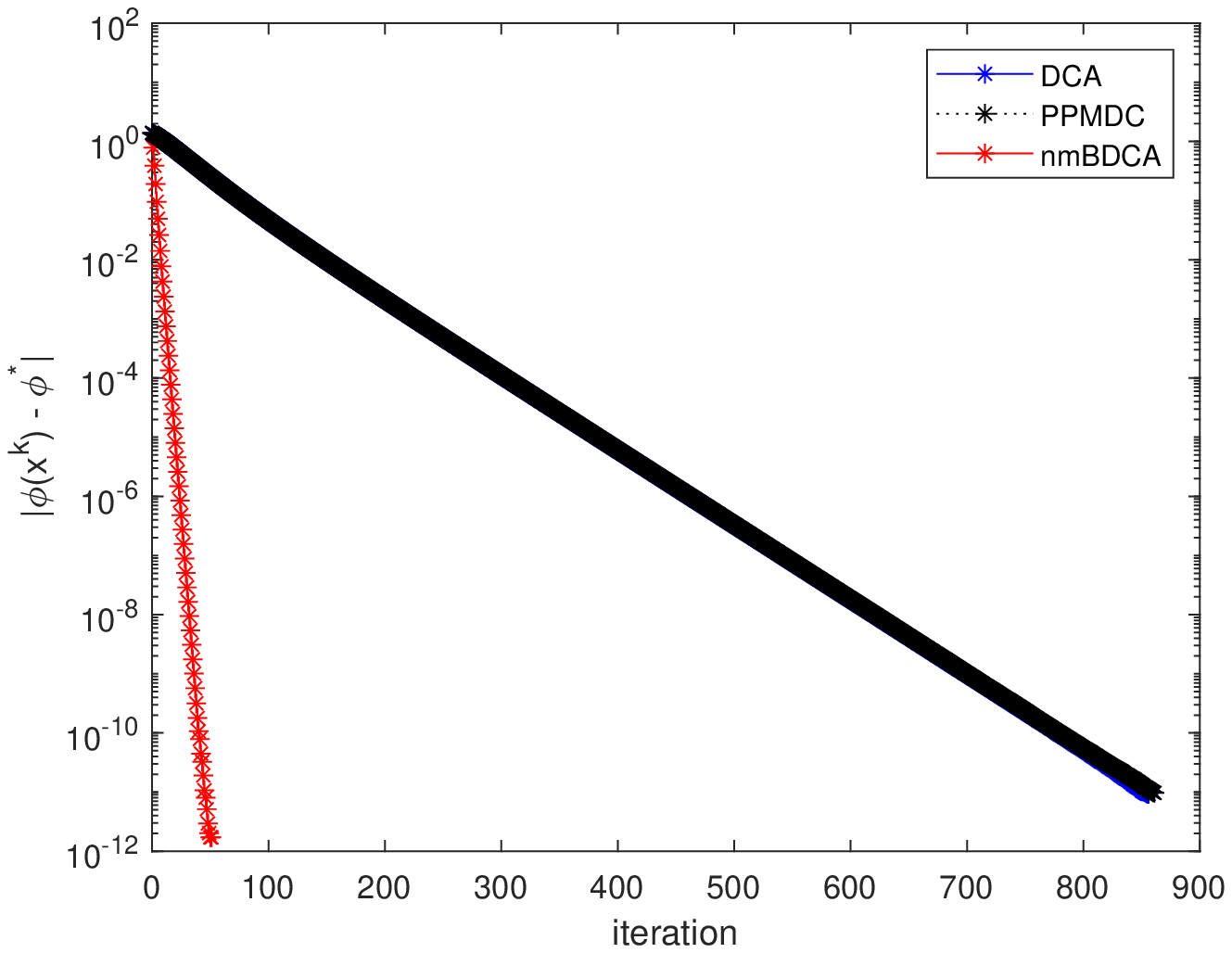}}
\subfloat{\label{fig8:b}\includegraphics[width=0.25\linewidth]{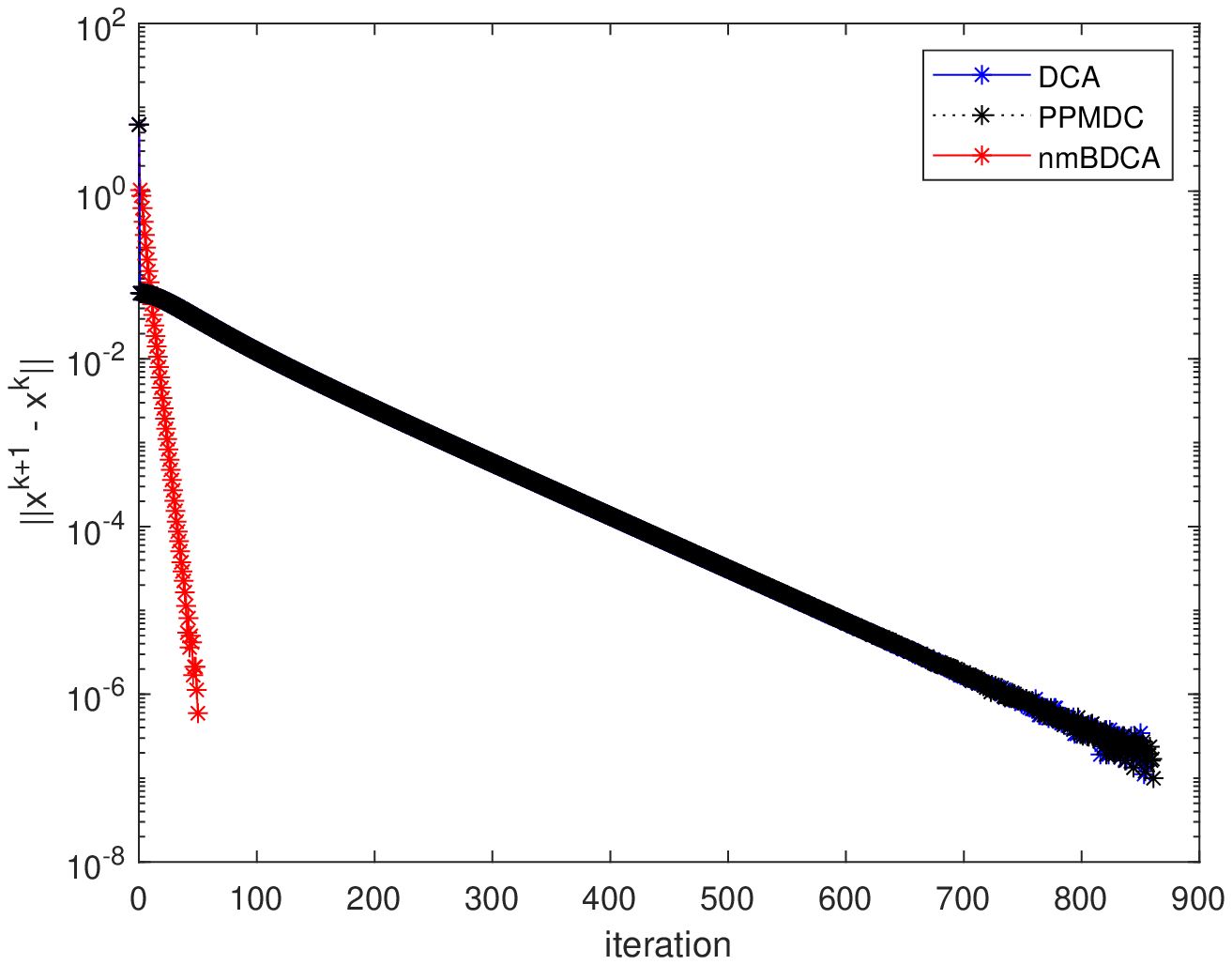}}
\subfloat{\label{fig8:c}\includegraphics[width=0.25\textwidth]{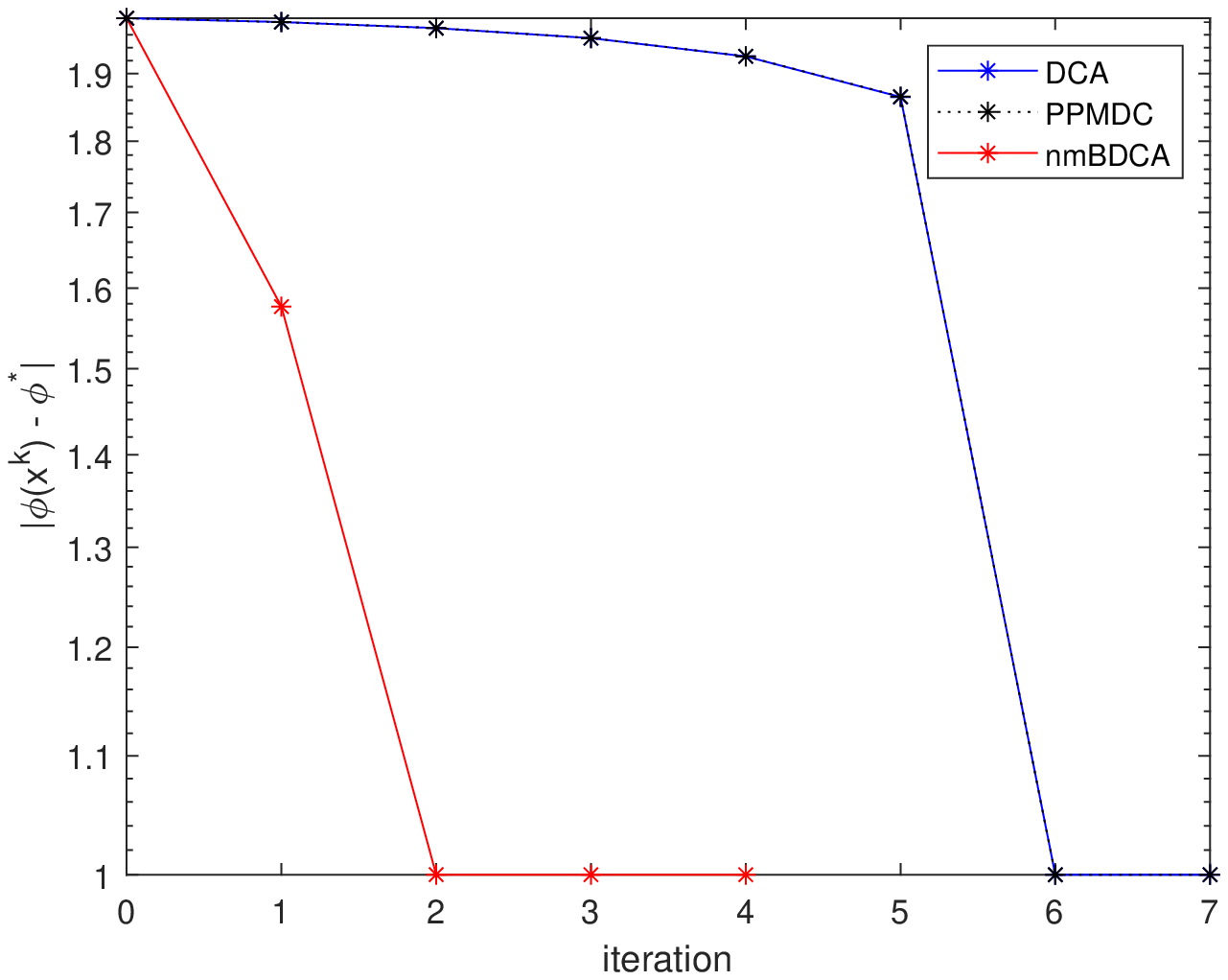}}
\subfloat{\label{fig8:d}\includegraphics[width=0.25\textwidth]{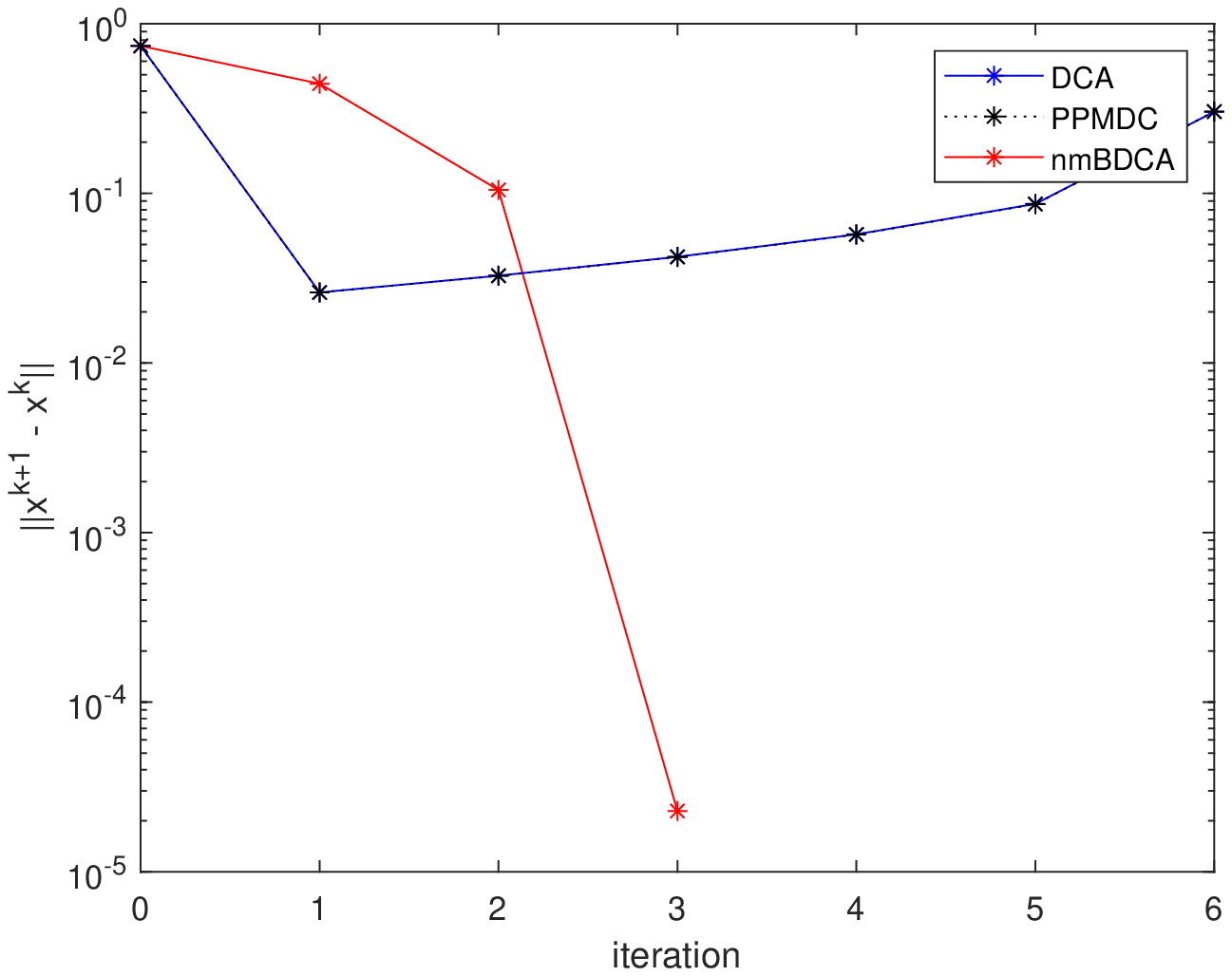}}%
\caption{Value of $||\phi(x^k)-\phi^*||$ and $||x^{k+1}-x^k||$ (using log. scale) for Problem~\ref{prob7}.}
\label{fig8}
\end{figure}

\begin{figure}[h!]
\centering
\subfloat{\label{fig7:a}\includegraphics[width=0.25\linewidth]{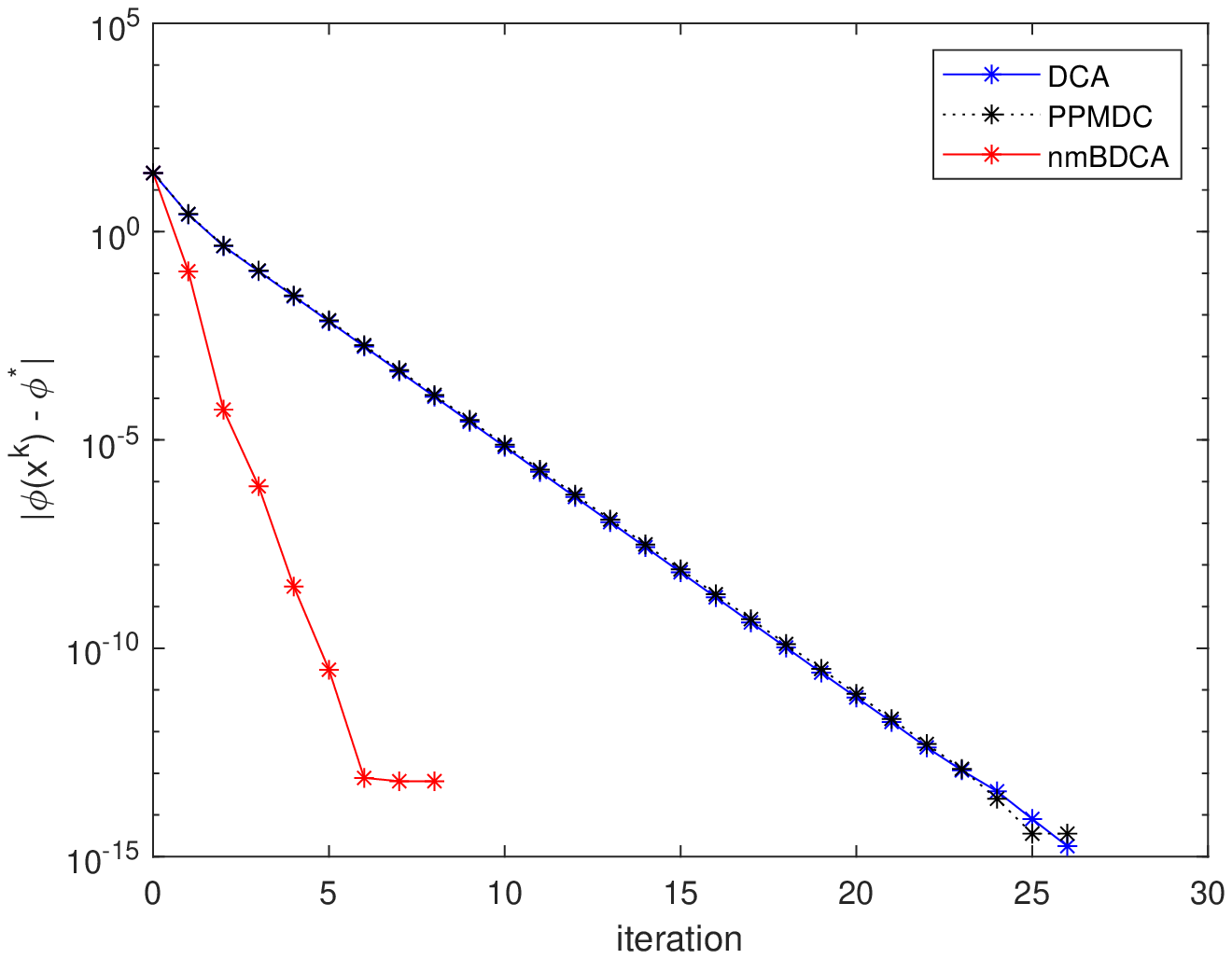}}
\subfloat{\label{fig7:b}\includegraphics[width=0.25\linewidth]{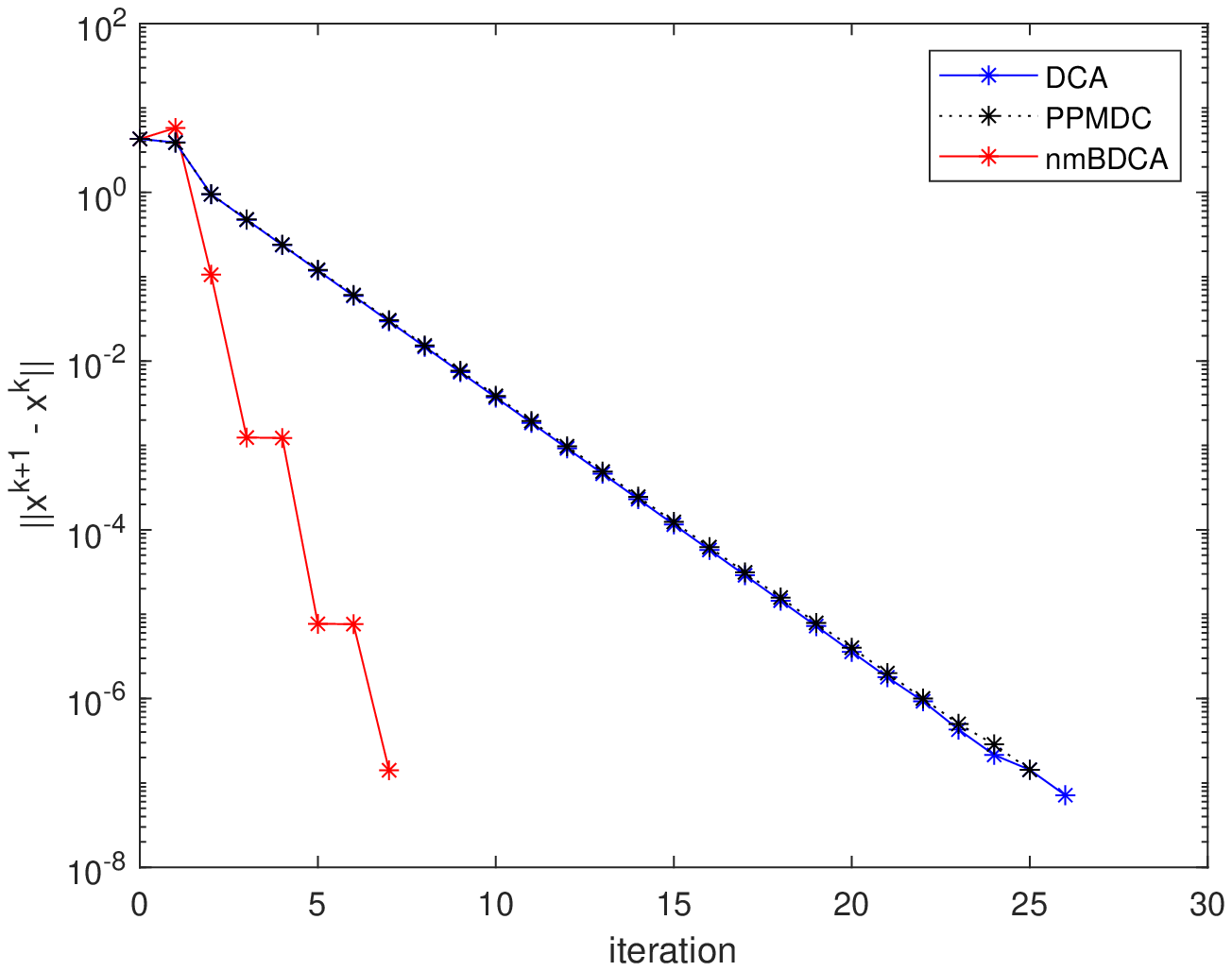}}
\subfloat{\label{fig7:c}\includegraphics[width=0.25\textwidth]{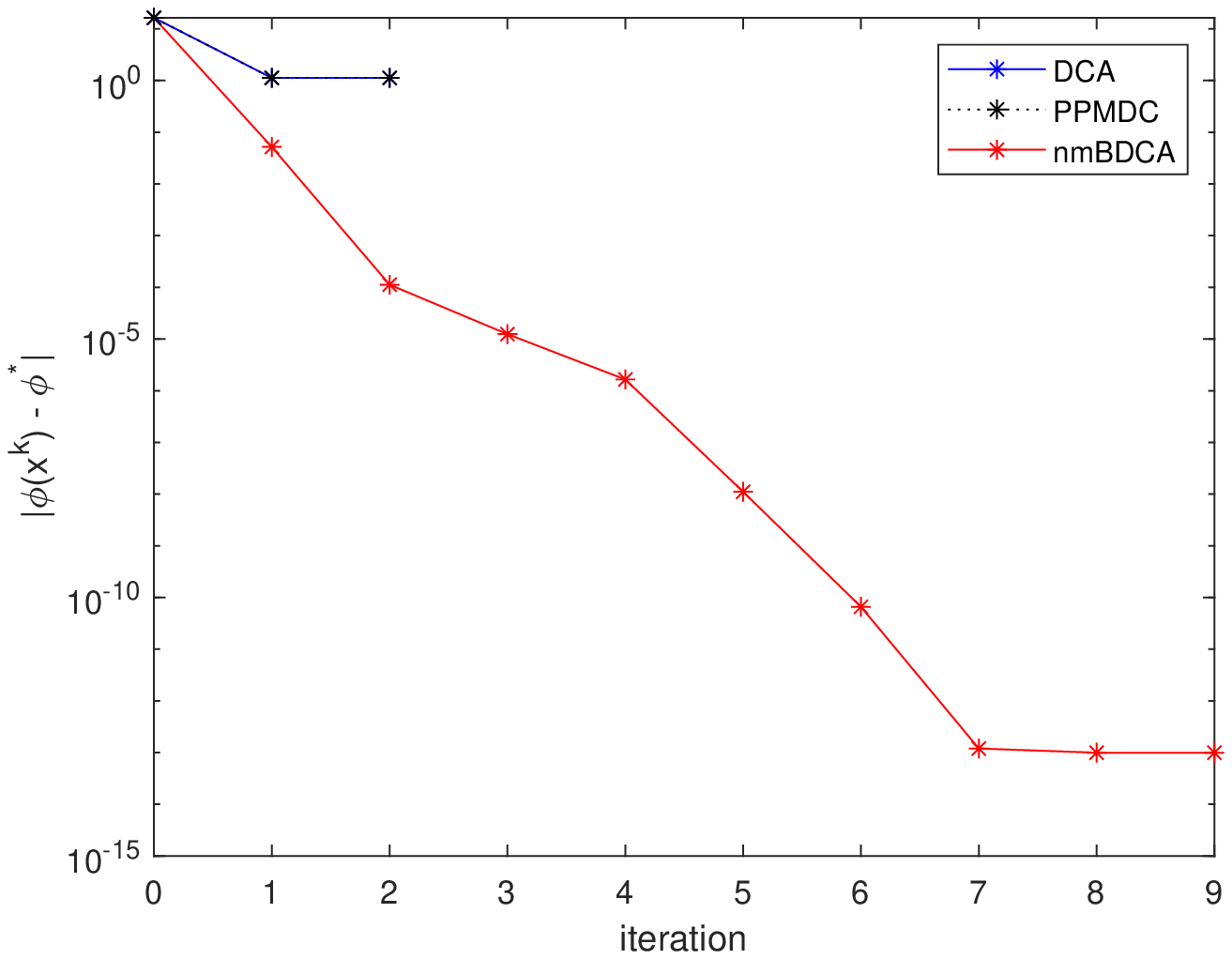}}
\subfloat{\label{fig7:d}\includegraphics[width=0.25\textwidth]{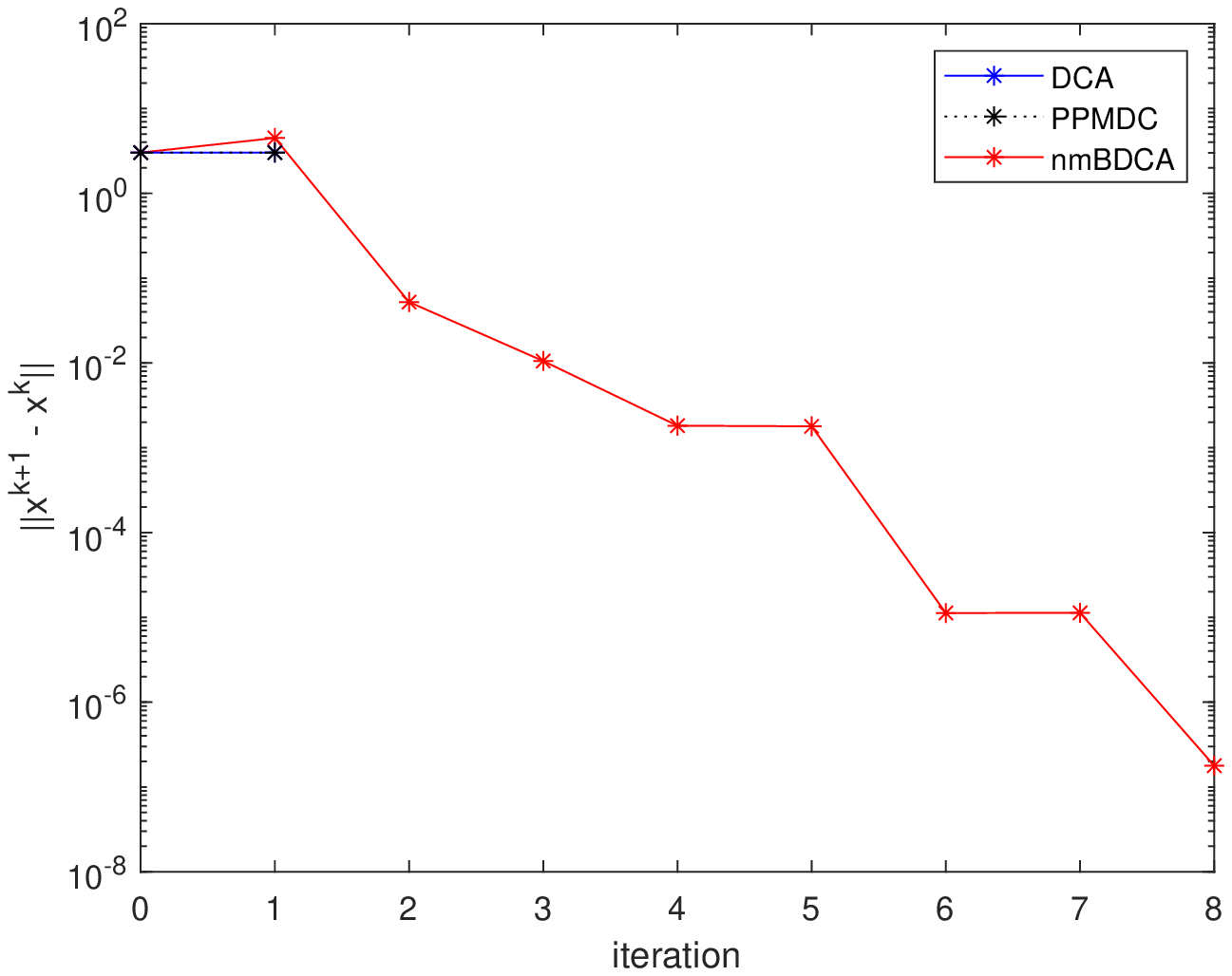}}%
\caption{Value of $||\phi(x^k)-\phi^*||$ and $||x^{k+1}-x^k||$ (using log. scale) for Problem~\ref{prob6}.}
\label{fig7}
\end{figure}

\begin{figure}[h!]
\centering
\subfloat{\label{fig2:a}\includegraphics[width=0.3\textwidth]{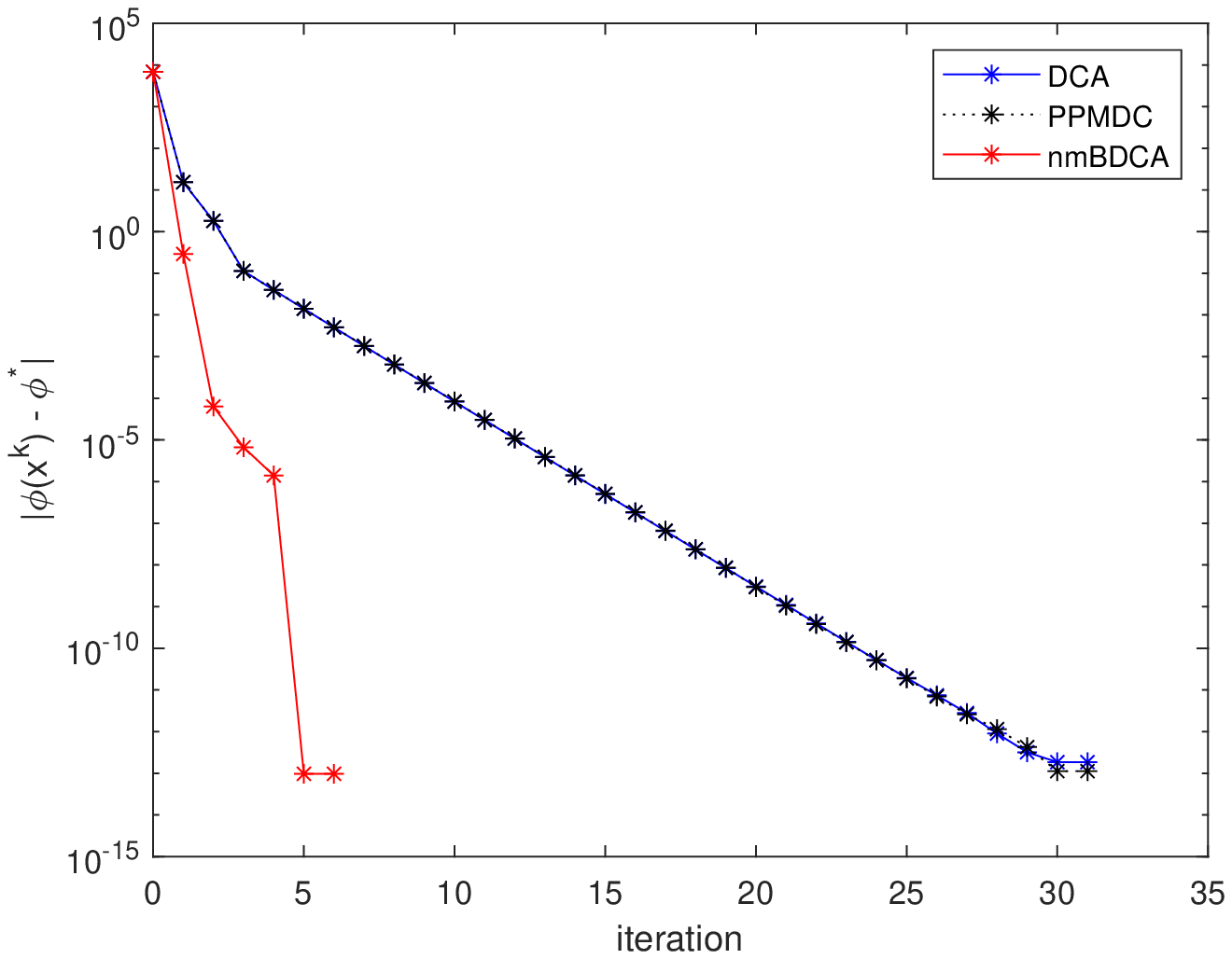}}
\subfloat{\label{fig2:b}\includegraphics[width=0.3\textwidth]{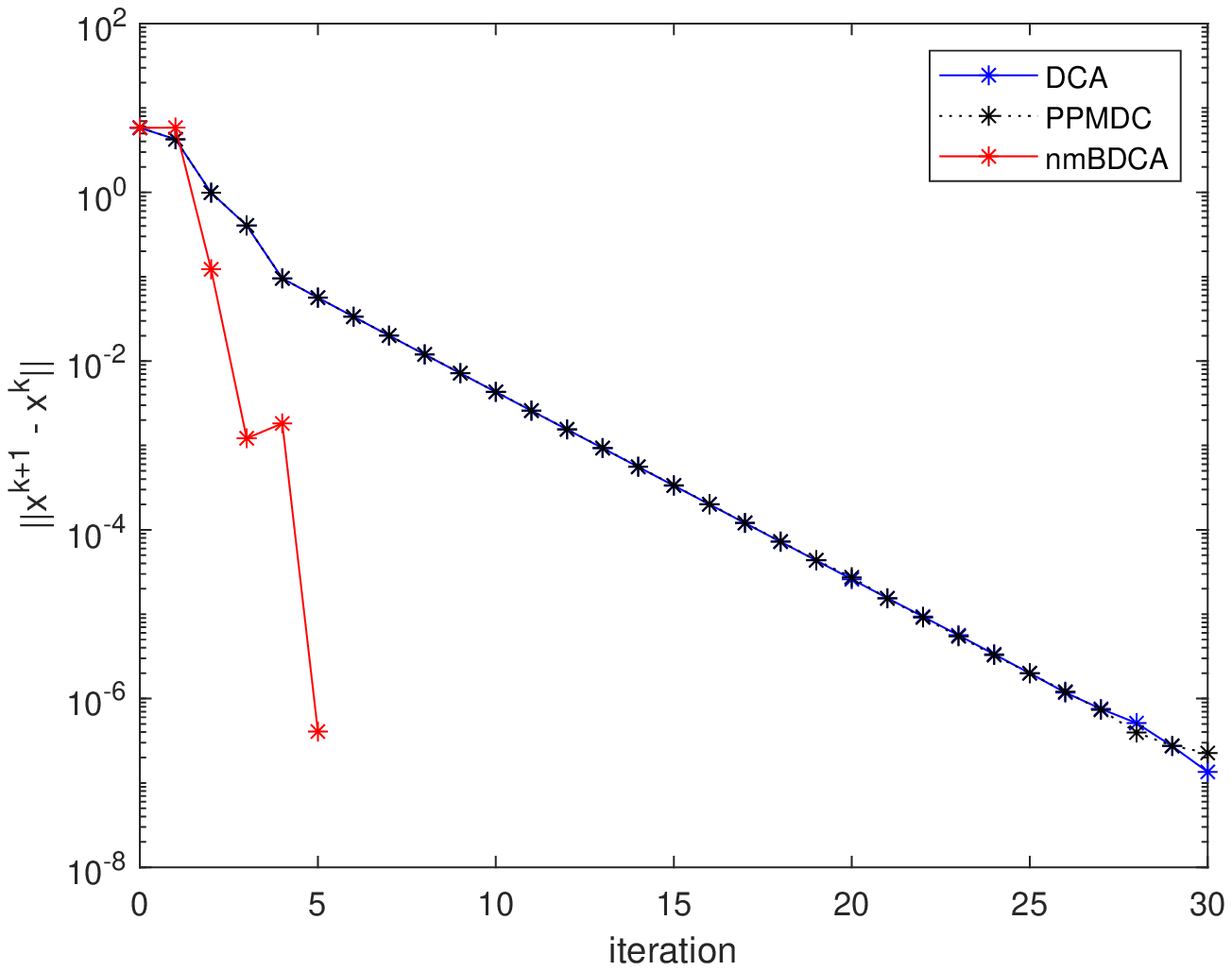}}
\caption{Value of $||\phi(x^k)-\phi^*||$ and $||x^{k+1}-x^k||$ (using log. scale) for Problem~\ref{prob1}.}
\label{fig2}
\end{figure}

\begin{figure}[h!]
\centering
\subfloat{\label{fig3:a}\includegraphics[width=0.25\linewidth]{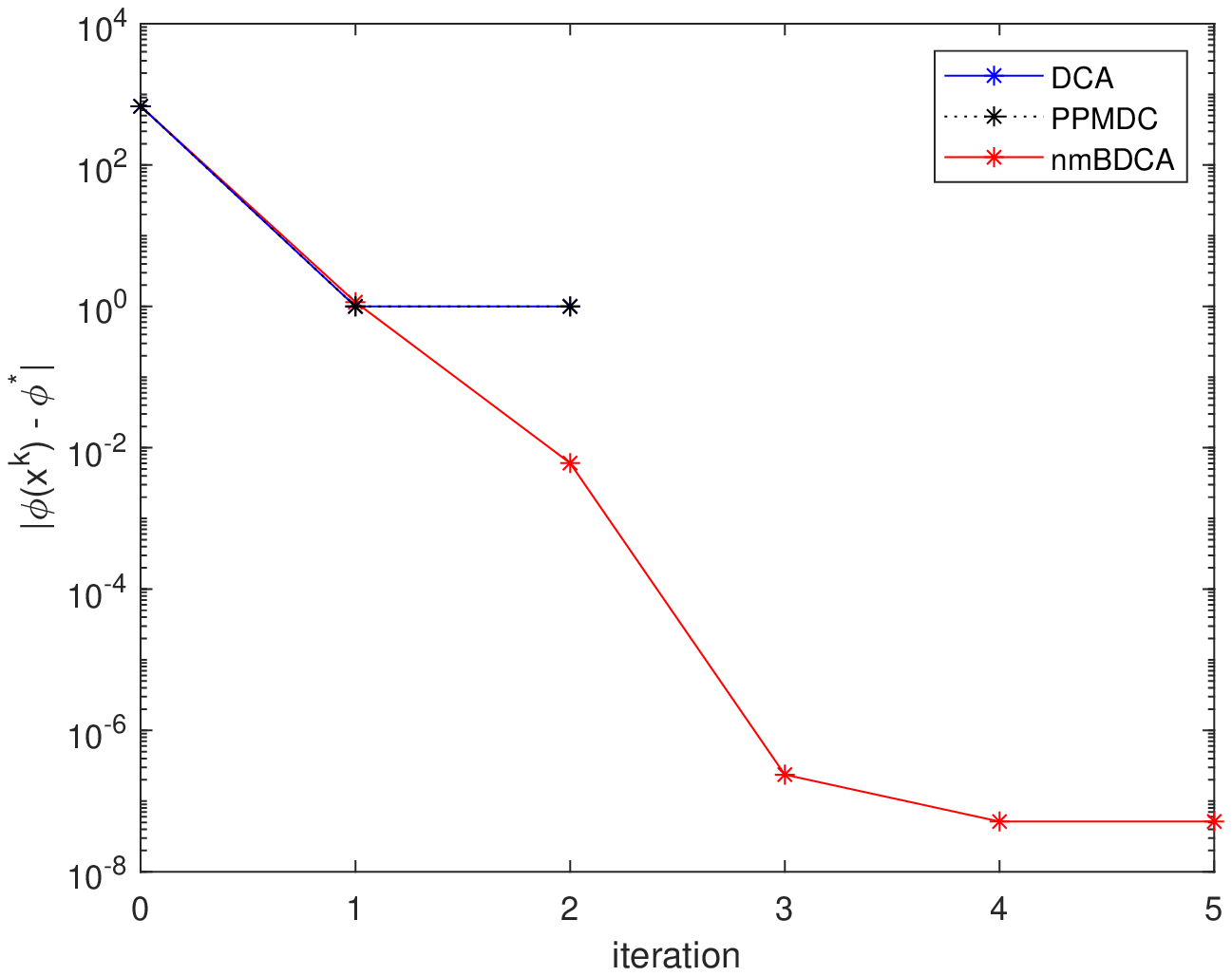}}
\subfloat{\label{fig3:b}\includegraphics[width=0.25\linewidth]{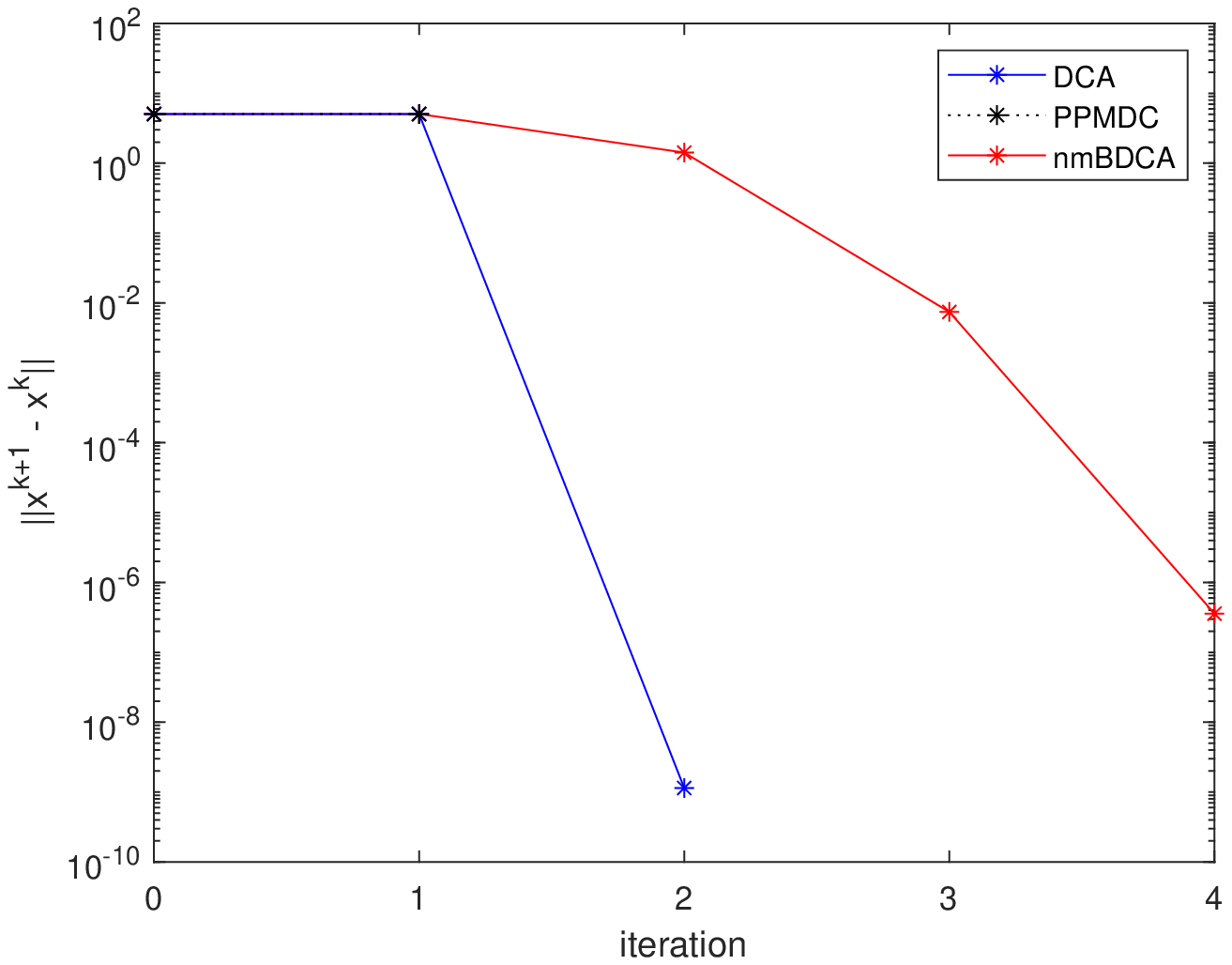}}
\subfloat{\label{fig3:c}\includegraphics[width=0.25\textwidth]{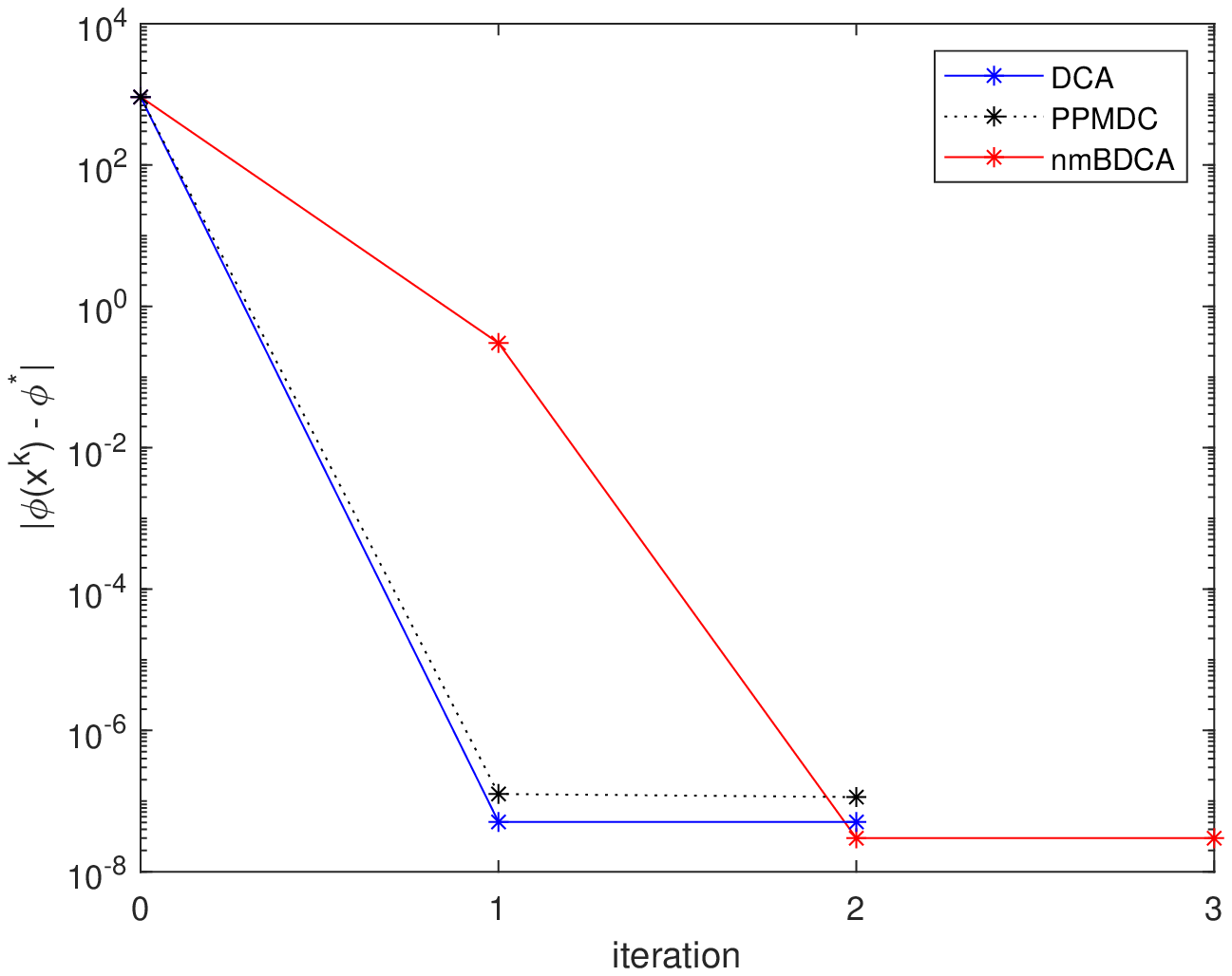}}
\subfloat{\label{fig3:d}\includegraphics[width=0.25\textwidth]{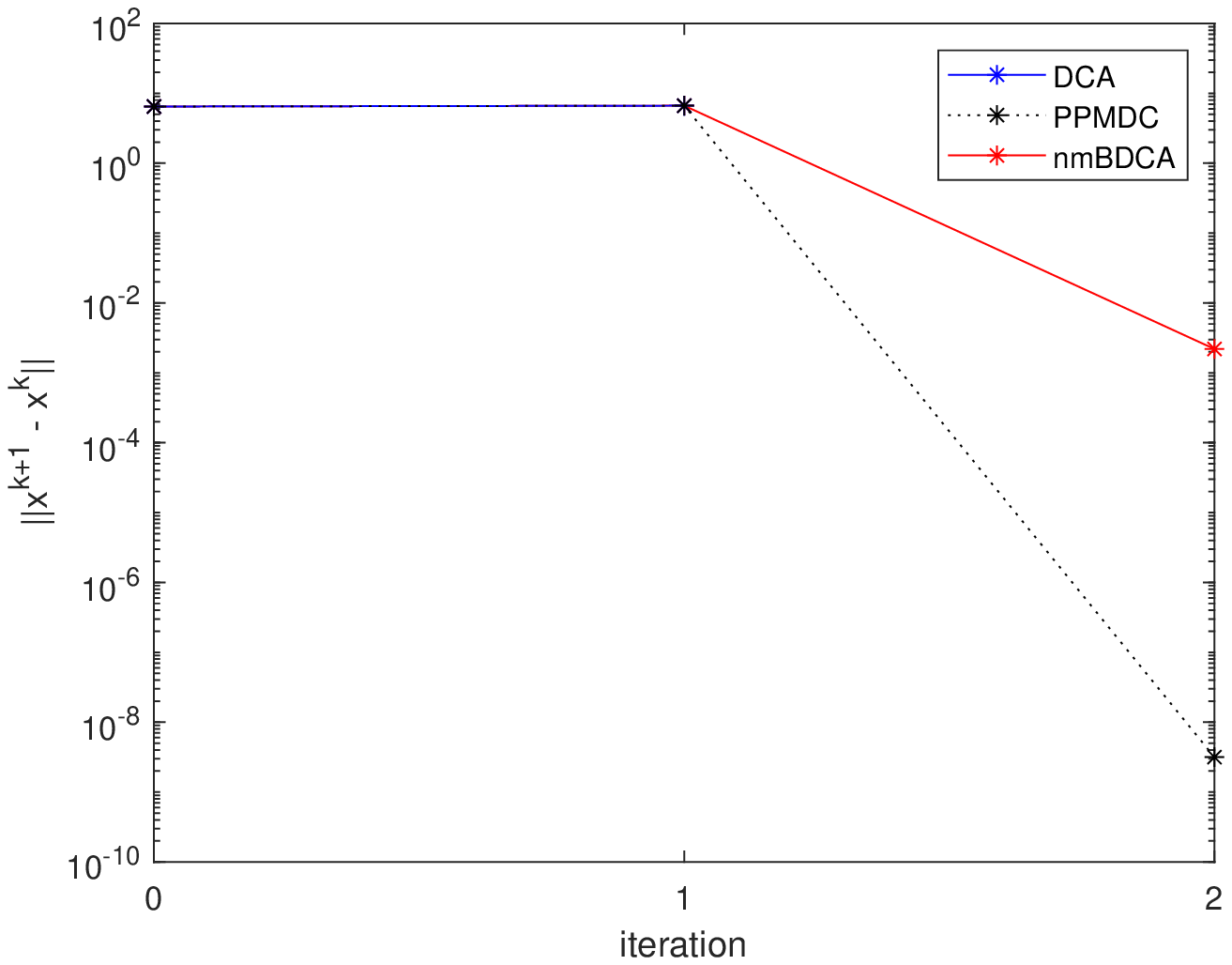}}%
\caption{Value of $||\phi(x^k)-\phi^*||$ and $||x^{k+1}-x^k||$ (using log. scale) for Problem~\ref{prob2}.}
\label{fig3}
\end{figure}

\begin{figure}[h!]
\centering
\subfloat{\label{fig4:a}\includegraphics[width=0.25\linewidth]{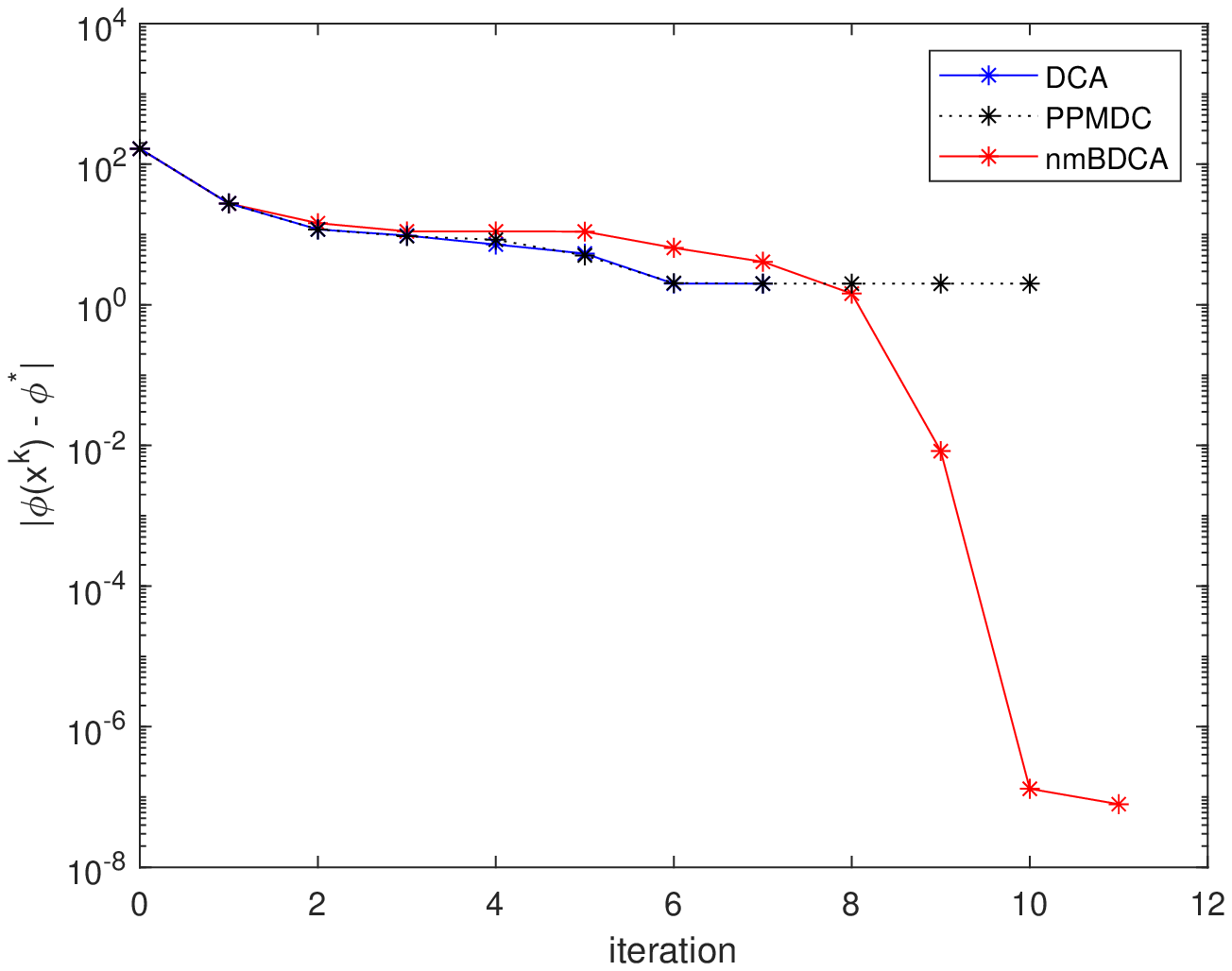}}
\subfloat{\label{fig4:b}\includegraphics[width=0.25\linewidth]{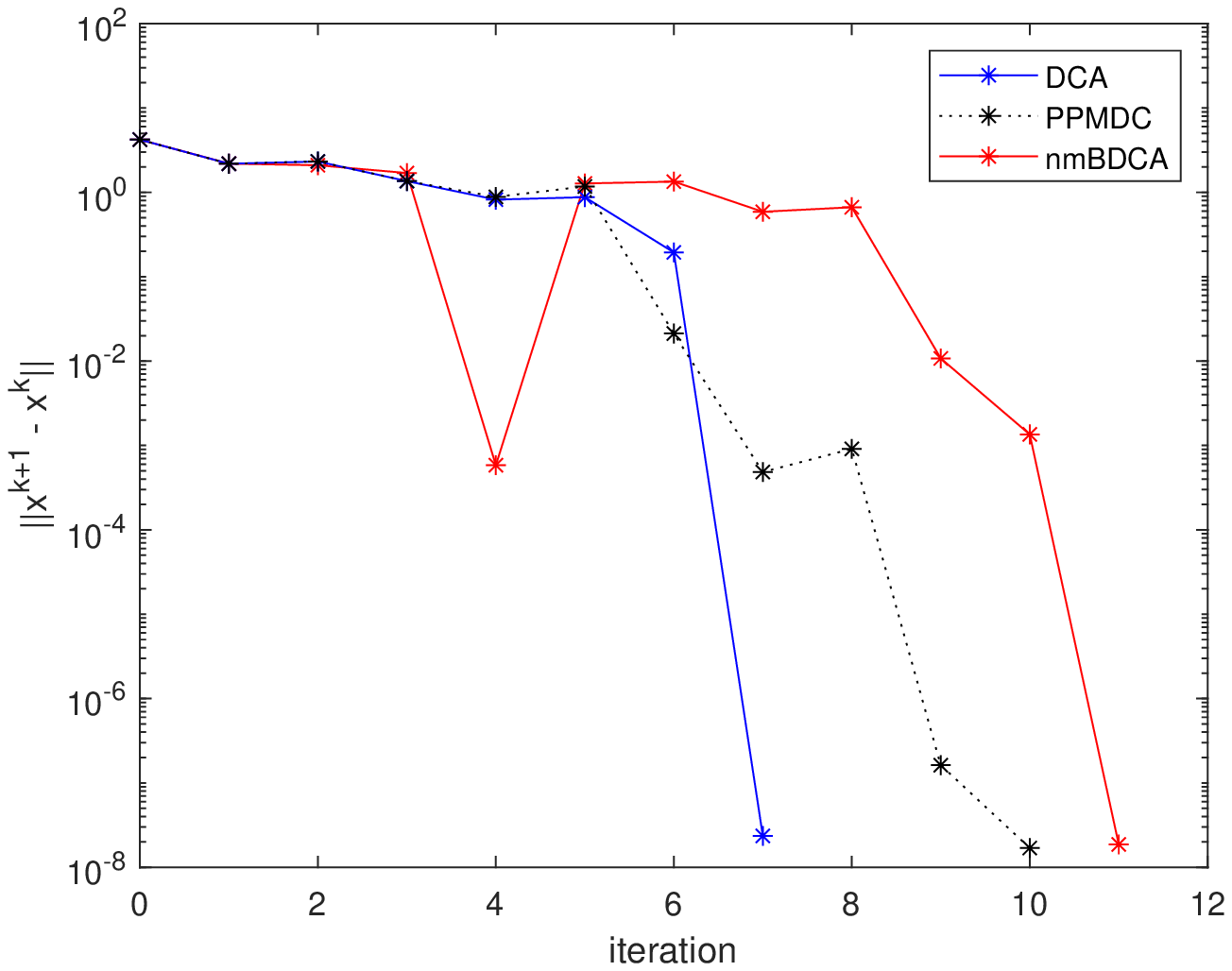}}
\subfloat{\label{fig4:c}\includegraphics[width=0.25\textwidth]{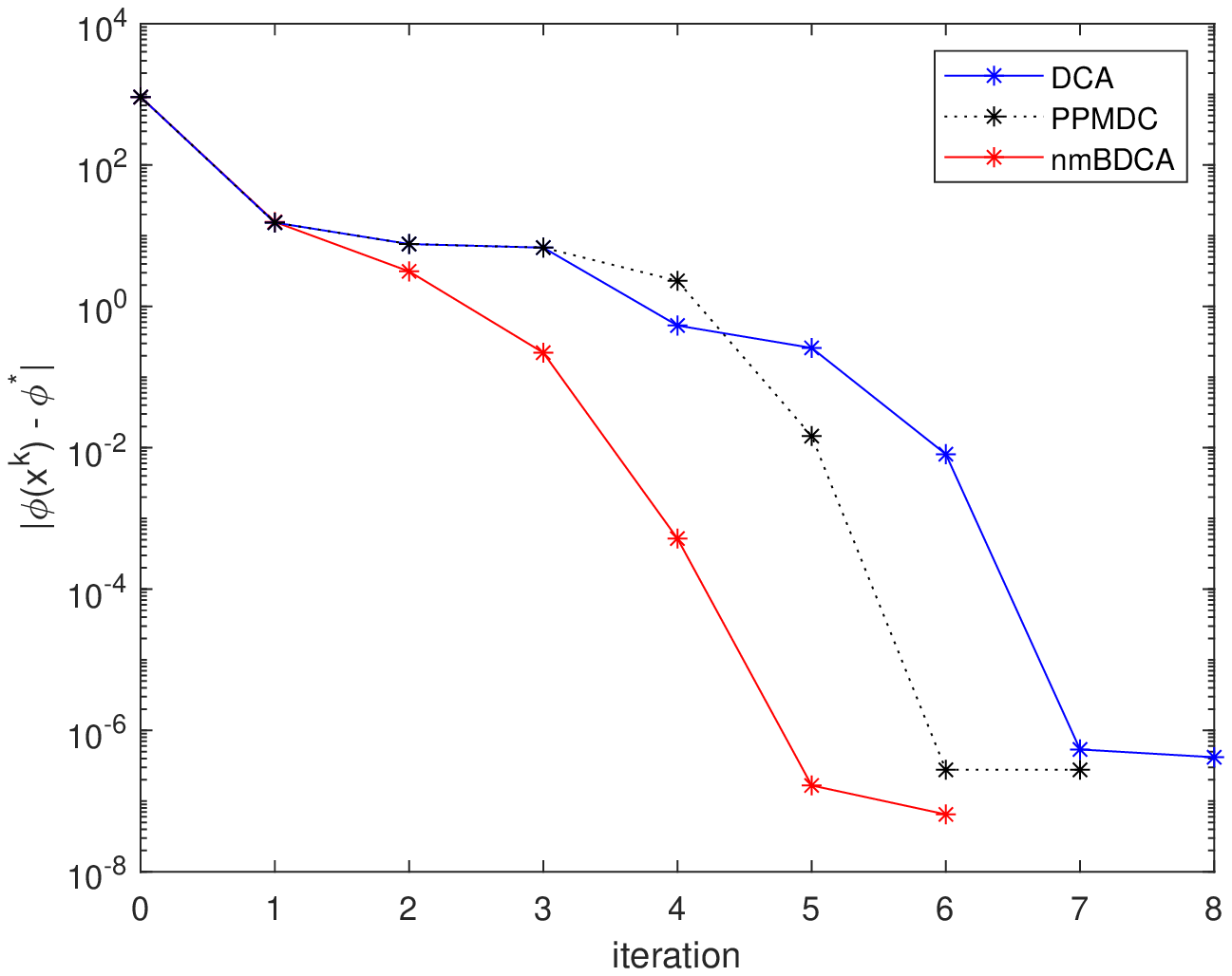}}
\subfloat{\label{fig4:d}\includegraphics[width=0.25\textwidth]{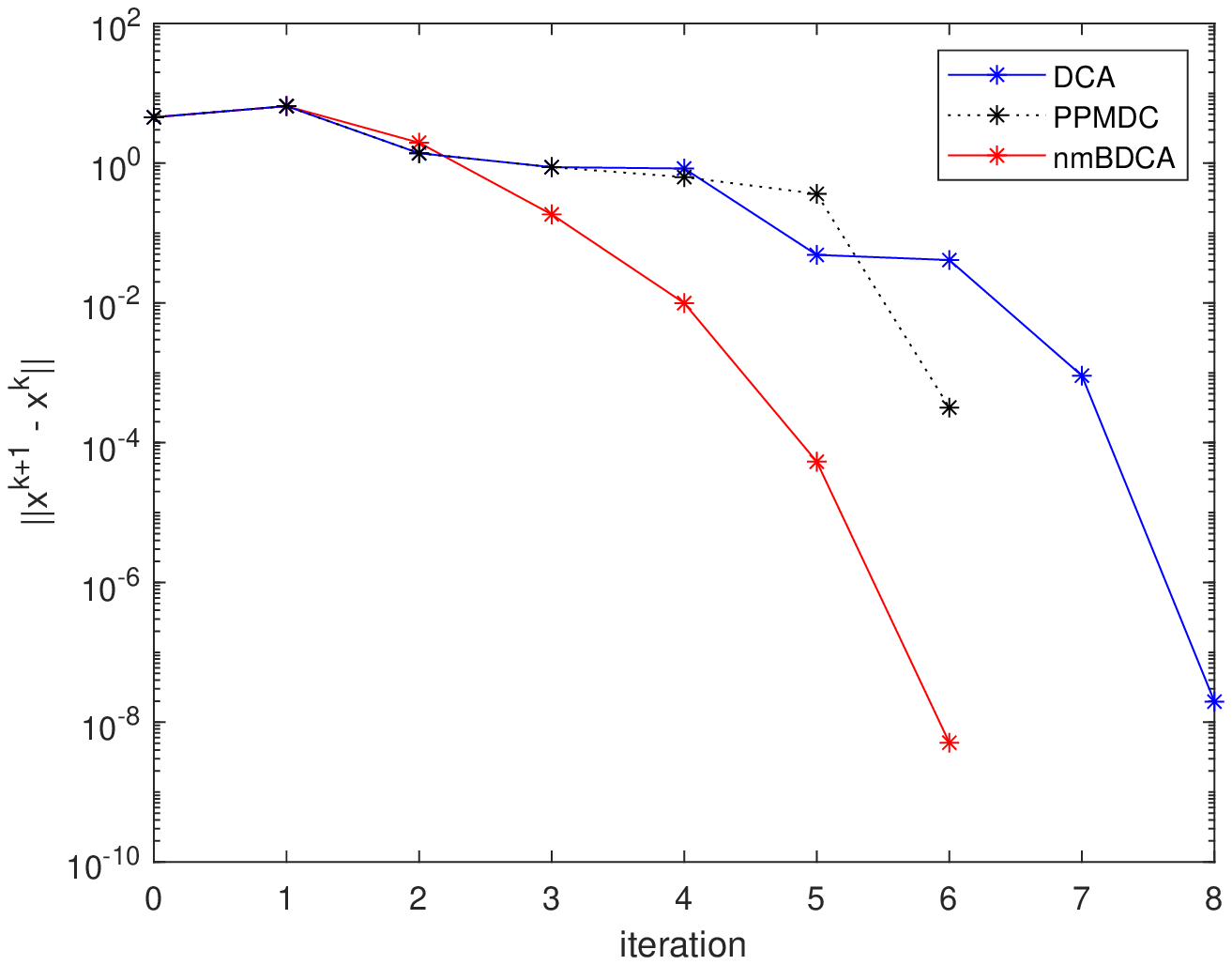}}%
\caption{Value of $||\phi(x^k)-\phi^*||$ and $||x^{k+1}-x^k||$ (using log. scale) for Problem~\ref{prob3}.}
\label{fig4}
\end{figure}

\begin{figure}[h!]
\centering
\subfloat{\label{fig5:a}\includegraphics[width=0.25\linewidth]{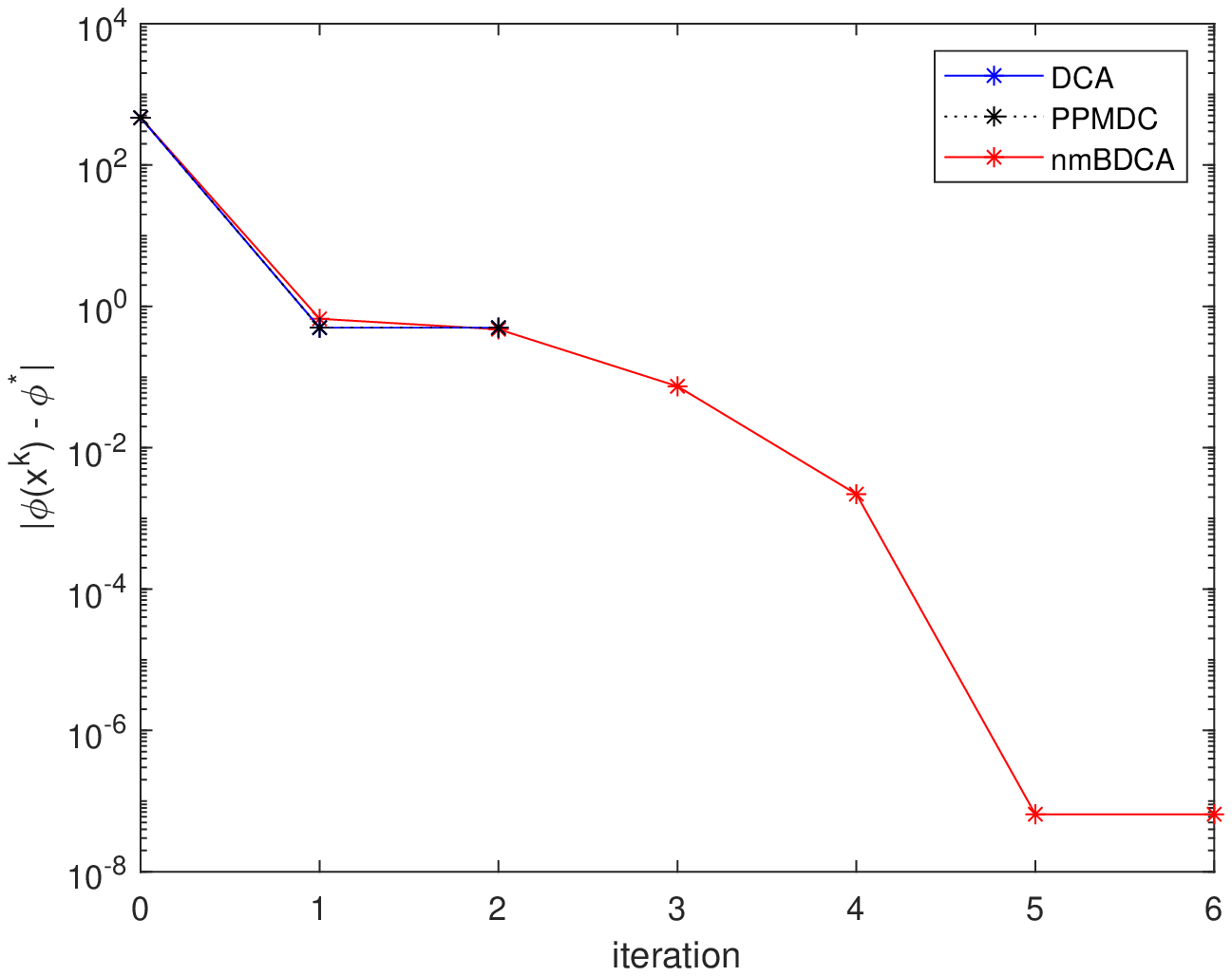}}
\subfloat{\label{fig5:b}\includegraphics[width=0.25\linewidth]{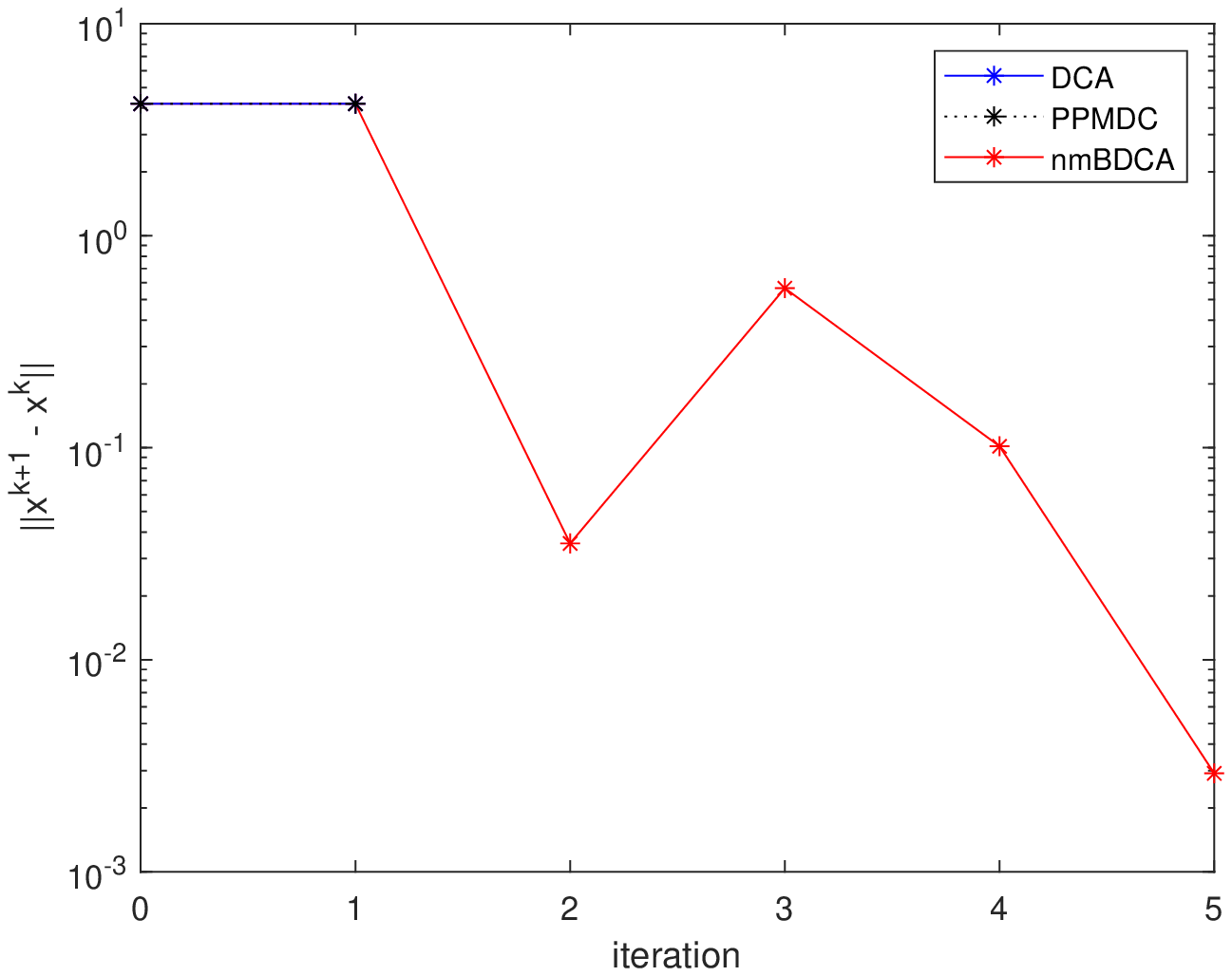}}
\subfloat{\label{fig5:c}\includegraphics[width=0.25\textwidth]{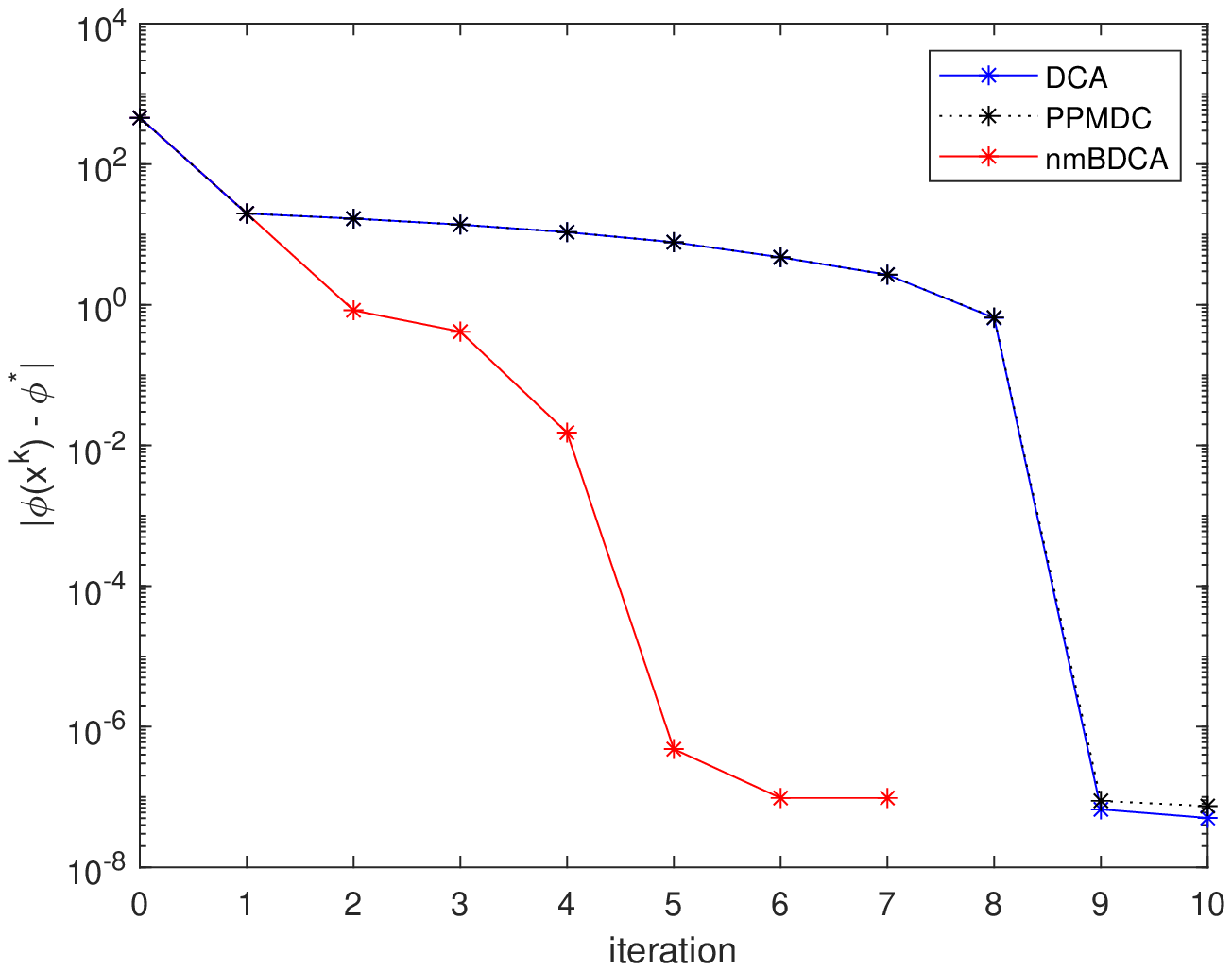}}
\subfloat{\label{fig5:d}\includegraphics[width=0.25\textwidth]{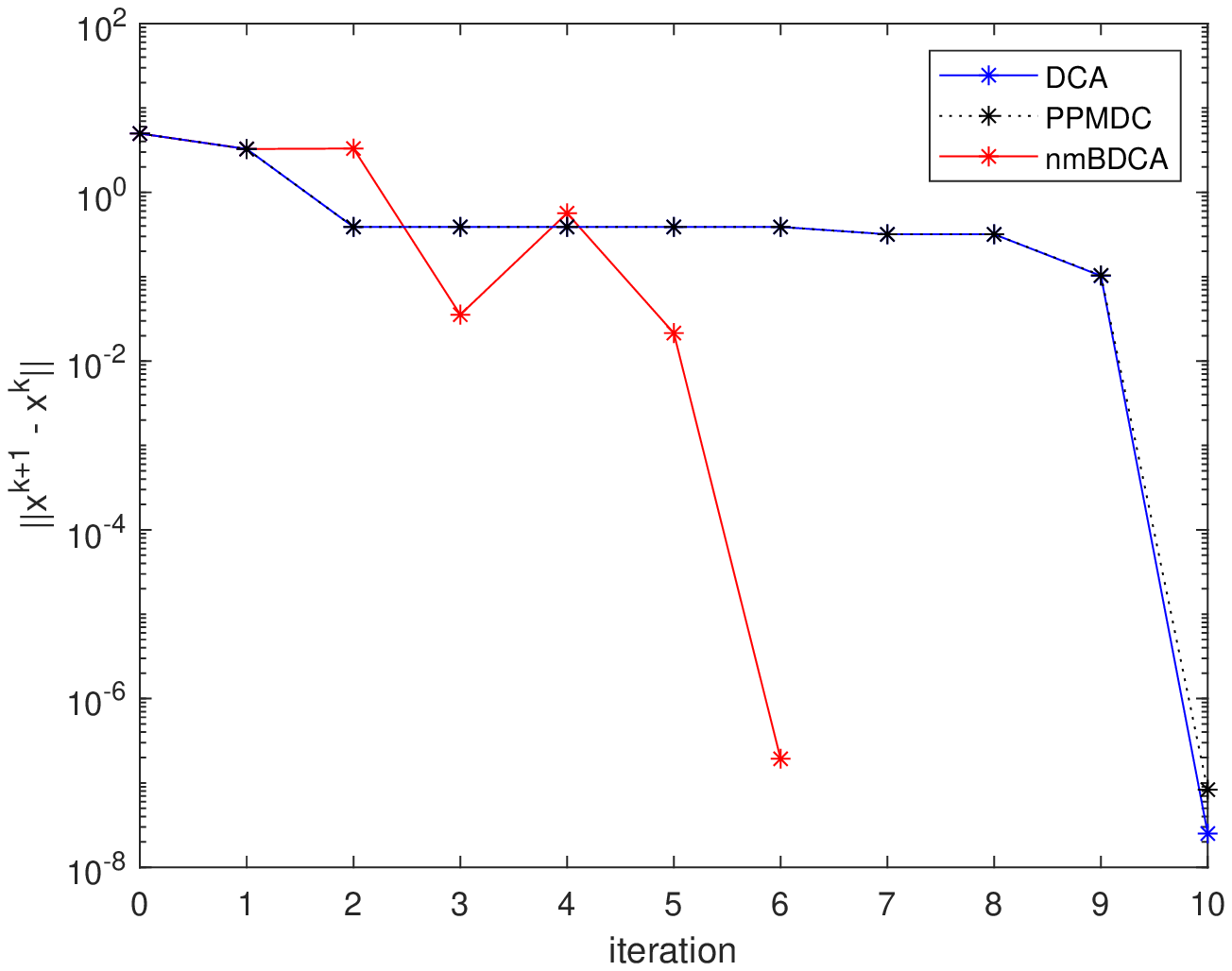}}%
\caption{Value of $||\phi(x^k)-\phi^*||$ and $||x^{k+1}-x^k||$ (using log. scale) for Problem~\ref{prob4}.}
\label{fig5}
\end{figure}

\begin{figure}[h!]
\centering
\subfloat{\label{fig6:a}\includegraphics[width=0.25\linewidth]{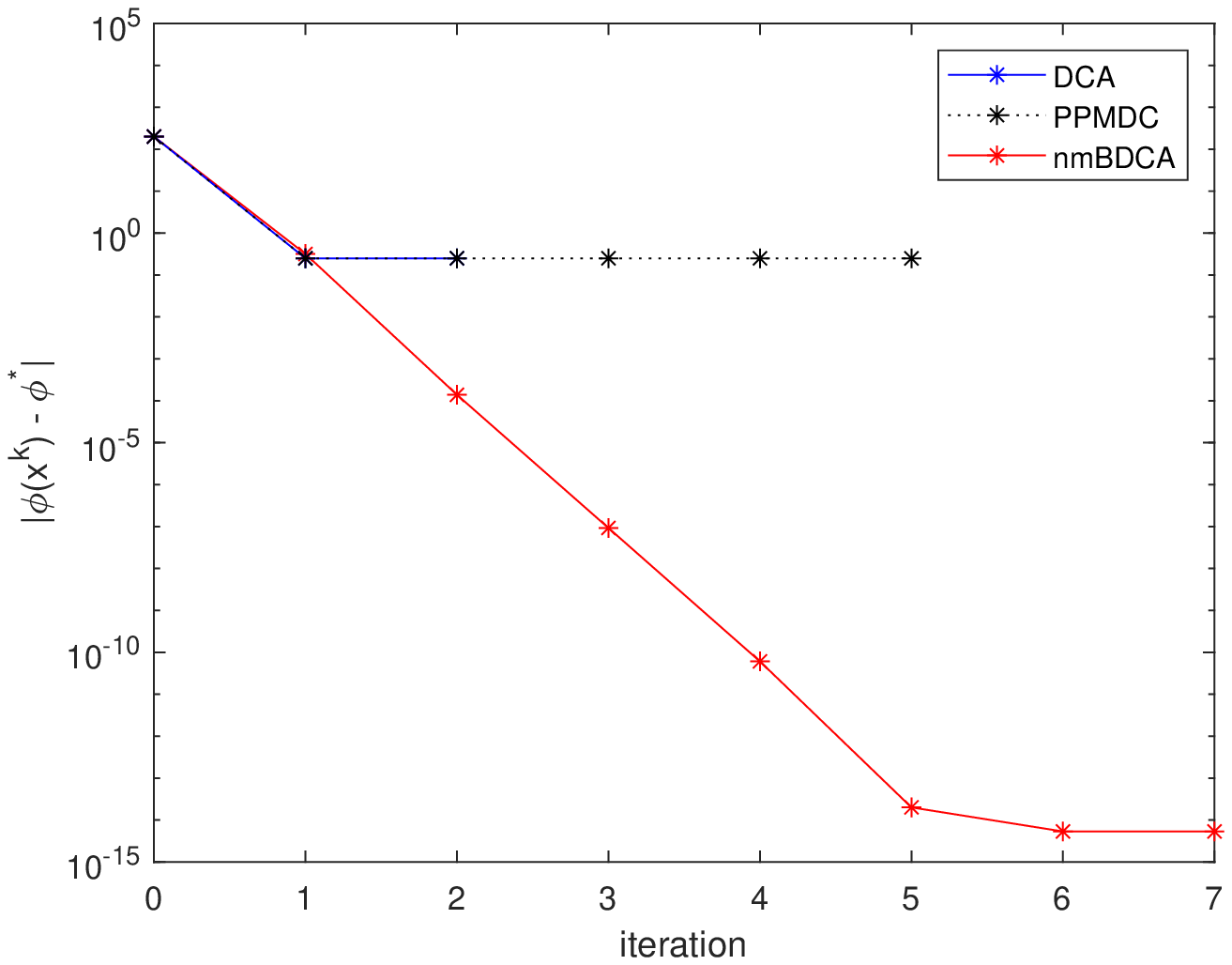}}
\subfloat{\label{fig6:b}\includegraphics[width=0.25\linewidth]{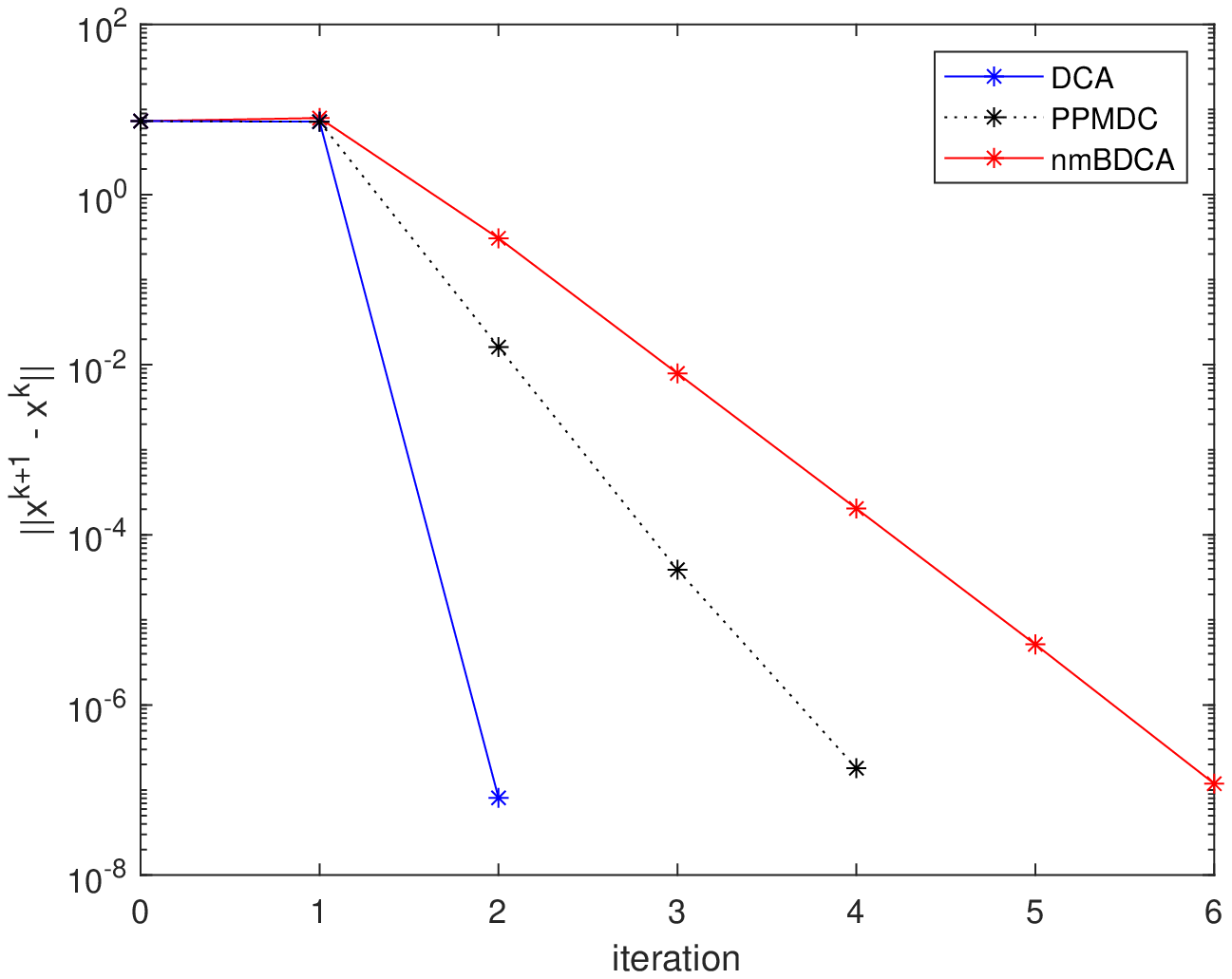}}
\subfloat{\label{fig6:c}\includegraphics[width=0.25\textwidth]{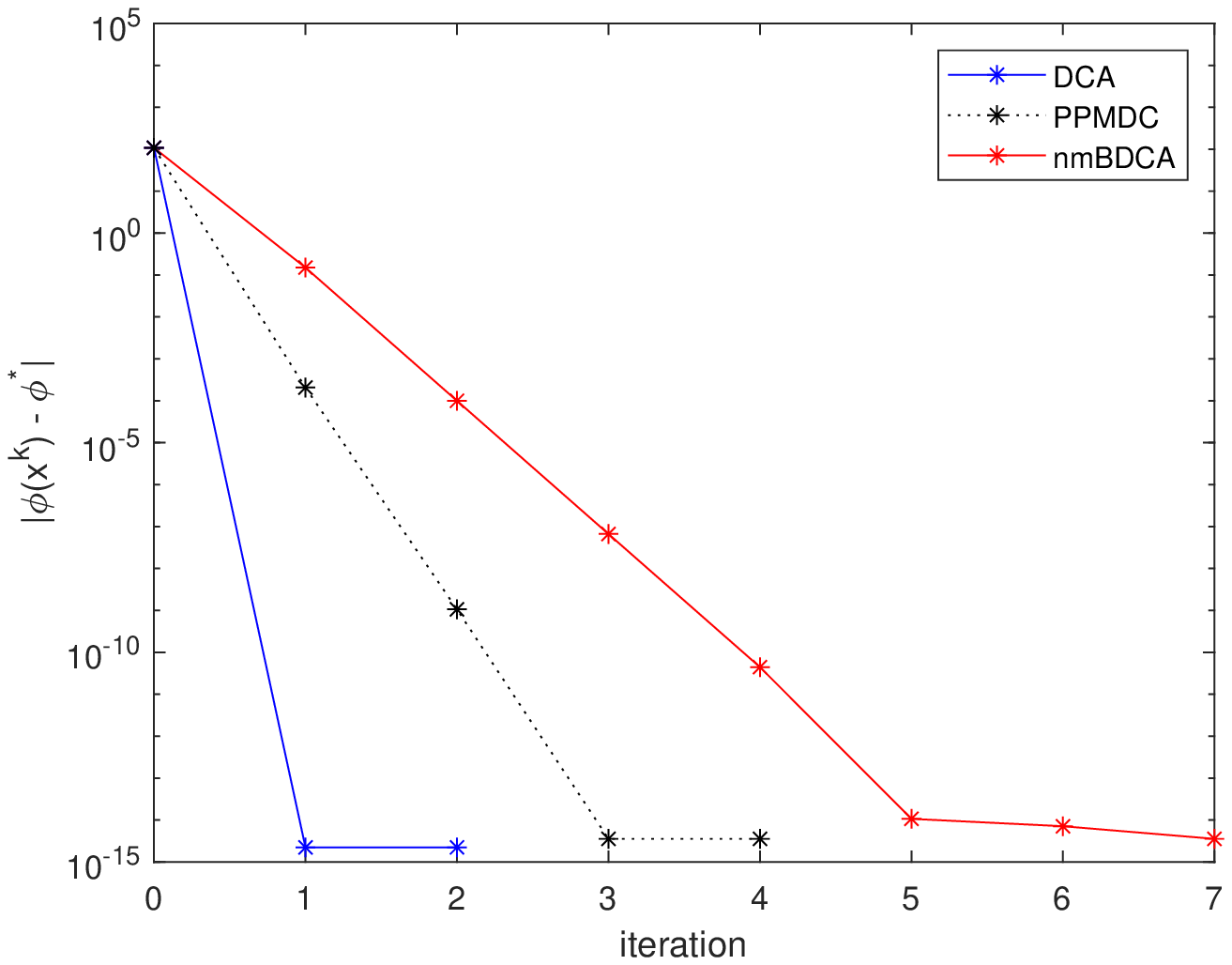}}
\subfloat{\label{fig6:d}\includegraphics[width=0.25\textwidth]{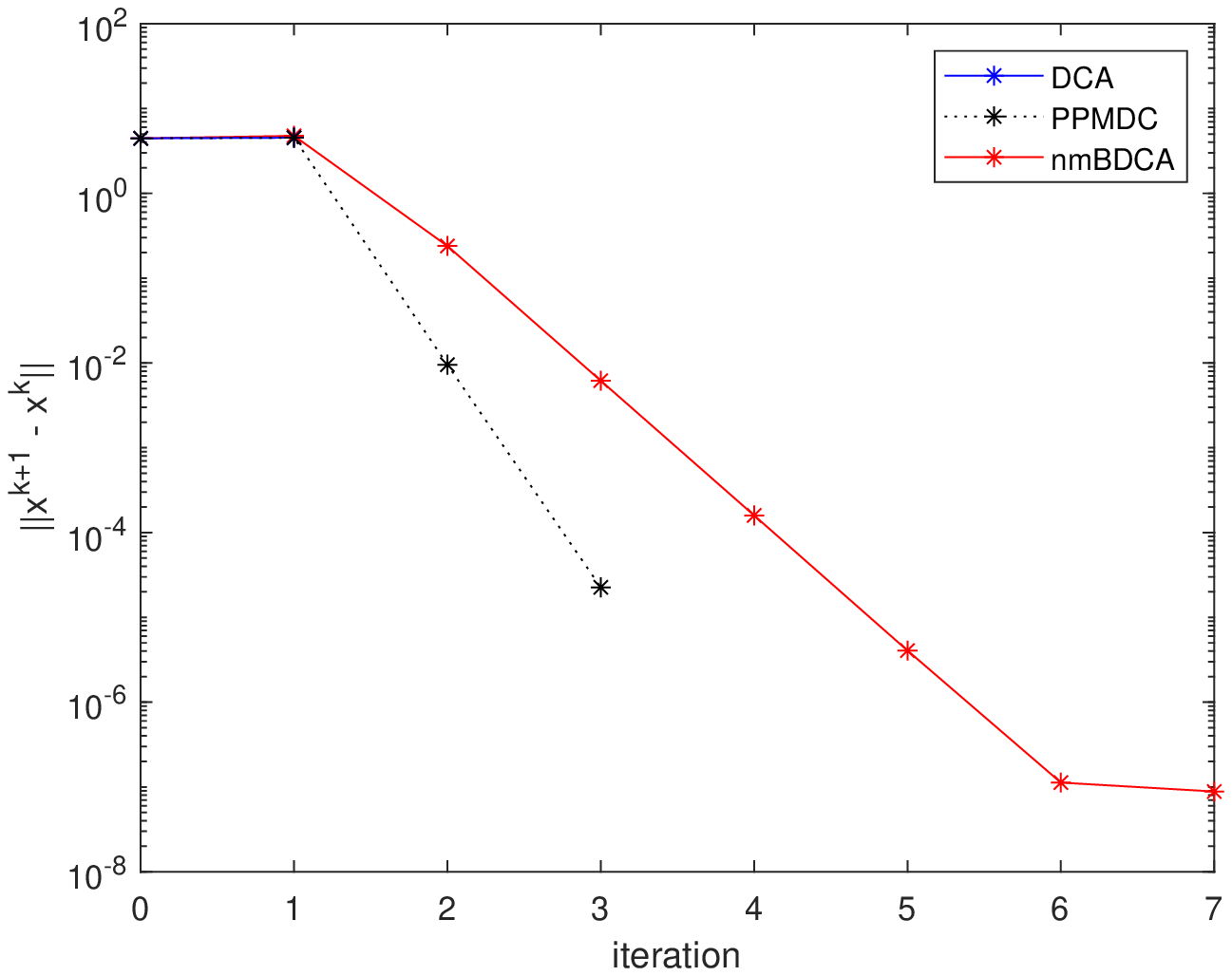}}%
\caption{Value of $||\phi(x^k)-\phi^*||$ and $||x^{k+1}-x^k||$ (using log. scale) for Problem~\ref{prob5}.}
\label{fig6}
\end{figure}

It is worth to mention that the freedom in the choice of the parameters $\lambda _{-1}$, $\rho$ and $\zeta$ in the line search of nmBDCA enables the possibility of speed up the method, specially as done in \cite{ARAGON2019} with the self-adaptive trial step size on the parameter $\lambda_k$. It is remarked in \cite{ARAGON2019} that this strategy allowed to obtain a two times speed up BDCA in their numerical experiments, when compared with the constant strategy. More precisely, other possibilities in the choice of the trial step size could improve the performance of nmBDCA.

Other important question is the computational influence in nmBDCA of the strongly convexity modulus $\sigma>0$ in the DC components. It does not explicitly appear neither in nmBDCA nor DCA (and PPMDC) but its influence can be seem for instance in Proposition~\ref{pr:ffr}~\ref{it:ffr:ineq} and Proposition~\ref{prop15u}~\ref{it:prop15u:decrease}. As mentioned in Remark~\ref{infinitydecomp}, given a DC function $\phi(x)$ with DC decomposition $\phi(x)=g(x) - h(x)$, we can add to both DC decomposition a strongly convex term $({\sigma}/{2})\| x \|^{2}$ to obtain a new DC representation $\phi(x)=(g(x)+\frac{\sigma}{2}\| x \|^{2}) - (h(x)+\frac{\sigma}{2}\| x \|^{2})$. It leads to the open problem whose answer is crucial to a deep understanding of the DC structure: Does exist an ``optimal" (in some sense) DC decomposition? This problem is intimately related to the notion of more convex and less convex (domination concept) introduced by Moreau; see \cite{Moreau}. This issue has been dealt for polynomial functions with in \cite{FerrerMartinezLegaz2009} and quadratic functions by \cite{BomzeLocatelli2004}.

In order to clarify the importance of this question, let us consider the following concept and a simple example.

\begin{definition}
We say that $\phi(x)=g(x) - h(x)$ is an undominated DC decomposition for $\phi$ if there is no other DC decomposition $\tilde{g}$ and $\tilde{h}$ for $\phi$ such that $g(x)=\tilde{g}(x)+p(x)$ and $h(x)=\tilde{h}(x)+p(x)$, for some non-constant convex function $p$.
\end{definition}

The key idea of algorithms for DC functions is to minimize convex bound functions instead of the possibly non-convex DC function. The following simple example shows that the interest for undominated DC decompositions lies in the fact that they allow us to deliver better bounds.

\begin{example}\cite[Example 4.85]{LocatelliSchoen2013}
Let $\phi(x)=x^3-x^2$ in $X=[0,1]$. Then, $$g_t(x)=x^3+tx^2\quad\mbox{and}\quad h_t(x)=(t+1)x^2,\quad t\geq 0,$$
define an infinite class of DC decomposition of $\phi$ over $X$, all dominated by the decomposition with $t=0$. Assume that we want to find the convex understimator of $\phi$ over $X$, i.e., $\psi_{t}(x)=g_t(x) - (t+1)x$ which is obtained by replacing the concave function $-h_t(x)$ by its convex envelope over $X$. Thus, the maximum distance between $\phi$ and its convex understimator $\psi_{t}$ is attained at $x=0.5$ and is equal to $$\max_{x\in X}\phi(x)-\psi_t(x)=\frac{1}{4}(1+t).$$
Therefore, the maximum distance is minimized for $t=0$, where $g_t$ and $h_t$ are undominated.
\end{example}

From a theoretical point of view $\sigma$ adds more structure to the DC representation. Nevertheless, from a computational point of view adding $\sigma$ may be a drawback. Next, we run nmBDCA, DCA and PPMDC in order to find a DC decomposition (related to $\sigma$) for Problem~\ref{prob7} and \ref{prob6} which needs less iterates and CPU time until the methods stop (in this sense it is more efficient). The results are presented in Figure~\ref{fig9} and \ref{fig10}. To this end, we consider the following DC components for Problem~\ref{prob7}
$$g(x)=\sin\left(\sqrt{|3x_1+|x_1-x_2| +2x_2|}\right)+\sigma(x_1^2+x_2^2) \quad \mbox{and}\quad  h(x)=\sigma(x_1^2+x_2^2),$$
and for Problem~\ref{prob6}
$$g(x)=-\frac{5}{2}x_1+|x_1|+|x_2|+\sigma(x_1^2+x_2^2) \quad \mbox{and}\quad  h(x)=(\sigma-0.5)(x_1^2+x_2^2).$$

\begin{figure}[tbp]
\centering
\subfloat[Iterations]{\label{fig10a}\includegraphics[width=0.5\linewidth]{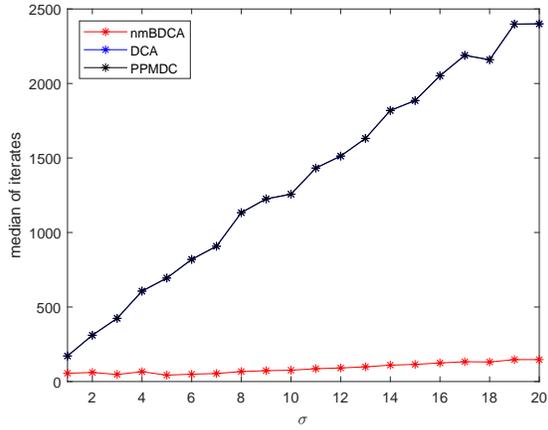}}
\subfloat[CPU time (in seconds)]{\label{fig10b}\includegraphics[width=0.5\linewidth]{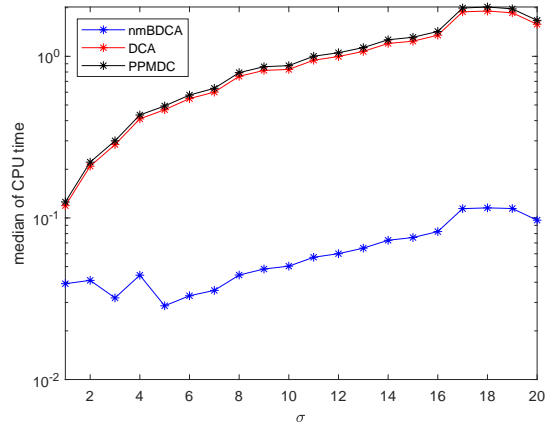}}
\caption{Median of 100 run for different values of $\sigma$ in Problem~\ref{prob7}.}
\label{fig10}
\end{figure}

\begin{figure}[tbp]
\centering
\subfloat[Iterations]{\label{fig9a}\includegraphics[width=0.5\linewidth]{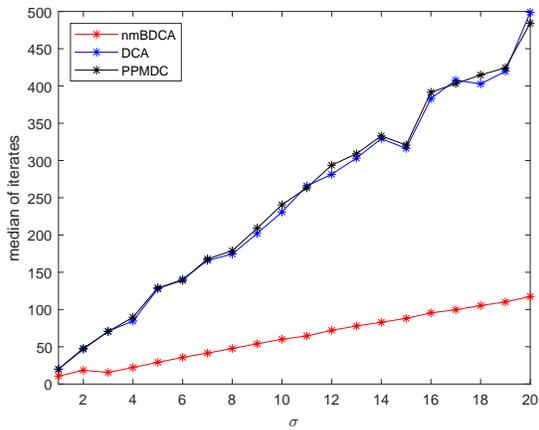}}
\subfloat[CPU time (in seconds)]{\label{fig9b}\includegraphics[width=0.5\linewidth]{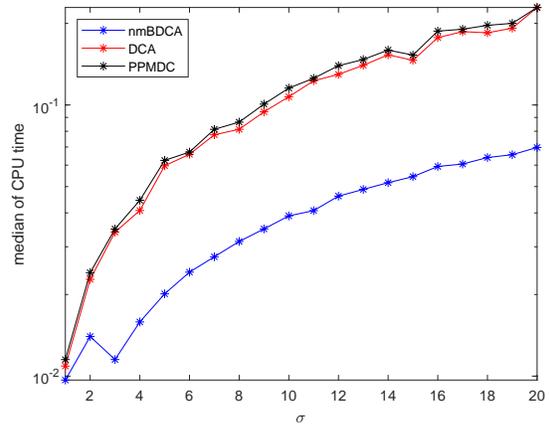}}
\caption{Median of 100 run for different values of $\sigma$ in Problem~\ref{prob6}.}
\label{fig9}
\end{figure}

In our numerical experiments, we consider positive integer values for $\sigma$ from $1$ to $20$. In Figure~\ref{fig10}, we can see that $\sigma=5$ provides the best performance for nmBDCA while $\sigma=1$ is the best choice for DCA and PPMDC in Problem~\ref{prob7}. In Problem~\ref{prob6}, all the methods have the best performance for $\sigma=1$ as we can see in Figure~\ref{fig9}. In both problems, Figures~\ref{fig10} and \ref{fig9} clearly show that the higher the value of $\sigma$, the worse the performance of the methods.

However, some questions still rise. Which is the best DC decomposition from a computational point of view and how can it be obtained? Could it be connected with some suitable theoretical concept? We refrain from discussing these questions to our general context of non-smooth DC functions because we understand it deserves to be deeply studied and maybe it cannot be completely answered unless it is considered for some specific cases. To illustrate how difficult are these questions we refer to \cite{BomzeLocatelli2004} where it is presented a quadratic problem which admits an infinite number of undominated DC decompositions. 

\section{Conclusions}
We have developed a non-monotone version of the boosted DC algorithm (BDCA) proposed in \cite{ARAGON2019} for DC programming when both DC components are not differentiable. Under mild conditions on the parameter that control the non-monotonicity of the objective function and standard assumptions on the DC function some convergence results and iteration-complexity bounds were obtained. In the case where the first DC component is differentiable, the global convergence and different iteration-complexity bounds were established assuming the Kurdyka-\L{}ojasiewicz property of the objective function. We have applied this non-monotone boosted DC algorithm (nmBDCA) for some academic tests. Our numerical experiments indicate that nmBDCA outperforms DC Algorithm (DCA~\cite{Pham1986}) and Proximal Point Method for DC functions (PPMDC~\cite{SunSampaio2003}) in both computational performance and quality of the solution found. 

A very interesting topic of future research is that the idea of using non-monotone line search to establish the well-definition of the nmBDCA can also be employed in other methods of non-differentiable convex optimization. For instance, let $f:\mathbb{R}^2\to \mathbb{R}$ be non-differentiable and convex function given by $f(x,y)=(x^2+y^2)/4+|x|+2 |y|$. The subdifferential of $f$ is given by 
\begin{align} \label{eq:sdex}
\partial f(x)=\begin{cases}
([-1,1], y/2+2 sgn(y)) , \qquad &\mbox{if } x=0, y\neq 0 ;\\
(x/2+sgn(x), [-2,2]) , \qquad &\mbox{if } x\neq 0, y=0 ;\\
([-1,1], [-2,2]) , \qquad &\mbox{if } x= 0, y=0 ;\\ 
(x/2+ sgn(x) , y/2+2 sgn(y)) , \qquad &\mbox{if } x\neq 0, y\neq 0.
\end{cases}
\end{align}
Take $x^0:=(4,4)$. Since $f$ is differentiable at $x^{0}$, we have $\partial f(x^{0})=\{(3,4)\}$ and setting $s^{0}=(3,4)\in \partial f(x^{0})$, we know that $-s^0$ is a descent direction of $f$ at $x^0$. Taking $\lambda _{-1}=1>0$, $\rho= \zeta = 1/2\in (0,1)$ we obtain 
\begin{align*}
5/4=f(x^0- \zeta ^{0}\lambda _{-1}s^0) \leq f(x^0)-\rho \zeta^{0}\lambda _{-1}\|s^0\|^{2}=30/4, 
\end{align*}
where $j^0=0$. Thus, $\lambda_{0}=1$ and setting $x^{1}=x^{0}-\lambda_{0}s^{0} =(1,0)$, we obtain that 
$$f(x^{1})\leq f(x^{0})-\lambda_{0}\|s^0\|^{2}.$$
This shows that starting with $x^0$ we can apply a monotone line search in order to find $x^1=(1,0)$. This was possible because $f$ is differentiable at $x^0=(4,4)$ and $-s^0=(-3,-4)$ is a descent direction of $f$ at $x^0$. On the other hand, since $f$ is not differentiable at $x^1=(1,0),$ given an arbitrary $s^1\in \partial f(x^1)$, the direction $-s^1$ may not be a descent direction, which means that we cannot apply monotone line search strategies in this case. Indeed, first note that \eqref{eq:sdex} gives $\partial f(x^1)=(3/2, [-2,2])$. Taking $s^1:=(3/2, -2)$ and $\lambda \in (0, 2/3)$, we have 
\begin{equation*} 
f(x^1-\lambda s^1)-f(x^1)= \frac{7}{4} \lambda+\frac{25}{16} \lambda^2.
\end{equation*} 
Hence, we conclude that $f'(x^{1}, -s^1)=7/4 >0$, which means that $-s^1=(-3/2, 2)$ is an ascent direction of $f$ at $x^1$. Thus, a monotone line search cannot be performed in this case. However, $ \lim_{\lambda \to 0^+} (f(x^1-\lambda s^1)-f(x^1)+\rho \lambda \| s^{1}\| ^{2})=0$. Thus, for any $\nu_1>0$, there exists $\delta >0$ such that $f(x^1-\lambda s^1)-f(x^1)+\rho \lambda \| s^1\| ^{2}<\nu_1$, for all $\lambda\in (0, \delta)$. Therefore, a non-monotone line search such as
\begin{equation*} 
j_k:=\min \left\{j\in {\mathbb N}: ~f( x^{k}-\zeta^{j} \lambda _{k-1}s^{k})\leq f(x^{k})-{\rho} \left(\zeta^{j}{\lambda_{k-1} }\right)\| s^{k}\| ^{2} +\nu _{k} \right\}. 
\end{equation*}
can be performed. This motivates us to define the following subgradient method with non-monotone line search to minimize a convex function $f:\mathbb{R}^n\to \mathbb{R}$ .

\begin{algorithm}[H]
\caption{SubGrad method with non-monotone line search}\label{Alg1s}
\begin{algorithmic}[1]
\STATE {Fix $\lambda _{-1}>0$, $0<\rho<1$ and $\zeta \in (0,1)$. Choose $x^0\in \mathbb{R}^{n}$. Set $k=0$.}
\STATE{Choose $s^{k}\in\partial f(x^{k})$. If $s^k=0$, then STOP and return $x^{k}$. Otherwise, take $\nu _{k}\in \mathbb{R}_{++}$ and set $\lambda _{k}:= \zeta ^{j_{k}}\lambda_{k-1},$ where
\begin{equation} \label{eq:jk}
j_k:=\min \left\{j\in {\mathbb N}: ~f( x^{k}-\zeta^{j} \lambda _{k-1}s^{k})\leq f(x^{k})-{\rho} \left(\zeta^{j}{\lambda_{k-1} }\right)\| s^{k}\| ^{2} +\nu _{k} \right\}. 
\end{equation}}
\STATE{ Set $x^{k+1}:=x^{k}-\lambda _{k}s^{k}$, and $k \leftarrow k+1$ and go to Step~2.}
\end{algorithmic}
\end{algorithm}
As we can see in the sequel, Algorithm~\ref{Alg1s} is well-defined. More precisely, if $\nu _{k}>0$, for all $k\in\mathbb{N}$, then Algorithm~\ref{Alg1s} is well defined. Indeed, since $f$ is a convex function, it is also continuous. Thus, we conclude that $\lim_{\lambda \to 0^+}(f(x^{k}-\lambda s^{k}) - f(x^{k})+\rho \lambda \|s^{k}\|^{2})=0$. Hence, due to 
$\nu _{k}>0$, there exists $\eta _{k}>0$ such that 
\begin{equation} \label{eq:ismwd}
f(x^{k}-\lambda s^{k}) - f(x^{k})+\rho \lambda \|s^{k}\|^{2}< \nu _{k}, \qquad \forall \lambda \in (0,{\eta _{k}}].
\end{equation} 
On the other hand, due to $\zeta \in (0,1)$ we have $\lim_{j\in {\mathbb N}} \zeta^{j}{\lambda_{k-1}}=0$. Hence, considering that $\eta_k>0$, there exists ${j_*}\in {\mathbb N}$ such that $\zeta^{j}{\lambda_{k-1}}\in (0, \eta_k]$, for all $j\geq {j_*}$. Therefore, \eqref{eq:ismwd} implies that there exists $j_k$ satisfying \eqref{eq:jk} and the claim is proved.

 It is worth to point out that $s^k$ in step 2 of Algorithm~\ref{Alg1s} is not in general a descent direction at 
$x^k$. However, it follows from \cite[Theorem 4.2.3]{Lemarechal1993} that the set where convex functions fail to be differentiable is of zero measure. Consequently, almost every $s^k\neq 0$ is a descent direction. Therefore, we expect that Algorithm~\ref{Alg1s} has a behavior similar to gradient method with non-monotone line search. We believe this is an issue that deserves to be investigated.

\section*{Acknowledgments}
We would like to thank the reviewers for their constructive remarks which allow us to improve the paper.

\bibliographystyle{plain}

\end{document}